\documentclass[11pt,reqno]{amsart}
\usepackage[margin=1in]{geometry}
\usepackage[utf8]{inputenc}
\usepackage[T1]{fontenc}

\usepackage{amssymb}
\usepackage{graphicx} 
\usepackage{mathrsfs}
\usepackage{enumerate}
\usepackage{xspace}
\usepackage{color}
\usepackage{enumerate}
\usepackage{enumitem}
\usepackage{amsthm}
\usepackage{dsfont}
\usepackage{bbm}
\usepackage[colorlinks=true,linkcolor=blue]{hyperref}
\usepackage[numbers]{natbib}
\usepackage[backgroundcolor=white,bordercolor=red]{todonotes}
\DeclareMathAlphabet{\mathpzc}{OT1}{pzc}{m}{it}
\usepackage[nameinlink]{cleveref}
\numberwithin{equation}{section}

\makeatletter
\def\@settitle{\begin{center}%
		\baselineskip14\p@\relax
		\bfseries
		\uppercasenonmath\@title
		\@title
		\ifx\@subtitle\@empty\else
		\\[1ex]\uppercasenonmath\@subtitle
		\footnotesize\mdseries\@subtitle
		\fi
	\end{center}%
}
\def\subtitle#1{\gdef\@subtitle{#1}}
\def\@subtitle{}

\makeatletter
\def\@settitle{\begin{center}%
		\baselineskip14\p@\relax
		\bfseries
		\uppercasenonmath\@title
		\@title
		\ifx\@subtitle\@empty\else
		\\[1ex]\uppercasenonmath\@subtitle
		\footnotesize\mdseries\@subtitle
		\fi
	\end{center}%
}
\def\subtitle#1{\gdef\@subtitle{#1}}
\def\@subtitle{}
\makeatother

\begin{document}

\theoremstyle{plain}

\newtheorem{theorem}{Theorem}[section]
\newtheorem{lemma}[theorem]{Lemma}
\newtheorem{proposition}[theorem]{Proposition}
\newtheorem{corollary}[theorem]{Corollary}
\newtheorem{Ass}[theorem]{Assumption}
\newtheorem{condition}[theorem]{Condition}
\newtheorem{definition}[theorem]{Definition}
\newtheorem*{lemmaOhne}{Lemma}

\theoremstyle{definition}

\newtheorem{example}[theorem]{Example}
\newtheorem{remark}[theorem]{Remark}
\newtheorem{SA}[theorem]{Assumption}
\newtheorem{discussion}[theorem]{Discussion}
\newtheorem{remarks}[theorem]{Remark}
\newtheorem*{notation}{Remark on Notation}
\newtheorem{application}[theorem]{Application}

\let\oldc\c

\newcommand{\of}{[\hspace{-0.06cm}[}
\newcommand{\gs}{]\hspace{-0.06cm}]}

\newcommand\llambda{{\mathchoice
		{\lambda\mkern-4.5mu{\raisebox{.4ex}{\scriptsize$\backslash$}}}
		{\lambda\mkern-4.83mu{\raisebox{.4ex}{\scriptsize$\backslash$}}}
		{\lambda\mkern-4.5mu{\raisebox{.2ex}{\footnotesize$\scriptscriptstyle\backslash$}}}
		{\lambda\mkern-5.0mu{\raisebox{.2ex}{\tiny$\scriptscriptstyle\backslash$}}}}}

\newcommand{\1}{\mathds{1}}

\newcommand{\F}{\mathbf{F}}
\newcommand{\G}{\mathbf{G}}

\newcommand{\B}{\mathbf{B}}

\newcommand{\M}{\mathcal{M}}

\newcommand{\la}{\langle}
\newcommand{\ra}{\rangle}

\newcommand{\lle}{\langle\hspace{-0.085cm}\langle}
\newcommand{\rre}{\rangle\hspace{-0.085cm}\rangle}
\newcommand{\blle}{\Big\langle\hspace{-0.155cm}\Big\langle}
\newcommand{\brre}{\Big\rangle\hspace{-0.155cm}\Big\rangle}

\newcommand{\X}{\mathsf{X}}

\newcommand{\tr}{\operatorname{tr}}
\newcommand{\N}{{\mathbb{N}}}
\newcommand{\cadlag}{c\`adl\`ag }
\newcommand{\on}{\operatorname}
\newcommand{\oP}{\overline{P}}
\newcommand{\oO}{\mathcal{O}}
\newcommand{\D}{\mathsf{D}} 
\newcommand{\bx}{\mathsf{x}}
\newcommand{\bb}{\hat{b}}
\newcommand{\bs}{\hat{\sigma}}
\newcommand{\bv}{\hat{v}}
\renewcommand{\v}{\mathfrak{m}}
\newcommand{\ob}{\bar{b}}
\newcommand{\oa}{\bar{a}}
\newcommand{\os}{\widehat{\sigma}}
\renewcommand{\j}{\varkappa}
\newcommand{\scl}{\ell}
\newcommand{\Y}{\mathscr{Y}}
\newcommand{\Z}{\mathscr{Z}}
\newcommand{\T}{\mathcal{T}}
\newcommand{\con}{\mathsf{c}}
\newcommand{\nk}{\hspace{-0.25cm}{{\phantom A}_k^n}}
\newcommand{\nl}{\hspace{-0.25cm}{{\phantom A}_1^n}}
\newcommand{\nm}{\hspace{-0.25cm}{{\phantom A}_2^n}}
\newcommand{\n}{\hspace{-0.35cm}{\phantom {Y_s}}^n}
\newcommand{\nme}{\hspace{-0.35cm}{\phantom {Y_s}}^{n - 1}}
\renewcommand{\o}{\hspace{-0.35cm}\phantom {Y_s}^0}
\newcommand{\e}{\hspace{-0.4cm}\phantom {U_s}^1}
\newcommand{\z}{\hspace{-0.4cm}\phantom {U_s}^2}
\newcommand{\iii}{|\hspace{-0.05cm}|\hspace{-0.05cm}|}
\newcommand{\co}{\overline{\on{co}}}
\renewcommand{\k}{\mathsf{k}}
\newcommand{\ovb}{\overline{b}}
\newcommand{\ova}{\overline{a}}
\newcommand{\s}{\mathfrak{s}}
\newcommand{\opsi}{\overline{\Psi}}
\newcommand{\ol}{\mathcal{L}}
\newcommand{\cW}{\mathscr{W}}
\newcommand{\cU}{\mathcal{U}}
\newcommand{\oD}{\overline{D}}
\newcommand{\ua}{\underline{a}}
\newcommand{\ou}{\overline{b}}
\newcommand{\uu}{\underline{b}}
\newcommand{\ovM}{\overline{\mathbb{M}}}
\newcommand{\ocK}{{\overline{\cK}}}
\newcommand{\ocR}{{\overline{\mathcal{R}}}}

\renewcommand{\epsilon}{\varepsilon}
\renewcommand{\rho}{\varrho}

\newcommand{\fPs}{\fP_{\textup{sem}}}
\newcommand{\fPas}{\mathfrak{S}_{\textup{ac}}}
\newcommand{\rrarrow}{\twoheadrightarrow}
\newcommand{\cA}{\mathcal{A}}
\newcommand{\ocA}{\mathcal{U}}
\newcommand{\cR}{\mathcal{R}}
\newcommand{\cK}{\mathcal{K}}
\newcommand{\cQ}{\mathcal{Q}}
\newcommand{\cF}{\mathcal{F}}
\newcommand{\cE}{\mathcal{E}}
\newcommand{\cC}{\mathcal{C}}
\newcommand{\cD}{\mathcal{D}}
\newcommand{\bC}{\mathbb{C}}
\newcommand{\cH}{\mathcal{H}}
\newcommand{\bth}{\overset{\leftarrow}\theta}
\renewcommand{\th}{\theta}
\newcommand{\cG}{\mathcal{G}}
\newcommand{\fPasn}{\mathfrak{S}^{\textup{ac}, n}_{\textup{sem}}}
\newcommand{\CLM}{\mathfrak{M}^\textup{ac}_\textup{loc}}
\newcommand{\Sd}{\mathcal{S}^\textup{sp}_{\textup{d}}}
\newcommand{\Sc}{\mathcal{S}}
\newcommand{\Sac}{\mathcal{S}_\textup{ac}}
\newcommand{\A}{\mathsf{A}}
\newcommand{\Td}{\mathsf{T}^\textup{d}}
\renewcommand{\t}{\mathfrak{t}}
\newcommand{\UC}{\hspace{-0.03cm}\textit{UC}}

\newcommand{\bR}{\mathbb{R}}
\newcommand{\nnabla}{\nabla}
\newcommand{\f}{\mathfrak{f}}
\newcommand{\g}{\mathfrak{g}}
\newcommand{\oconv}{\overline{\on{co}}\hspace{0.075cm}}
\renewcommand{\a}{\mathfrak{a}}
\renewcommand{\b}{\mathfrak{b}}
\renewcommand{\d}{\mathsf{d}}
\newcommand{\bS}{\mathbb{S}^d_+}
\newcommand{\p}{\mathsf{p}}
\newcommand{\dr}{r} 
\newcommand{\m}{\mathbb{M}} 
\newcommand{\Q}{Q}
\newcommand{\usc}{\textit{USC}}
\newcommand{\lsc}{\textit{LSC}}
\newcommand{\q}{\mathfrak{q}}
\renewcommand{\X}{\mathscr{X}}
\newcommand{\W}{\mathscr{W}}
\newcommand{\fP}{\mathcal{P}}
\newcommand{\w}{\mathsf{w}}
\newcommand{\oM}{\mathsf{M}}
\newcommand{\oZ}{\mathsf{Z}}
\newcommand{\oK}{\mathsf{K}}
\renewcommand{\Re}{\operatorname{Re}}
\newcommand{\cCk}{\mathsf{c}_k}
\newcommand{\C}{\mathsf{C}}
\newcommand{\cP}{\mathcal{P}}
\newcommand{\oPi}{\overline{\Pi}}
\newcommand{\cI}{\mathcal{I}}
\renewcommand{\P}{\mathbf{P}}
\renewcommand{\X}{\mathsf{X}}
\renewcommand{\Q}{\mathsf{Q}}
\renewcommand{\Y}{\mathsf{Y}}
\renewcommand{\Z}{\mathsf{Z}}
\newcommand{\E}{\mathbf{E}}
\renewcommand{\Q}{\mathbf{Q}}
\renewcommand{\X}{E} 
\renewcommand{\bS}{\mathbb{S}}
\newcommand{\USA}{\textit{USA}}
\renewcommand{\p}{\dot{\partial}}
\renewcommand{\d}{d}
\newcommand{\cV}{\mathcal{V}}
\newcommand{\ca}{\on{ca} \hspace{0.01cm}}
\renewcommand{\cR}{\mathsf{RM}}
\newcommand{\RPF}{\mathsf{RPF}}
\renewcommand{\c}{\alpha\hspace{0.02cm}} 
\newcommand{\cFs}{\cF_s}
\newcommand{\cFt}{\cF_t}
\newcommand{\cT}{\mathcal{T}}
\newcommand{\hcE}{\widehat{\cE}}
\newcommand{\hcU}{\widehat{\cU}}
\newcommand{\hc}{\widehat{\c}}
\newcommand{\hT}{\widehat{T}}

\renewcommand{\emptyset}{\varnothing}

\allowdisplaybreaks

\makeatletter

 \title[Representation Theorems for Convex Expectations and Semigroups]{Representation Theorems for Convex Expectations \\ and Semigroups on Path Space} 
\author[D. Criens]{David Criens}
\address{D. Criens - Albert-Ludwigs-University of Freiburg, Ernst-Zermelo-Str. 1, 79104 Freiburg, Germany.}
\email{david.criens@stochastik.uni-freiburg.de}

\author[M. Kupper]{Michael Kupper}
\address{M. Kupper - University of Konstanz, 78357 Konstanz, Germany.}
\email{kupper@uni-konstanz.de}



\thanks{We thank Daniel Bartl for helpful comments and discussions.}
\date{\today}

\maketitle

\begin{abstract}
The objective of this paper is to investigate the connection between penalty functions from stochastic optimal control, convex semigroups from analysis and convex expectations from probability theory. Our main result provides a one-to-one relation between these objects. As an application, we use the representation via penality functions and duality arguments to show that convex expectations are determined by their finite dimensional distributions. To illustrate this structural result, we show that Hu and Peng's axiomatic description of \(G\)-L\'evy processes in terms of finite dimensional distributions extends uniquely to the control approach introduced by Neufeld and Nutz. Finally, we show that convex expectations with a Markovian structure are fully determined by their one-dimensional distributions, which give rise to a classical semigroup on the state space. As an application of this result, we establish a Laplace principle for entropic risk measures associated to controlled diffusions.

\smallskip\noindent
 \emph{Key words:} Convex expectation, semigroup, penalty function, path space, relaxed control problem 
   
 		\smallskip
		\noindent \emph{AMS 2020 Subject Classification:} Primary 46N10; 60G65; 93E20 Secondary 47D07; 49J53
\end{abstract}

\section{Introduction}
It is well-known that the unique solution to the Hamilton--Jacobi--Bellman PDE
\begin{align} \label{eq: intro HJB}
\partial_t u = \max_{a \in [a_*, a^*]} \frac{a}{2}\, \partial^2_x u, \quad u (0, \cdot \,) = g, 
\end{align} 
has a canonical control representation as
\begin{align*} \label{eq: intro control} 
v (t, x) = \sup_{\alpha} E \big[ g (X^{x,\alpha}_{t}) \big], 
\end{align*} 
where the supremum is taken over all \([a_*, a^*]\)-valued controls \(\alpha\), and the controlled process \(X^{x, \alpha}\) follows the dynamics
\[
d X^{x, \alpha}_t = \sqrt{\alpha_t} \, d W_t, \quad X_0^{x, \alpha} = x, \quad W = \text{Brownian motion}. 
\]
The PDE problem can also be understood in terms of the Nisio semigroup (cf.~\cite{denk2020semigroup,nisio76}), which is given by
\[
S_t (g) (x) = \sup_{\alpha} E \big[ g (X^{x,\alpha}_{t}) \big]. 
\]
This is a sublinear semigroup on the space \(C_b (\bR; \bR)\) of bounded continuous functions whose generator equation is given by \eqref{eq: intro HJB}. A third perspective is given by Peng's \(G\)-expectation (cf.~\cite{peng07}), a sublinear expectation whose one-dimensional distributions are characterized by the PDE \eqref{eq: intro HJB}. Summarizing these observations, the PDE \eqref{eq: intro HJB} establishes connections between control theory, analysis, and probability theory. The objective of this paper is to investigate the relationships among these concepts from a bird perspective within a broad framework on the path space \(\Omega\), consisting of \cadlag functions from \(\bR\) into a Polish space~\(E\).  Notably, we consider a two-sided time index set, which is essential for ensuring the semigroup property in our path-dependent setting. This is analogous to the evolutionary semigroups in \cite{DKK24}, which arise as the composition of the shift semigroup and a linear expectation operator, thereby formalizing the concept of linear transition semigroups within a path-dependent framework.

\smallskip 
Let us explain our \emph{main contributions} in more detail. Our first main result shows that every convex expectation \(\cE = ((t, \omega) \mapsto \cE_t (\, \cdot \,) (\omega))\) has a representation of the form 
\begin{align}\label{eq: ex penality formula intro}
\cE_t (\varphi) (\omega) = \sup_P \Big( E^P \big[ \varphi \big] - \alpha (t, \omega, P) \Big)
\end{align} 
for a penalty function \(\alpha\), and that conversely, for any penalty function \(\alpha\), the formula~\eqref{eq: ex penality formula intro} defines a convex expectation. 
Here, \(\varphi\) is a real-valued upper semianalytic function and the supremum is taken over all probability measures on the path space.
This representation corresponds to a dynamic version of Choquet's capacitability theorem and relies strongly on its conditional version from \cite{bartl_20}. We emphasize that our definition of a convex expectation includes the time consistency property and continuity from above and below. In the sublinear case, the penalty function is either zero or \( + \infty \), and the representation \eqref{eq: ex penality formula intro} reduces to
\begin{align} \label{eq: sublinear intro}
\cE_t (\varphi) (\omega) = \sup_{P \in \cU (t, \omega)} E^P \big[ \varphi \big]
\end{align} 
for a so-called (set-valued) representation map \((t, \omega) \mapsto \cU (t, \omega)\).
 
In our second main result, we establish a connection between convex expectations \(\cE\) to convex semigroups \((S_t)_{t \geq 0}\) that are both assumed to be homogeneous in the sense that 
\[
\cE_t (\varphi) = \cE_0 (\varphi \circ \theta_{-t}) \circ \theta_t, \qquad S_t (\varphi) = S_0 (\varphi \circ \theta_t),
\]
where \(\theta_t (\omega) = \omega (\, \cdot + t)\) denotes the usual shift operator. To define homogeneity, the two-sided time interval is crucial. 
For the convex expectation, homogeneity can be viewed as a two-sided time-homogeneous Markov property. To understand this, think of \(\cE_t\) as conditional expectation given the past up to time \(t\), and notice that homogeneity entails
\[
\cE_t (\varphi \circ \theta_t) = \cE_0 (\varphi) \circ \theta_t, 
\]
which corresponds to the classical homogeneous Markov property. 
In particular, by enlarging the state space from $E$ to $\Omega$, this extends the concept of Markov processes to the path-dependent framework, as discussed in \cite{DKK24} within a linear setting.
The homogeneity of the semigroup \((S_t)_{t \geq 0}\) is best understood by recalling the classical representation of a linear Feller semigroup \((T_t)_{t \geq 0}\) as  
\(T_t (f) = E [ f (Y_t) ]\),
where \((Y_t)_{t \geq 0}\) is a Feller process; see, e.g., \cite{kallenberg}.
As a main result, we show that, for a homogeneous convex expectation \(\cE\), the formula
\[
S_t (\varphi) := \cE_0 (\varphi \circ \theta_t) 
\]
defines a homogeneous convex semigroup and conversely, for a homogeneous convex semigroup \((\widehat{S}_t)_{t \geq 0}\), the formula
\[
\widehat{\cE}_t (\varphi) := \widehat{S}_0 (\varphi \circ \theta_{-t}) \circ \theta_t
\]
defines a homogeneous convex expectation. This resembles the classical linear result for Feller semigroups.

\smallskip
Our main results have the following \emph{implications}. First, using the formula \eqref{eq: ex penality formula intro}, we can derive a control representation for convex expectations. In particular, since the penalty function requires convexity and compactness assumptions, it is natural to consider relaxed control rules, as studied in \cite{EKNJ88,ElKa15,K90}. This idea is further explored in Subsection~\ref{sec: example relaxed control}, where we show that a broad class of convex expectations naturally arises within a classical relaxed control framework for controlled stochastic differential equations.

Second, relying on the representation \eqref{eq: ex penality formula intro} and convex duality arguments, 
we show that convex expectations are determined by their finite dimensional distributions.
This means that two convex expectations \(\cE\) and \(\widehat{\cE}\) coincide (on the space of all bounded upper semianalytic functions from \(\Omega\) into \(\bR\)) whenever 
\[
\cE_t \big( f (X_{t_1}, \dots X_{t_n}) \big) = \widehat{\cE}_t \big( f (X_{t_1}, \dots X_{t_n}) \big)
\]
for all \(t \leq t_1 < \dots < t_n < \infty\), $n\in\mathbb{N}$ and \(f \in C_b (E^n; \bR)\), where \(X_t (\omega) = \omega (t)\) denotes the coordinate process. 
This resembles the linear result that laws of \cadlag processes are determined by their finite dimensional distributions, being a central tool for the analysis of stochastic processes.
As an application, we connect different approaches for nonlinear L\'evy processes, namely those from \cite{HP_21}, that characterized the finite dimensional distributions, and those from \cite{neufeld2017nonlinear}, which considered worst case expectations of the form \eqref{eq: sublinear intro} on the space of upper semianalytic functions.

Third, we establish a general comparison result for homogeneous Markovian convex expectations in terms of their one-dimensional distributions. To be more precise, for two homogeneous Markovian convex expectations \(\cE\) and \(\widehat{\cE}\), we prove that 
\begin{align*}
\cE_0 (f (X_t)) \leq \widehat{\cE}_0 (f (X_t)) \ &\text{for all } t \in \bR, \, f \in C_b (E; \bR) \quad
\Longleftrightarrow \quad \cE_t \leq \widehat{\cE}_t \ \text{for all } t \in \bR.
\end{align*} 
As an application of this result, in the Markovian case, convex expectations are fully determined by the associated Markovian transition semigroups and their generators.

Let us now comment on \emph{related literature}. An important inspiration for our main result is the article \cite{bartl_20} by Bartl, which establishes the connection between convex expectations and penalty functions in a static conditional setting (an earlier version also included results in finite discrete time).

The systematic study of sublinear expectations on path spaces started with the seminal work \cite{peng07} of Peng on the \(G\)-expectation; see also \cite{P2019} for an exhaustive overview. More recently, extensions to settings with path-dependent domains of uncertainty and related measure theoretic approaches came into the focus of interest. 
We highlight \cite{nutz_13} for introducing random \(G\)-expectation, \cite{neufeld2014measurability,NVH} for measure-theoretic foundations, and \cite{hol16} for the first general treatment of the semigroup connection.
Studies on more general stochastic processes under uncertainty, including their PDE, semigroup, and Feller properties, can be found in \cite{C_24_JMAA,CN_23_EJP,CN_23_arxiv_jumps,CN_23_arxiv_viscosity, CN_24_SPA}.
An analytic approach based on Nisio and Chernoff-type approximations to stochastic processes under uncertainty is given in \cite{BK2023, denk2020semigroup, NR2021}; see also \cite{BDKN2025} for a general treatment of convex monotone semigroups.

The connection between Peng's \(G\)-expectation and stochastic control problems was first explored in~\cite{denis_hu_peng_11} through the PDE \eqref{eq: intro HJB}. In more general settings, related results follow from uniqueness results for viscosity solutions to Hamilton–Jacobi–Bellman PDEs. For instance, \cite{CN_23_EJP} identifies continuous semimartingales under uncertainty as the unique viscosity solutions to path-dependent extensions of \eqref{eq: intro HJB}, as studied in~\cite{cosso_russo_22, zhou_23}. Recently, without relying on PDE techniques, \cite{CN_23_arxiv_jumps} established a connection between the measure-theoretic approach to semimartingales under uncertainty and weak and relaxed control formulations, as introduced in the seminal works \cite{EKNJ88, ElKa15, K90}.

Stochastic control problems are intrinsically linked to backward SDEs (BSDEs). In \cite{bismut_72}, Bismut observed that the adjoint variable of a stochastic control problem satisfies a BSDE, see also \cite{benesoussan_82} for further details. The success of BSDEs started with the groundbreaking work of Pardoux and Peng \cite{padoux_peng_90} on nonlinear BSDEs. Since then, interest in BSDEs has grown significantly due to their applications in control theory and mathematical finance (cf.~\cite{EPQ1997,pardoux_R_14}). This has led to several generalizations, including 2BSDEs (cf.~\cite{CSTV2007, KPZ2015, STZ2012}), which provide stochastic representations for fully nonlinear PDEs. Moreover, in Brownian filtrations, BSDEs allow for the representation of \(g\)-expectations (cf.~\cite{CHMP2002}) and certain convex expectations in terms of their penalty functions (cf.~\cite{DPG_10}).

\smallskip
The paper is structured as follows. Section~\ref{sec: main1} introduces the main concepts and our key results. Section~\ref{sec: sublinear} specializes these results to the sublinear case, and Section~\ref{sec: 1D extensions} establishes a comparison result for Markovian convex expectations. Section~\ref{sec: application} illustrates the results by applying them to a Laplace principle for controlled diffusion processes, while Section~\ref{sec:related concepts} relates our findings to the existing literature. The paper concludes with an appendix.

 \section{Convex expectations, representation penalty functions and semigroups} \label{sec: main1}

In this section, we introduce the basic concepts and establish a one-to-one correspondence between homogeneous convex expectations, their representation penalty functions, and associated semigroups. These are all defined on path spaces of \cadlag and continuous functions on the two-sided time index set \(\mathbb{R}\). 
Furthermore, we show that convex expectations are uniquely determined by their finite-dimensional distributions. To illustrate these concepts, we conclude with an example that provides a typical class of convex expectations related to stochastic control problems in relaxed form.

\subsection{The path space}
Let \(\X\) be a Polish space and define the \emph{two-sided path space} \(\Omega\)  
either as the space  \(D ( \bR; \X )\) of all \cadlag functions, or as the space \(C (\bR; E)\) of all continuous functions from $\bR$ to $E$.
The coordinate process \(X\)  on \(\Omega\) is given by \(X_t (\omega) = \omega (t)\) for all \(\omega \in \Omega\) and \(t \in \bR\).

 We set \(\cF := \sigma (X_t,\, t \in \bR)\) and \(\cF_t := \sigma (X_s,\, s \in (- \infty, t])\) for \(t \in \bR\), with  the convention \(\cF_\infty := \cF\).\footnote{Note that \(\cF_t = \sigma (X_s, s \in (\mathbb{Q} \cap (- \infty, t)) \cup \{t\})\), which shows that \(\cF_t\) is countably generated.} The space \(\Omega\) is endowed with the two-sided Skorokhod \(J_1\) topology, which coincides with the local uniform topology if  \(\Omega = C (\bR; E)\). 
It is well known that \(\cF\) coincides with the Borel $\sigma$-field \(\mathcal{B} (\Omega)\); see \cite[Chapter~12]{Whitt} and \cite[Appendix~B]{DKK24} for more details on this topology. 

 For a \(\sigma\)-field \(\mathcal{G}\), a set \(A\) is called \emph{\(\mathcal{G}\)-analytic} if there is a Souslin scheme \(\{A_{n_1, \dots, n_k} \} \subset \mathcal{G}\) such that 
\[
A = \bigcup_{(n_i) \in \mathbb{N}^\infty} \bigcap_{k \in\N } A_{n_1, \dots, n_k}.
\]
In case \(\mathcal{G} = \cF\), we simply speak of analytic sets.
Analytic sets are sometimes also called Souslin sets, e.g., in the monograph \cite{bogachev}. We refer to \cite[Chapter~1]{bogachev} for an exhaustive overview on properties of analytic sets. 

A function \(\varphi \colon \Omega \to \bR\) is called \emph{upper (resp., lower) \(\mathcal{G}\)-semianalytic} if the sets \(\{ \varphi \geq c\}\) (resp., \(\{\varphi \leq c\}\)) are \(\mathcal{G}\)-analytic for all \(c \in \bR\). In case \(\mathcal{G} = \cF\), we simply speak of upper (resp., lower) semianalytic functions. The space of all bounded upper \(\mathcal{G}\)-semianalytic functions is denoted by \(\USA_b (\Omega, \mathcal{G}; \bR)\), and we set \(\USA_b (\Omega; \bR) := \USA_b (\Omega, \cF; \bR)\). For an overview of basic results on semianalytic functions; see \cite[Section~7.7]{bershre}.

Let \(C_b(\Omega, \mathcal{G}; \bR)\) be the space of all bounded continuous \(\mathcal{G}\)-measurable functions on \(\Omega\), and set \(C_b (\Omega; \bR) := C_b (\Omega, \cF; \bR)\). We denote by \(\mathfrak{P}(\Omega)\) the set of all Borel probability measures on \(\Omega\), endowed with the weak topology \(\sigma(\mathfrak{P}(\Omega), C_b(\Omega; \bR))\), which is separable and completely metrizable.

We recall that a \emph{stochastic kernel} is a map
 $Q\colon \Omega\times \cF\to [0,1]$ such that $Q_\omega\in  \mathfrak{P} (\Omega)$ for all $\omega\in\Omega$, and $\omega\mapsto Q_\omega(B)$ is $\mathcal{F}$-measurable for all $B\in \cF$. It is called $\cF_t$-measurable if $\omega\mapsto Q_\omega(B)$ is $\cF_t$-measurable for all $B\in \cF$. For \(P \in \mathfrak{P} (\Omega)\) and $t\in\bR$, we denote by
\(P (\, \cdot \mid \cFt)\) the \emph{regular version of the conditional probability}, i.e., an $\cF_t$-measurable stochastic kernel $P (\, \cdot \mid \cFt)(\cdot)\colon \Omega\times\cF\to[0,1]$ satisfying $\int_B P(A\mid \cF_t)\,dP=P(A\cap B)$ for all $A\in \cF$ and $B\in \cF_t$. 

We conclude this subsection with additional notations and definitions that will be used frequently throughout the article.
For \(\omega \in \Omega\) and \(t \in \bR\), we define
\[
[\omega]_t := \{ \omega' \in \Omega \colon \omega' = \omega \text{ on } [0, t] \}.
\] 
For \(\omega, \omega' \in \Omega\) and \(t \in \bR\), we define the pasting 
\[
\omega \otimes_t \omega' := \begin{cases} \omega, & \text{on } (- \infty, t), \\ 
	\omega' - \omega' (t) + \omega (t), 
	& \text{on } [t, \infty). \end{cases} 
\]
Given \(P \in \mathfrak{P} (\Omega)\), a stochastic kernel \(Q\) and \(t \in \bR\), we define the pasting measure 
\[
(P  \otimes_t  Q) (A) := \iint \1_A (\omega  \otimes_{t}  \omega')\, Q_\omega (d \omega') P (d \omega), \quad A \in \cF. 
\]
We mainly consider kernels which satisfy \( Q_\omega([\omega]_t) = 1 \). In this case, the pasting measure has the simpler representation
\[
(P \otimes_t  Q) (A) = \int Q_\omega (A) \, P (d \omega) = E^P \big[ Q_\cdot (A) \big], \quad A \in \cF.
\]
For \(t \in \bR\), let \(\theta_t \colon \Omega \to \Omega\) be the usual shift given by \(\theta_t (\omega) := \omega (\, \cdot  + t)\). Observe that in our two-sided setting, we have 
\(
\theta_t \circ \theta_{-t} = \textup{id}.
\) 
Finally, we set 
\(P_t := P \circ \theta^{-1}_t
\)
for all \(P \in \mathfrak{P}(\Omega)\) and \(t \in \bR\).

\subsection{Convex expectations and representation penalty functions} 
Random outcomes are commonly modeled using probabilities and expectations. 
However, in many situations,  these probabilities are only partially known (aleatoric uncertainty), resulting in expectations that are imprecise and lie in intervals defined by lower and upper expectations. In this article, we focus on upper expectations $\mathcal{E}$, as they typically translate to lower expectation through $\underline{\mathcal{E}}(\varphi):=-\mathcal{E}(-\varphi)$. An upper expectation is generally not linear and is therefore often referred to as a \emph{nonlinear expectation}. More specifically, a \emph{convex expectation} is a map $\mathcal{E}\colon\mathcal{H}\to\bR$, defined on some pointwise ordered convex space $\mathcal{H}$ of functions including constants, which satisfies the following properties:
\begin{itemize}
    \item $\mathcal{E}(\varphi)=c$ for all constant functions $\varphi\equiv c$.
    \item $\mathcal{E}(\varphi)\le\mathcal{E}(\psi)$ for all $\varphi,\psi\in\mathcal{H}$ with $\varphi\le\psi$.
    \item $\mathcal{E}(\lambda\varphi+(1-\lambda)\psi)\le\lambda\mathcal{E}(\varphi)+(1-\lambda)\mathcal{E}(\psi)$ for all $\varphi,\psi\in\mathcal{H}$ and $\lambda\in(0,1)$.
\end{itemize}
Up to sign conventions, convex expectations are also called convex risk measures; see~\cite{FS2016}. If they are additionally positively homogeneous, they are referred to as sublinear expectations in stochastic calculus under uncertainty \cite{P2019}, upper coherent previsions in the theory of imprecise probabilities \cite{W1991}, or upper expectations in robust statistics \cite{H1981}. Depending on $\mathcal{H}$ and under additional continuity conditions, 
a convex expectation admits a dual representation \[\mathcal{E}(\varphi)=\sup_{P \in \mathfrak{P}(\Omega)} \Big(E^P[\varphi]-\alpha(P)\Big),\] which establishes the connection to classical probabilities 
$P$ and their expectations $E^P$. For $\mathcal{H}=\USA_b (\Omega; \bR)$, 
such a representation\footnote{Under continuity conditions related to \ref{NE5} below.} follows from Choquet’s capacitability theorem; see~\cite{bartl_20, BCK_19}. 

\smallskip
For conditional expectations $\cE\colon \USA_b (\Omega; \bR)\to \USA_b (\Omega,\cG; \bR)$, a similar representation result holds, as shown in the remarkable work \cite{bartl_20} by Bartl, which builds the foundation for the next result. In contrast to the linear case, where the tower property for conditional expectations always holds, this is not true for convex expectations, and additional conditions are necessary. 

\begin{definition} \label{def: sublinear expectation}
	A \emph{convex expectation} on $\USA_b (\Omega; \bR)$ is a map
	\[\cE \colon \USA_b (\Omega; \bR) \to \USA_b (\bR \times \Omega; \bR),\] 
written \((\cE (\varphi)) (t, \omega) \equiv \cE_t (\varphi) (\omega)\), which satisfies the following properties:
 	\begin{enumerate} [label=\textup{(E\arabic*)},ref=\textup{(E\arabic*)}]
  \item \label{NE0}  \(\cE_t (\varphi) \in \USA_b (\Omega, \cF_t; \bR)\) for all \(t \in \bR\) and \(\varphi \in \USA_b (\Omega; \bR)\).
 	\item \label{NE2}
		\(\cE_t (\varphi)  = \varphi \) for all \(t \in \bR\) and \(\varphi \in \USA_b (\Omega,\cF_t; \bR)\).
		\item \label{NE1}
		\(\cE_t (\varphi)  \leq \cE_t (\psi) \) for all \(t \in \bR\) and \(\varphi, \psi \in \USA_b (\Omega; \bR)\) with \(\varphi \leq \psi\).
		
			\item \label{NE3}
		\(\cE_t (\lambda \varphi + (1 - \lambda) \psi) \leq \lambda \cE_t (\varphi) + (1- \lambda) \cE_t (\psi)\) for all \(t \in \bR\), \(\varphi, \psi \in \USA_b (\Omega; \bR)\) and \(\lambda \in (0, 1)\).
			\item \label{NE4}
		The map \(\omega \mapsto \cE_t (\varphi) (\omega)\) is \(\cFt\)-measurable for all \(\varphi \in C_b (\Omega; \bR)\).
			
			\item \label{NE6}
  \(\mathcal{E}_{t} ( \mathcal{E}_s (\varphi) )  = \mathcal{E}_t (\varphi) \) for all \(s > t\) and all \(\varphi \in \USA_b (\Omega; \bR)\).
  \item \label{NE5}
   For every $t\in \bR$, the map \(\cE_t\) is continuous from above on $C_b (\Omega; \bR)$,\footnote{I.e., \(\cE_t (\varphi_n)   \searrow \cE_t (\varphi) \) for all $(\varphi_n)_{n\in\N}\subset C_b (\Omega; \bR)$ with \(\varphi_n \searrow \varphi \in \USA_b (\Omega; \bR)\).} and continuous from below.\footnote{I.e., \(\cE_t (\varphi_n)  \nearrow \cE_t (\varphi) \) for all $(\varphi_n)_{n\in\N}\subset \USA_b (\Omega; \bR)$ with \(\varphi_n \nearrow \varphi \in \USA_b (\Omega; \bR)\).}
   \end{enumerate}

A convex expectation \(\cE\) is said to have \emph{no fixed times of discontinuity} if, for every \(t \in \bR\) and \(\omega \in \Omega\),
\[
\cE_t ( a \1_{\{ \Delta X_s\, \not =\, 0 \}} ) (\omega) = 0 \quad\text{for all } s > t \text{ and } a > 0.
\] 
Here, \(\Delta X_t := X_t - X_{t-}\) denotes the jump of the process \(X\) at time \(t\).

A convex expectation \(\cE\) is called \emph{homogeneous} if 
	\begin{align} \label{eq: cE homog}
	\cE_t (\varphi) = \cE_0 (\varphi \circ \theta_{-t}) \circ \theta_{t}\quad \text{for all } t \in \bR \text{ and }\varphi \in \USA_b (\Omega; \bR).
	\end{align}
\end{definition}

\begin{discussion}\label{dicussion:expectation}
\begin{enumerate}
\item[(a)] The properties \ref{NE0}--\ref{NE3} are conditional analogs to the corresponding static properties of convex expectations. Condition \ref{NE6} is the tower property, while \ref{NE5} is a technical continuity condition that ensures the dual representation. In our dynamic setting, in addition to \cite{bartl_20}, we impose that convex expectations are jointly upper semianalytic in $\mathbb{R} \times \Omega$ (e.g., allowing for their evaluation at stopping times).
\smallskip
\item[(b)] Condition \ref{NE4} seems somewhat unnatural at first glance. However, it is purely technical and guarantees the dual version of the tower property, also known as the cocycling property. As shown in Appendix \ref{app:tower property}, in the homogeneous case without fixed times of discontinuity, it can be replaced by the following version which gives the condition a finite dimensional flavor:\smallskip
\begin{itemize}
\item[{\rm (E5'})] \label{NE4'}
		The map \(\omega \mapsto \cE_t (\varphi) (\omega)\) is \(\cFt\)-measurable for all \(\varphi \in D_t\),  where \[D_{t}:= \Big\{ \, g (X_{t_1}, \dots, X_{t_n}) \colon g \in C_b (E^n; \bR), \, t_1, \dots, t_n \in \mathbb{Q} \cap (t,\infty), \, n \in \mathbb{N}\, \Big\}.\]
 \end{itemize} 
\smallskip
\item[(c)] As is evident by definition, a homogeneous convex expectation \(\cE\) has no fixed times of discontinuity if and only if
	\[
	\cE_0 (a \1_{\{\Delta X_s \, \not = \, 0\}} ) (\omega) = 0 \quad\text{for all } a, s > 0 \text{ and } \omega \in \Omega.
	\] 
\end{enumerate}
\end{discussion}

The following lemma is a version of Galmarino's test for upper semianalytic functions. It explains that  
condition \ref{NE0} implies that \(\cE_t (\varphi) (\omega)\) depends on \(\omega\) only through \((\omega (s))_{s \leq t}\).

\begin{lemma} \label{lem: galmarino upper semianalytic}
Fix \(t \in \bR\) and a function \(\psi \colon \Omega \to \bR\). Then, \(\psi \in \USA (\Omega, \cF_t; \bR)\) if and only if \(\psi \in \USA (\Omega; \bR)\) and \(\psi = \psi (X_{\cdot \wedge t})\). 
\end{lemma} 
\begin{proof}
    First, suppose that \(\psi \in \USA (\Omega, \cF_t; \bR)\). Then, clearly \(\psi \in \USA (\Omega; \bR)\). Further, \(\psi\) is universally \(\cF_t\)-measurable by \cite[Corollary~1.10.6]{bogachev} and Galmarino's test for universally measurable functions (cf.~\cite[Lemma~2.5]{NVH}) implies that \(\psi = \psi (X_{\cdot \wedge t})\). 

    Conversely, suppose that \(\psi \in \USA (\Omega; \bR)\) and \(\psi = \psi (X_{\cdot \wedge t})\). Fix an arbitrary \(c \in \bR\). By definition of upper semianalyticity, \(\{\psi \geq c\}\) is \(\cF\)-analytic, i.e., there exists a Souslin scheme \(\{A_{n_1, \dots, n_k}\} \subset \cF\) such that 
    \(
    \{\psi \geq c \} = \bigcup_{(n_i) \in \mathbb{N}^\infty} \bigcap_{k\in\N} \, A_{n_1, \dots, n_k}.
    \)
    Using \(\psi = \psi (X_{\cdot \wedge t})\), we get that
    \(
    \{ \psi \geq c \} = X_{\cdot \wedge t}^{-1} \, (\{ \psi \geq c\}) =  \bigcup_{(n_i) \in \mathbb{N}^\infty} \bigcap_{k\in \N} \, X_{\cdot \wedge t}^{-1} \, (A_{n_1, \dots, n_k}). 
    \)
    Since \(\{X^{-1}_{\cdot \wedge t} \, (A_{n_1, \dots, n_k}) \}\subset \cF_t \) is a Souslin scheme, it follows that \(\psi \in \USA (\Omega, \cF_t; \bR)\).
\end{proof}

The dual analogue of a convex expectation is a representation penalty function. 
\begin{definition} \label{def: RPF}
	We call a function 
	\[
	\c \colon \bR \times \Omega \times \mathfrak{P} (\Omega) \to [0, \infty] 
	\]
	a {\em representation penalty function} if it satisfies the following properties:
	\begin{enumerate} [label=\textup{(P\arabic*)},ref=\textup{(P\arabic*)}]
		\item \label{PF1} The map \((t, \omega, P) \mapsto \c (t, \omega, P)\) is lower semianalytic.
		\item \label{PF2} The map \((\omega, P) \mapsto \c (t, \omega, P)\) is \(\cFt \otimes \mathcal{B}(\mathfrak{P}(\Omega))\)-lower semianalytic for all \(t \in \bR\).
		\item \label{PF3} The map \(P \mapsto \c (t, \omega, P)\) is convex for all \((t, \omega) \in \bR \times \Omega\).
		\item \label{PF4}  For every \((t, \omega) \in \bR \times \Omega\), \(\inf_{P\in \mathfrak{P}(\Omega)} \c (t, \omega, P) = 0\). 
		\item \label{PF5} The set \(\{P \colon \c (t, \omega, P) \leq c\}\) is compact in \(\mathfrak{P}(\Omega)\) for all \((t, \omega) \in \bR \times \Omega\) and \(c \in \bR_+\).
		\item \label{PF6}  For every \((t, \omega) \in \bR \times \Omega\) and \(P \in \mathfrak{P}(\Omega)\), \(\c (t, \omega, P) < \infty\) implies \(P ([\omega]_t) = 1\). 
		\item \label{PF7}  The map
		\(
		\omega \mapsto \sup_{P \in \mathfrak{P}(\Omega)} ( E^P [ \varphi ] - \c (t, \omega, P))\)
		is \(\cFt\)-measurable for all \(t \in \bR\) and \(\varphi \in C_b (\Omega; \bR)\).
			\item \label{PF8} 
		For every \(s, t \in \bR\) with \(t < s\), \(\omega \in \Omega\) and \(P \in \mathfrak{P}(\Omega)\),
		\begin{equation}\label{eq: P8 equation}
  \begin{split}
			\c (t, \omega, P) = \sup_{ \varphi \in \USA_b (\Omega, \cF_s; \bR)} \inf_{Q \in \mathfrak{P} (\Omega)} \Big( E^P \big[ \varphi \big] - E^Q & \big[ \varphi \big] + \c (t, \omega, Q) \Big) 
   \\ &+ E^P \big[ \c (s, X, P (\, \cdot \mid \cFs)) \big].
   \end{split} 
		\end{equation} 
		\end{enumerate}
We denote by \(\RPF\) the set of all representation penalty functions.

Furthermore, we say that \(\c\in \RPF\) has {\em no fixed times of discontinuity} if, for every \(t \in \bR\), \(\omega \in \Omega\) and $P \in \mathfrak{P}(\Omega)$ with \( \c (t, \omega, P ) < \infty \), 
\[
P ( \Delta X_s = 0 ) = 1 \quad\text{for all } s > t.
\] 

Finally, we call \(\c \in \RPF\) {\em homogeneous} if
\begin{align} \label{eq: RPF homog}
\c (t, \omega, P) = \c (0, \theta_t (\omega), P_t)\quad \text{for all }(t, \omega, P) \in \bR\times\Omega\times\mathfrak{P}(\Omega).
\end{align} 

\end{definition}  

\begin{discussion}
 \begin{enumerate}
 \item[(a)] In the absence of fixed times of discontinuity, the cocycle property \ref{PF8} gets more transparent. Indeed, let \(\c \colon \bR \times \Omega \times \mathfrak{P}(\Omega) \to [0, \infty]\) be a function without fixed times of discontinuity, which satisfies \ref{PF3}--\ref{PF5}. Then, by Lemma~\ref{lem: rem idenity} below, for every \(s, t \in \bR\) with \(s > t\), \(\omega \in \Omega\) and \(P \in \mathfrak{P}(\Omega)\), we have 
	\begin{align*}
		\sup_{ \varphi \in \USA_b (\Omega, \cF_s; \bR)} \inf_{Q \in \mathfrak{P} (\Omega)}  \Big( E^P &\big[ \varphi \big] - E^Q \big[ \varphi \big] + \c (t, \omega, Q) \Big)   = \inf \Big\{ \c (t, \omega, Q) \colon Q = P \text{ on } \cFs\Big\}.
	\end{align*}
\smallskip
\item[(b)] The technical condition \ref{PF7} is the dual analogue of condition \ref{NE4}. In line with Discussion~\ref{dicussion:expectation}, in the homogeneous case without fixed times of discontinuity, one can replace \( C_b(\Omega, \mathbb{R}) \) in condition \ref{PF7} with \( D_t \).
\end{enumerate}
\end{discussion}

The following is our first main result, which links convex expectations with representation penalty functions. We point out that the dual formulation establishes a connection between convex expectations and stochastic control problems, see Subsection~\ref{sec: example relaxed control} below.
\begin{theorem} \label{theo: represntation convex}
	The following are equivalent:
	\begin{enumerate}
		\item[\textup{(i)}] 
		\(\cE\) is a convex expectation. 
		\item[\textup{(ii)}] 
		There exists a representation penalty function \(\c \in \RPF\) such that 
		\begin{align} \label{eq: penalty expectation}
			\cE_t (\varphi) (\omega) = \sup_{P \in \mathfrak{P}(\Omega)} \Big( E^P \big[ \varphi \big] - \c (t, \omega, P) \Big)
		\end{align}
		for all \((t, \omega) \in \bR \times \Omega\) and \(\varphi \in \USA_b (\Omega; \bR)\).
	\end{enumerate}
	Moreover, \(\cE\) is homogeneous if and only if \(\c\) is homogeneous.
 Further, \(\cE\) has no fixed times of discontinuity if and only if \(\c\) has no fixed times of discontinuity.
 In the homogeneous case without fixed times of discontinuity, $C_b(\Omega;\bR)$ can be replaced by $D_t$ in both conditions \ref{NE4} and \ref{PF7}.
\end{theorem}

\begin{proof} (i) \(\implies\) (ii): By \cite[Theorem~2.6]{bartl_20}, there exists a map \(\c \colon \bR \times \Omega \times \mathfrak{P} (\Omega) \to [0, \infty]\) such that \ref{PF2}--\ref{PF6} and \eqref{eq: penalty expectation} hold. 
	By \cite[Remark~2.7]{bartl_20}, we have 
	\begin{align} \label{eq: penalty representation}
	\c (t, \omega, P) = \sup_{\varphi \in C_b (\Omega; \bR)} \Big( E^{P} \big[ \varphi \big] - \cE_t (\varphi) (\omega) \Big).
	\end{align} 
Furthermore, in the proof of \cite[Theorem~2.6]{bartl_20} it was shown that there exists a countable set \(D \subset C_b (\Omega; \bR)\) such that 
\[
	\c (t, \omega, P) = \sup_{\varphi \in D} \Big( E^{P} \big[ \varphi \big] - \cE_t (\varphi) (\omega) \Big).
\]
Hence, \ref{PF1} follows from \cite[Lemma~7.30]{bershre}, as \((t, \omega, P) \mapsto E^{P} [ \varphi ] - \cE_t (\varphi) (\omega)\) is lower semianalytic, given that  \((t, \omega) \mapsto \cE_t (\varphi) (\omega) \in \USA_b(\bR \times \Omega; \bR)\) by definition. 
The property \ref{PF7} is a reformulation of \ref{NE4} and therefore satisfied.
Finally, \ref{PF8} follows from \ref{NE4}, \ref{NE6} and \cite[Theorem~2.11]{bartl_20}. Thus, we have established (ii).

\smallskip
(ii) \(\implies\) (i): The properties \ref{NE0}-\ref{NE3} and \ref{NE5} follow from \cite[Theorem~2.6]{bartl_20}. Furthermore, \ref{NE4} is ensured by \ref{PF7}, and finally, \ref{NE6} is due to \cite[Theorem~2.11]{bartl_20} and \ref{PF8}. As a consequence, \(\cE\) is a convex expectation.

\smallskip 
\textit{Equivalence of absence of fixed times of discontinuity:}
	First, assume that \(\cE\) has no fixed times of discontinuity. Fix \(s, t \in \bR\) with \(s > t\), \(\omega \in \Omega\), \(c \in \bR_+\) and \(P\in\mathfrak{P}(\Omega)\) with  \(\c (t, \omega, P) \leq c\). 
	For every \(a \geq \frac{c}{2}\), we have 
	\begin{align*}
	0 = \cE_t (a \1_{\{\Delta X_s \, \not = \, 0\}}) (\omega) &= \sup_{Q \in \{ \c (t, \omega, \, \cdot \,) \,\leq\, 2 a \}} \Big(E^Q \big[ a \1_{\{ \Delta X_s \, \not= \, 0 \}} \big] - \c (t, \omega, Q) \Big) 
	\\& \geq a \, P (\Delta X_s \not = 0) - \alpha (t, \omega, P) \geq a \, P (\Delta X_s\not  = 0) - c, \phantom \int
	\end{align*}
	and therefore
	\(
	P (\Delta X_s \not = 0) \leq \frac{c}{a} \to 0\) as \( a \to \infty\). 
	This shows that \(\c\) has no fixed times of discontinuity.

	Second, assume that \(\c\) has no fixed times of discontinuity. For every \(s, t \in \bR\) with \(s > t\), \(\omega \in \Omega\) and \(a > 0\), we have 
	\begin{align*}
		\cE_t (a \1_{\{\Delta X_s \, \not= \, 0\}}) (\omega) &= \sup_{Q \in \{ \c (t, \omega, \, \cdot \,) \, \leq\, 2 a \}} \Big(E^Q \big[ a \1_{\{ \Delta X_s \, \not= \, 0 \}} \big] - \c (t, \omega, Q) \Big) 
		\\&= - \inf_{Q \in \{ \c (t, \omega, \, \cdot \,) \, \leq\, 2 a \}} \, \c (t, \omega, Q) 
		\leq - \inf_{Q \in \mathfrak{P}(\Omega)} \c (t, \omega, Q) = 0.
	\end{align*}
	This shows that \(\cE\) has no fixed times of discontinuity. 

 \smallskip 
 \textit{Equivalence of homogenity:}
	First, we assume that \(\cE\) is homogeneous. By \eqref{eq: penalty representation}, for every \((t, \omega) \in \bR \times \Omega\) and \(\varphi \in \USA_b (\Omega; \bR)\), we have 
	\begin{align*}
		&\c (t, \omega, P) = \sup_{\varphi \in C_b (\Omega; \bR)} \Big( E^{P} \big[ \varphi \big] - \cE_t (\varphi) (\omega) \Big)\\ 
		&= \sup_{\varphi \in C_b (\Omega; \bR)} \Big( E^{P} \big[ \varphi \circ \theta_{-t} \circ \theta_t \big] - \cE_0 (\varphi \circ \theta_{-t}) (\theta_t (\omega)) \Big) 
		\\
		&= \sup_{\varphi \in C_b (\Omega; \bR)} \Big( E^{P_t} \big[ \varphi \big] - \cE_0 (\varphi) (\theta_t (\omega)) \Big) 
		= \c (0, \theta_t (\omega), P_t), \phantom \int
	\end{align*} 
	where we used that \(\omega' \mapsto \theta_{-t} (\omega')\) is continuous by \cite[Proposition~B.5]{DKK24}.
	
Second, we  assume that \(\c\) is homogeneous. For every \((t, \omega) \in \bR\times \Omega\) and \(\varphi \in \USA_b (\Omega; \bR)\), 
\begin{equation*} \begin{split}
	(\cE_0 &(\varphi \circ \theta_{-t}) \circ \theta_t) (\omega) = \sup_{P \in \mathfrak{P}(\Omega)} \Big( E^P \big[ \varphi \circ \theta_{-t} \big] - \c (0, \theta_t (\omega), P) \Big)  \\&= \sup_{P \in \mathfrak{P}(\Omega)} \Big( E^P \big[ \varphi  \big] - \c (0, \theta_t (\omega), P_t) \Big)  
		= \sup_{P \in \mathfrak{P}(\Omega)} \Big( E^P \big[ \varphi  \big] - \c (t, \omega, P) \Big)  
		= \cE_t (\varphi) (\omega), \phantom \int
	\end{split}\end{equation*} 
which shows that \(\cE\) is homogeneous. 

\smallskip
Finally, let us explain that in the homogeneous case without fixed times of discontinuity, the equivalence of (i) and (ii) remains true if we replace $C_b(\Omega;\bR)$ by $D_t$ in both conditions \ref{NE4} and \ref{PF7}. The only incidents where the precise formulation of \ref{NE4} and \ref{PF7} are used are the applications of \cite[Theorem~2.11]{bartl_20}. In the homogeneous case without fixed times of discontinuity, we can use Theorem~\ref{theo: theorem 2.11 bartl} instead of \cite[Theorem~2.11]{bartl_20}, which validates our claim. 
\end{proof} 

\begin{remark} 
The space \(\USA_b (\Omega; \bR)\) of upper semianalytic functions is very natural in the context of stochastic control. 
However, in some applications one may not need this full generality, and it is natural to ask whether a version of Theorem~\ref{theo: represntation convex} could be established when the space of upper semi{\em analytic} functions is replaced by the smaller space \(\textit{USC}_b (\Omega; \bR)\) of upper semi{\em continuous} functions, and how the Definitions~\ref{def: sublinear expectation} and \ref{def: RPF} would have to be adjusted. A full discussion of this problem is beyond the scope of this remark, but let us comment on two specific points that have to be taken into account.
On the one hand, if \(\cE_s\) is only defined on the space \(\textit{USC}_b (\Omega; \bR)\) of upper semi{\em continuous} functions, the tower rule \(\cE_s \circ \cE_t = \cE_s\) can only be well-defined if the map \(\omega \mapsto \cE_t (\varphi) (\omega)\) is upper semi{\em continuous}. This invariance property entails additional regularity properties of the representation penalty function \(\c\). This is best understood in the sublinear case (see Section~\ref{sec: sublinear} below), where \(\cE_t (\textit{USC}_b (\Omega; \bR)) (\omega) \subset \textit{USC}_b (\Omega; \bR)\) is essentially equivalent to the property that the set-valued map \(\omega \mapsto \{ P \colon \c (t, \omega, P) = 0\} \) is upper hemicontinuous; see \cite[Theorem~II.20]{castaing_valadier_77}. In the upper semi{\em analytic} case, this set-valued map only needs to be measurable; see Lemma~\ref{lem: scalar measurability}.
On the other hand, working with upper semi{\em continuous} functions allows us to drop the continuity from below property from condition \ref{NE5}; see \cite[Theorem~2.2]{BCK_19}. 
\end{remark}

\begin{remark} \label{rem: regular cond property convex}
    Classical regular conditional expectations \(E [ \, \cdot \mid \mathcal{G}]\) have the property that, for almost all \(\omega\), 
    \[
    Y \text{ is \(\mathcal{G}\)-measurable} \implies E \big[ g (Y, Z) \mid \mathcal{G}\big] (\omega) = E \big[ g (Y (\omega), Z) \mid \mathcal{G} \big] (\omega).
    \]
    Theorem~\ref{theo: represntation convex} guarantees that convex expectations have a similar property. Namely, the representation \eqref{eq: penalty expectation} implies that 
    \[
    \cE_t ( g (Y, Z) ) (\omega) = \cE_t ( g ( Y(\omega), Z) ) (\omega)
    \]
    whenever \(Y (\omega)\) depends on \(\omega\) only through \((\omega (s))_{s \leq t}\).    
\end{remark}

It is well-known that linear Feller processes are quasi-left continuous and, in particular, have no fixed times of discontinuity in the classical sense; see \cite[Proposition~17.29]{kallenberg}. 
We interpret the following result in a similar way as it establishes a condition for the absence of fixed times of discontinuity in terms of a continuity property of the convex expectation, which is in the spirit of the time continuity of a Feller semigroup.
\begin{proposition} \label{prop: ftd left continuity}
	Let \(\cE\) be a convex expectation and let \(d \colon E \times E \to \bR_+\) be a bounded continuous function such that \(d (x, y) = 0\) if and only if  \(x = y\). 
	The following are equivalent:
	\begin{enumerate}
		\item[\textup{(i)}] 
		\(\cE\) has no fixed times of discontinuity.
		\item[\textup{(ii)}] 
		For every \(t, s \in \bR\) with \(s > t\), \(\omega \in \Omega\) and \(a > 0\), the map
		\[
		r \mapsto \cE_t ( a \, d ( X_s, X_r) ) (\omega) 
		\] 
		is left-continuous at \(s\).
	\end{enumerate}
\end{proposition}
\begin{proof}
	(i) \(\implies\) (ii): Fix \(s > t\) and \(\omega \in \Omega\). 
	By Theorem~\ref{theo: represntation convex}, there exists  \(\c \in \RPF\) without fixed times of discontinuity such that 
	\begin{align} \label{eq: pf prop rep}
		\cE_t (\varphi) (\omega) = \sup_{P \in \mathfrak{P}(\Omega)} \Big( E^P \big[ \varphi \big] - \c (t, \omega, P) \Big).
	\end{align}
	The function
	\(
	(t', \omega') \mapsto \omega' (t')
	\)
	is continuous at every \((t', \omega')\) with \(\Delta \omega' (t') = 0\); see  \cite[Remark, p.~452]{HWY}. Hence, since \(\c\) has no fixed times of discontinuity, the continuous mapping theorem (cf.~\cite[Theorem~5.27]{kallenberg} for a suitable formulation) implies that 
	\[
	(r, P) \mapsto E^P \big[ a\, d (X_s, X_{r}) \big]
	\] 
	is continuous on \((t, \infty) \times \{ \c (t, \omega, \, \cdot \, ) \leq 2a \|d\|_\infty \}\). 
	As \(P \mapsto \c (t, \omega, P)\) is lower semicontinuous and \(\{\c (t, \omega, \, \cdot \,) \leq 2a\|d\|_\infty\}\) is compact, 
	a version of Berge's maximum theorem (cf.~\cite[Proposition~1.3.3]{hu_papa}) implies that
	\begin{align*}
		\limsup_{r \nearrow s} \cE_t (a \, d (X_s, X_r)) (\omega) &= \limsup_{r \nearrow s} \sup_{ P \in \{ \c (t, \omega, \, \cdot \,) \, \leq \, 2a \|d\|_\infty\}} \Big( E^P \big[ a \, d (X_s, X_{r}) \big] - \c (t, \omega, P) \Big)
		\\&\leq \sup_{ P \in \{ \c (t, \omega, \, \cdot \,) \, \leq \, 2a \|d\|_\infty\}} \Big( E^P \big[ a \, d (X_s, X_{s}) \big]- \c (t, \omega, P) \Big) = 0.
	\end{align*}
	This shows $\lim_{r \nearrow s} \cE_t (a \, d (X_s, X_r)) (\omega) = 0$, proving in particular (ii).
	
	\smallskip 
	(ii) \(\Longrightarrow\) (i): We still assume that \(\c \in \RPF\) is such that \eqref{eq: pf prop rep} holds. In the following, we prove that \(\c\) has no fixed times of discontinuity.
Fix \(s > t, \omega \in \Omega\), \(c \in \bR_+\) and \(P \in \{ \c (t, \omega,\, \cdot \,) \leq c \}\). Then, for every \(a > 0\), 
\[
E^P \big[ a \, d (X_s, X_{r}) \big] - \c (t, \omega, P) \leq \cE_t (a \, d (X_s, X_{r})) \to 0 \quad\text{as } r \nearrow s.
\]
Consequently, we obtain that 
\begin{align*}
	E^P \big[ d (X_s, X_{s-}) \big] = \lim_{r \nearrow s} E^P \big[ d (X_s, X_r) \big] \leq \frac{\c (t, \omega, P)}{a} \leq \frac{c}{a} \to 0 \quad\text{as } a \to \infty.
\end{align*}
Hence, \(P\)-a.s. \(\Delta X_s = 0\), which shows that \(\c\) has no fixed times of discontinuities. By Theorem~\ref{theo: represntation convex}, we conclude that (i) holds. The proof is complete.
\end{proof}

\subsection{Uniqueness via finite-dimensional marginal distributions}
In this subsection, we use duality arguments to show that representation penalty functions, and thus convex expectations, are uniquely determined by their finite-dimensional marginal distributions. 
To this end, for \(t \in \bR\) and \(t < s \leq \infty\), let
\begin{align*}
D_{t, s} := \Big\{ \, g (X_{t_1}, \dots, X_{t_n}) \colon g \in C_b (E^n; \bR),\, t < t_1 < \dots < t_n \leq s, \, n \in \mathbb{N}\, \Big\}
\end{align*}
denote the set of finite-dimensional test functions on the interval $(t,s]$. The subsequent argumentation strongly relies on the following result.

\begin{lemma} \label{lem: rem idenity}
Let \(\c \colon \bR \times \Omega \times \mathfrak{P}(\Omega) \to [0, \infty]\) be a function without fixed times of discontinuity, which satisfies \ref{PF3}-\ref{PF6}. 
Then, for every \(-\infty< t < s \leq \infty\), \(\omega \in \Omega\) and \(P \in \mathfrak{P}(\Omega)\) with \(P ([\omega]_t) = 1\), we have
	\begin{align*}
		\sup_{ \varphi \in D_{t, s}} \inf_{Q \in \mathfrak{P} (\Omega)} &\Big( E^P \big[ \varphi \big] - E^Q \big[ \varphi \big] + \c (t, \omega, Q) \Big)
  \\&= \sup_{ \varphi \in \USA_b (\Omega, \cF_s; \bR)} \inf_{Q \in \mathfrak{P} (\Omega)} \Big( E^P \big[ \varphi \big] - E^Q \big[ \varphi \big] + \c (t, \omega, Q) \Big)  
  \\&= \inf \Big\{ \c (t, \omega, Q) \colon Q = P \text{ on } \cFs\Big\}. \phantom {\int^A}
	\end{align*}
\end{lemma}

\begin{proof}
First of all, the inequality
 \begin{align*}
\sup_{ \varphi \in D_{t, s}} &\inf_{Q \in \mathfrak{P} (\Omega)} \Big( E^P \big[ \varphi \big] - E^Q \big[ \varphi \big] + \c (t, \omega, Q) \Big)
  \\&\leq \sup_{ \varphi \in \USA_b (\Omega, \cF_s; \bR)} \inf_{Q \in \mathfrak{P} (\Omega)} \Big( E^P \big[ \varphi \big] - E^Q \big[ \varphi \big] + \c (t, \omega, Q) \Big)
 \end{align*} 
 is trivial. Next, notice that 
 \begin{align*}
     \sup_{ \varphi \in \USA_b (\Omega, \cF_s; \bR)} & \inf_{Q \in \mathfrak{P} (\Omega)} \Big( E^P \big[ \varphi \big] - E^Q \big[ \varphi \big] + \c (t, \omega, Q) \Big)
     \\&\leq \sup_{ \varphi \in \USA_b (\Omega, \cF_s; \bR)} \inf \Big\{ E^P \big[ \varphi \big] - E^Q \big[ \varphi \big] + \c (t, \omega, Q) \colon Q = P \text{ on } \cF_s \Big\}
     \\&= \inf \Big\{ \c (t, \omega, Q) \colon Q = P \text{ on } \cFs\Big\}. \phantom \int 
 \end{align*}
 We now establish the final remaining inequality:
\begin{align*}
     \inf \Big\{ \c (t, \omega, Q) \colon Q = P \text{ on } \cFs\Big\} \leq \sup_{ \varphi \in D_{t, s}} &\inf_{Q \in \mathfrak{P} (\Omega)} \Big( E^P \big[ \varphi \big] - E^Q \big[ \varphi \big] + \c (t, \omega, Q) \Big).
\end{align*}
 Of course, we may assume that 
 \[
\sup_{\varphi \in D_{t, s}} \inf_{Q \in \mathfrak{P} (\Omega)} \Big( E^P \big[ \varphi \big] - E^Q \big[ \varphi \big] + \c (t, \omega, Q) \Big) =: c < \infty.
 \]
For \(N \in \mathbb{N}\), set \(D^N_{t, s} := \{ \varphi \in D_{t, s} \colon \|\varphi\|_\infty \leq N \}\) and \(\mathcal{P}_N := \{ Q \in \mathfrak{P} (\Omega) \colon \c (t, \omega, Q) \leq N \}\). 
By virtue of \cite[Equation (7)]{bartl_20}, we obtain that 
\begin{align*} 
c &= \liminf_{N \to \infty} \sup_{\varphi \in D^N_{t, s}} \inf_{Q \in \mathcal{P}_{2N}} \Big( E^P \big[ \varphi \big] - E^Q \big[ \varphi \big] + \c (t, \omega, Q) \Big).
\end{align*}
Thanks to \ref{PF4} and \ref{PF5}, the set \(\cP_{2N}\) is nonempty and compact and, for every \(\varphi \in D^N_{t, s}\), the map 
\begin{align*} 
Q \mapsto E^P \big[ \varphi \big] - E^Q \big[ \varphi \big] + \c (t, \omega, Q)
\end{align*} 
is lower semicontinuous on \(\cP_{2N}\) (the map \(Q \mapsto E^Q [ \varphi ]\) is continuous on \(\cP_{2N}\) by the continuous mapping theorem, because \(\varphi\) is \(Q\)-a.s. continuous due to the absence of fixed times of discontinuity). As \(D^N_{t, s}\) and \(\cP_{2N}\) are convex (for \(\cP_{2N}\) this follows from \ref{PF3}), using the terminology from \cite{fan}, the map
\[
(\varphi, Q) \mapsto E^P \big[ \varphi \big] - E^Q \big[ \varphi \big] + \c (t, \omega, Q)
\]
is convex on \(\cP_{2N}\) and concave on \(D^N_{t, s}\). As a consequence, we can apply Fan's minimax theorem \cite[Theorem~2]{fan} to obtain that
\begin{align*}
c = \liminf_{N \to \infty} \inf_{Q \in \mathcal{P}_{2N}} \sup_{\varphi \in D^N_{t, s}} \Big( E^P \big[ \varphi \big] - E^Q \big[ \varphi \big] + \c (t, \omega, Q) \Big).
\end{align*}
Now, for every \(N \in \mathbb{N}\), there exists a measure \(Q^N \in \cP_{2N}\) such that 
\begin{align*}
\inf_{Q \in \mathcal{P}_{2N}} &\sup_{\varphi \in D^N_{t, s}} \Big( E^P \big[ \varphi \big] - E^Q \big[ \varphi \big] + \c (t, \omega, Q) \Big) 
\\&\geq \sup_{\varphi \in D^N_{t, s}} \Big( E^P \big[ \varphi \big] - E^{Q^N} \big[ \varphi \big] + \c (t, \omega, Q^N) \Big) - \frac{1}{N}.
\end{align*}
Hence, by passing to a suitable subsequence \((N_n)_{n = 1}^\infty\), we get that
\begin{equation} \label{eq: first ineq}
\begin{split}
c &\geq \liminf_{N \to \infty} \sup_{\varphi \in D^N_{t, s}} \Big( E^P \big[ \varphi \big] - E^{Q^N} \big[ \varphi \big] + \c (t, \omega, Q^N) \Big)
\\&= \lim_{n \to \infty} \sup_{\varphi \in D^{N_n}_{t, s}} \Big( E^P \big[ \varphi \big] - E^{Q^{N_n}} \big[ \varphi \big] + \c (t, \omega, Q^{N_n}) \Big).
\end{split}
\end{equation} 
Take \(M \in \mathbb{N}\) large enough such that 
\[
\sup_{n \geq M} \, \Big( \sup_{\varphi \in D^{N_n}_{t, s}} \Big( E^P \big[ \varphi \big] - E^{Q^{N_n}} \big[ \varphi \big] \Big) + \c (t, \omega, Q^{N_n}) \Big) \leq c + 1.
\]
As \(\varphi \equiv 0 \in D^{N_n}_{t, s}\), we have 
\[
\sup_{\varphi \in D^{N_n}_{t, s}} \Big( E^P \big[ \varphi \big] - E^{Q^{N_n}} \big[ \varphi \big] \Big) \geq 0 \quad\text{for all } n \in \mathbb{N}, 
\]
and consequently,
\[
\sup_{n \geq M}\, \c (t, \omega, Q^{N_n})\leq c + 1.
\]
This shows that \((Q^{N_n})_{n = M}^\infty \subset \cP_{c + 1}\). By the compactness of \(\cP_{c + 1}\), which is due to~\ref{PF5}, the sequence \((Q^{N_n})_{n = 1}^\infty\) has a weakly convergent subsequence with a limit point in~\(\cP_{c + 1}\). To ease our notation, we simply assume that \((Q^{N_n})_{n = 1}^\infty\) converges weakly to a probability measure \(Q^* \in \cP_{c + 1}\). 
By the lower semicontinuity of \(Q \mapsto \c (t, \omega, Q)\), we get that 
\begin{align} 
\lim_{n \to \infty} &\sup_{\varphi \in D^{N_n}_{t, s}} \Big( E^P \big[ \varphi \big] - E^{Q^{N_n}} \big[ \varphi \big] + \c (t, \omega, Q^{N_n}) \Big) \nonumber 
\\&\geq \liminf_{n \to \infty} \, N_n \sup_{\varphi \in D^{1}_{t, s}} \Big( E^P \big[ \varphi \big] - E^{Q^{N_n}} \big[ \varphi \big] \Big) + \c (t, \omega, Q^*) \label{eq: middel ineq}
\\&\geq \liminf_{n \to \infty} \, N_n \sup_{\varphi \in D^{1}_{t, s}} \Big( E^P \big[ \varphi \big] - E^{Q^{N_n}} \big[ \varphi \big] \Big). \nonumber
\end{align} 
Recalling \eqref{eq: first ineq}, and passing again to a subsequence that we ignore in our notation for simplicity, we observe that 
\[
\lim_{n \to \infty} \sup_{\varphi \in D^{1}_{t, s}} \Big( E^P \big[ \varphi \big] - E^{Q^{N_n}} \big[ \varphi \big] \Big) = 0.
\]
For all \(\varphi \in D^1_{t, s}\), this implies that 
\[
E^P \big[ \varphi \big] = \lim_{n \to \infty} E^{Q^{N_n}} \big[ \varphi \big] = E^{Q^{*}} \big[ \varphi \big], 
\]
where the last equality follows from the continuous mapping theorem, \(Q^* \in \cP_{c + 1}\) and the assumption that \(\c\) has no fixed times of discontinuity (which then implies that all functions in \(D^1_{t, s}\) are \(Q^*\)-a.s. continuous). 
Taking \ref{PF6} into account, and recalling that \(P ([\omega]_t) = 1\), we conclude that \(P = Q^*\) on \(\cF_s\), because \(D^1_{t, s}\) is measure determining for all probability measures on \((\Omega, \cF_s)\) that give probability one to \([\omega]_t\).
Finally, from \eqref{eq: first ineq} and \eqref{eq: middel ineq}, we get that 
\begin{align*}
    c &\geq \liminf_{n \to \infty} \, N_n \sup_{\varphi \in D^{1}_{t, s}} \Big( E^P \big[ \varphi \big] - E^{Q^{N_n}} \big[ \varphi \big] \Big) + \c (t, \omega, Q^*)
    \\&\geq \inf \Big\{ \c (t, \omega, Q) \colon Q = P \text{ on } \cF_s \Big\}.\qedhere
\end{align*}
\end{proof}

\begin{remark} \label{rem: D replace}
It is evident from the proof that Lemma~\ref{lem: rem idenity} remains valid whenever \(D_{t, s} \subset \USA_b (\Omega; \bR)\) is a convex set that has the following two properties: 
\begin{enumerate} 
\item[(i)] \(D_{t, s}\) is measure determining for all probability measures \(P\) on \((\Omega, \cF_s)\) such that \(P ([\omega]_t) = 1\).
\item[(ii)] If \(P (\Delta X_r = 0 \text{ for all } r > t) = 1\), then all functions from \(D_{t, s}\) are \(P\)-a.s. continuous.
\end{enumerate} 

Thanks to this observation, it is clear that the statement of Lemma~\ref{lem: rem idenity} remains valid when \(D_{t, s}\) is defined as
\begin{align} \label{eq: uni cont D}
D_{t, s} := \Big\{ \, g (X_{t_1}, \dots, X_{t_n}) \colon g \in UC_b (E^n; \bR), \, t_1, \dots, t_n \in (\mathbb{Q} \cap (t, s]) \cup \{s\}, \, n \in \mathbb{N}\, \Big\}, 
\end{align}
where \(UC_b (E^n; \bR)\) denotes the space of uniformly continuous functions w.r.t. a metric \(d^n\) on \(E^n\) that induces the product topology. It is also possible to take \(d^n\) totally bounded so that \(UC_b (E^n; \bR)\) is separable in the uniform topology (cf.~\cite[Lemma~6.3, p.~43]{Part_67}).
\end{remark}

As a consequence of Lemma~\ref{lem: rem idenity}, we obtain that convex expectations without fixed times of discontinuity are fully determined by their finite-dimensional marginal distributions.
\begin{theorem} \label{theo: FDD}
    Let \(\cE\) and \(\hcE\) be two convex expectations without fixed times of discontinuity. 
    Then, for every $t\in\bR$, the following are equivalent:
    \begin{enumerate}
        \item[\textup{(i)}] 
        \(\cE_t \leq \hcE_t\) on \(D_{t,\infty}\). 
        \item[\textup{(ii)}] 
    \(\cE_t \leq \hcE_t\) on \(\USA_b (\Omega; \bR)\). 
    \end{enumerate}
\end{theorem}

\begin{proof}
The implication (ii) \(\implies\) (i) is trivial. 
As for the converse direction, let \(\c\) and \(\hc\) be the representation penalty functions from Theorem~\ref{theo: represntation convex} such that 
\begin{align*}
\cE_t (\varphi) (\omega) = \sup_{P \in \mathfrak{P}(\Omega)} \Big( E^P \big[ \varphi \big] - \c (t, \omega, P) \Big), \\ 
\hcE_t (\varphi) (\omega) = \sup_{P \in \mathfrak{P}(\Omega)} \Big( E^P \big[ \varphi \big] - \hc (t, \omega, P) \Big).
\end{align*}
Using the previous equations, it is enough to show that 
$\hc (t, \omega, P)\le \c (t, \omega, P)$ for all $P \in \mathfrak{P}(\Omega)$ with \(P ([\omega]_t) = 1\). To that end, we apply Lemma~\ref{lem: rem idenity} with $s = \infty$ and $\cF_\infty = \cF$, and obtain
\begin{align*}
\sup_{\varphi \in D_{t, \infty}} \Big( E^P \left[ \varphi \right] - \cE_t (\varphi)(\omega) \Big) &=
\sup_{\varphi \in D_{t, \infty}} \inf_{Q \in \mathfrak{P}(\Omega)} \Big( E^P \left[ \varphi \right] - E^Q \left[ \varphi \right] + \c(t, \omega, Q) \Big) 
\\&= \c(t, \omega, P), \phantom \int
\end{align*}
and similarly,
\[
\sup_{\varphi \in D_{t, \infty}} \left( E^P \left[ \varphi \right] - \hcE_t (\varphi)(\omega) \right) = \hc(t, \omega, P).
\]
Hence, since $\cE_t (\varphi)(\omega) \leq \hcE_t (\varphi)(\omega)$ for all $\varphi \in D_{t, \infty}$ by hypothesis~(i),
we conclude that \(\widehat{\alpha} (t, \omega, P) \leq \alpha (t, \omega, P)\) for all \(P \in \mathfrak{P}(\Omega)\).
\end{proof} 
For future reference, we need the following version of the previous result.
\begin{remark}\label{rem:uniqueness}
Let $\Phi\colon \USA_b (\Omega; \bR)\to \USA_b (\Omega,\cF_0; \bR)$ be a map which satisfies the following properties:
\begin{enumerate}
    \item[(i)] \(\Phi (\varphi)  = \varphi \) for all \(\varphi \in \USA_b (\Omega,\cF_0; \bR)\).
		\item[(ii)]
		\( \Phi (\varphi)  \leq \Phi (\psi) \) for all \(\varphi, \psi \in \USA_b (\Omega; \bR)\) with \(\varphi \leq \psi\).
		\item[(iii)] 
		\( \Phi (\lambda \varphi + (1 - \lambda) \psi) \leq \lambda \Phi (\varphi) + (1- \lambda) \Phi (\psi)\) for all \(\varphi, \psi \in \USA_b (\Omega; \bR)\) and \(\lambda \in (0, 1)\).
  \item[(iv)] $\Phi$ is continuous from above on $C_b(\Omega;\bR)$ and continuous from below.
  \item[(v)] $\Phi( a \1_{\{ \Delta X_t\, \not =\, 0 \}} ) = 0$ for all $t > 0$ and  $a > 0$.
\end{enumerate}
Then, $\Phi$ is uniquely determined on $D_{0,\infty}$. Indeed,  by \cite[Theorem~2.6]{bartl_20}, there exists a function \(\c \colon \Omega \times \mathfrak{P} (\Omega) \to [0, \infty]\) such that 
\[
\Phi (\varphi) (\omega) = \sup_{P \in \mathfrak{P}(\Omega)} \Big( E^P \big[ \varphi \big] - \c ( \omega, P) \Big)
\]
for all $\varphi\in\USA_b(\Omega;\bR)$. It follows from the proof of Theorem~\ref{theo: represntation convex} that $\alpha$ has no fixed times of discontinuity. Hence, the claim follows by the same argumentation as in Theorem~\ref{theo: FDD}.
\end{remark}

\begin{discussion} \label{diss: NVH example}
Convex expectations from Definition~\ref{def: sublinear expectation}, or functions as discussed in Remark~\ref{rem:uniqueness}, have the appealing property of being determined by their finite-dimensional marginal distributions. However, not all nonlinear expectations have this property, as illustrated by the following example. Fix \(\Theta \subset \mathbb{S}^d_+\), where \(\mathbb{S}^d_+\) denotes the set of all symmetric positive semidefinite \(d \times d\) matrices with real entries, and set 
\begin{align*}
\mathcal{M} (\Theta) := \Big\{ P \in \mathfrak{P} (C (\bR_+; \bR^d)) \colon &\text{ the coordinate process \((X_t)_{t \geq 0}\) on \(C (\bR_+; \bR^d)\)} \\&\text{ is a \(P\)-martingale such that \(P\)-a.s. \(X_0 = 0\) } \phantom \int
\\ &\text{ and \((P\otimes \llambda)\)-a.e. } \frac{d \langle X, X\rangle_t}{dt} \in \Theta \, \Big\}. 
\end{align*}
Upper expectations of the form
\begin{align} \label{eq: G expectation}
\cE^\Theta_t (\varphi) (\omega) := \sup_{P \in \mathcal{M} (\Theta)} E^P \big[ \varphi (\omega \, \otimes_t\, X) \big], \quad \omega \in C (\bR; \bR^d), \ \varphi \in \USA_b (\Omega; \bR),
\end{align}
are called \emph{\(G\)-expectations} with domain of uncertainty \(\Theta\). In case that \(\Theta\) is a Borel set, such \(G\)-expectations satisfy \ref{NE0}-\ref{NE3} and \ref{NE6}. In particular, except for the continuity properties, \(G\)-expectations satisfy the conditions in Remark~\ref{rem:uniqueness}. 
As pointed out in \cite[Example~5.1]{NVH}, \(G\)-expectations are not determined by the class of continuous functions, and they are also not determined by their finite-dimensional marginal distributions, which is in contrast to convex expectations from Definition~\ref{def: sublinear expectation}. To repeat the Examples~5.1 and 5.2 from \cite{NVH}, let \(\Theta' = \{1, 2\}, \Theta^* = [1, 2)\) and \(\Theta = [1, 2]\). Then, noting that \(\mathcal{M}(\Theta)\) is the closed convex hull of both \(\mathcal{M}(\Theta')\) and \(\mathcal{M}(\Theta^*)\), we have \(\cE^\Theta = \cE^{\Theta'} = \cE^{\Theta^*}\) on \(C_b (\Omega; \bR)\). However, \(0 = \cE^{\Theta'}_0 (\1_A) < \cE^{\Theta}_0 (\1_A) = 1\) for \(A = \{ \int_0^\infty | d \langle X, X\rangle_t / dt - 3/2 |\, dt = 0 \}\). Similarly, with \(A = \{ \langle X, X\rangle_1 \geq 2\}\), we have \(0 = \cE^{\Theta^*}_0 (\1_A) < \cE^{\Theta}_0 (\1_A) =1\).
This example explains that the continuity properties of \(\Phi\) are cruical for Remark~\ref{rem:uniqueness}. 

To understand why \(\cE^{\Theta'}\) and \(\cE^{\Theta^*}\) are not convex expectations, notice they can be represented as in \eqref{eq: penalty representation} with 
\begin{align*} 
\alpha (t, \omega, P) = \begin{cases} 0, & P \circ (\omega \, \otimes_t \, X)^{-1} \in \mathcal{M} (\tilde{\Theta}), \\ \infty, &\text{otherwise}, \end{cases} \quad \mbox{for }\tilde{\Theta} \in\{ \Theta, \Theta^*\}.
\end{align*} 

However, these are not representing penalty functions in the sense of Definition~\ref{def: RPF}, as \(\mathcal{M} (\Theta')\) is not convex and \(\mathcal{M} (\Theta^*)\) is not closed. 
\end{discussion}

\subsection{Relation between convex expectations and convex semigroups} In this subsection, we introduce convex semigroups and show their relation to homogeneous convex expectations. By restricting them to  \(C_b (\Omega,\cF_0; \bR)\), we obtain a convex version of evolutionary semigroups as introduced in \cite{DKK24}.

\begin{definition} \label{def: SG}
A {\em convex semigroup} on \(\USA_b (\Omega; \bR)\) is
a family \((S_t)_{t \geq 0}\) of maps from \(\USA_b (\Omega; \bR)\) into itself, which satisfies the following properties:
	\begin{enumerate} [label=\textup{(S\arabic*)},ref=(S\arabic*)]
\item \label{SG0} \(S_t (\varphi)\in  \USA_b (\Omega, \cF_0; \bR)\) for all \(\varphi \in \USA_b (\Omega; \bR)\).		
  \item \label{SG2} \(S_0 (\varphi) = \varphi\) for all  \(\varphi \in \USA_b(\Omega,\cF_0; \bR)\).
		\item \label{SG3} \(S_t (\varphi) \leq S_t (\psi)\) for all \(t \in \bR_+\) and \(\varphi, \psi \in \USA_b (\Omega; \bR)\) with \(\varphi \leq \psi\).
	\item \label{SG4}\(S_t (\lambda \varphi + (1 - \lambda) \psi) \leq \lambda S_t (\varphi) + (1 - \lambda) S_t (\psi)\) 
 for all \(t \in \bR_+\),  \(\varphi, \psi \in \USA_b (\Omega; \bR)\) and \(\lambda \in (0, 1)\).
 \item \label{SG1} \(S_0 (\varphi)\) is \(\cF_{0}\)-measurable for all \(\varphi \in C_b (\Omega; \bR)\).
\item \label{SG6} \(S_{t + s} = S_t S_s\) for all \(t, s \in \bR_+\).		
  \item \label{SG5} For every \(t \in\bR_+\), the map \(S_t\) is continuous from above on $C_b (\Omega; \bR)$ and continuous from below.   		
	\end{enumerate}

	A convex semigroup \((S_t)_{t \geq 0}\) is called {\em homogeneous} if 
		\begin{align} \label{eq: SG homog}
		S_t (\varphi) = S_0 (\varphi \circ \theta_t)
		\end{align}
		for all \(t \in \bR_+\) and \(\varphi \in \USA_b (\Omega; \bR)\). 
		
		A homogeneous convex semigroup \((S_t)_{t \geq 0}\) is said to have {\em no fixed times of discontinuity} if 
		\[
		S_0 (a \1_{\{ \Delta X_s \, \not= \, 0\}}) = 0 \text{ for all } a, s > 0.
		\] 
\end{definition}

The following result clarifies the relation between homogeneous convex semigroups, homogeneous convex expectations and homogeneous representation penalty functions. 
\begin{theorem} \label{theo: semigroup}
	\begin{enumerate}
	\item[\textup{(i)}] 
	Let \(\cE\) be a homogeneous convex expectation. Then, 
	\[
	S_t (\varphi) := \cE_0 (\varphi \circ \theta_t), \quad t \in \bR_+, \, \varphi \in \USA_b (\Omega; \bR), 
	\] 
	defines a homogeneous convex semigroup. Further, \((S_t)_{t \geq 0}\) has no fixed times of discontinuity, provided that \(\cE\) has this property.
	\item[\textup{(ii)}]
	Let \((S_t)_{t \geq 0}\) be a homogeneous convex semigroup. Then, 
	\[
	\cE_t (\varphi) := S_0 (\varphi \circ \theta_{-t}) \circ \theta_{t}, \quad t \in \bR, \, \varphi \in \USA_b (\Omega; \bR), 
	\]
	defines a homogeneous convex expectation that has no fixed times of discontinuity once \((S_t)_{t \geq 0}\) has this property. Furthermore, there exists a homogeneous representation penalty function \(\c \in \RPF\) such that 
	\begin{align*}
	S_t (\varphi) (\omega) = \sup_{P \in \mathfrak{P}(\Omega)} \Big( E^P \big[ \varphi \circ \theta_t \big] - \c (0, \omega, P) \Big)
	\end{align*} 
	for all \((t, \omega) \in \bR_+ \times \Omega\) and \(\varphi \in \USA_b (\Omega; \bR)\).
\end{enumerate}
\end{theorem}

\begin{proof} (i). The only non-immediate property is \ref{SG6}.
By \ref{NE6} and \eqref{eq: cE homog}, we have
	\begin{align*}
		&S_t (S_s (\varphi)) = \cE_0 ( \cE_0 (\varphi \circ \theta_s) \circ \theta_t )
		= \cE_0 (\mathcal{E}_t (\varphi \circ \theta_{s + t})) 
		= \cE_0 (\varphi \circ \theta_{s + t}) 
		= S_{s + t} (\varphi).
	\end{align*}
	
	(ii). 
	The only non-immediate property is \ref{NE6}. Using \ref{SG6} and \eqref{eq: SG homog}, we get that
	\begin{align*}
		\cE_t (\cE_s (\varphi)) &= S_0 ( S_0 (\varphi \circ \theta_{-s}) \circ \theta_{s - t} ) \circ \theta_t 
		= S_{s - t} (S_{0} (\varphi \circ \theta_{-s})) \circ \theta_t 
		\\&= S_{s - t} (\varphi \circ \theta_{-s}) \circ \theta_t
		= S_0 (\varphi \circ \theta_{-s + s - t}) \circ \theta_t
		= \cE_t (\varphi).
	\end{align*}
	For the second claim, recall from Theorem~\ref{theo: represntation convex} that there exists a homogeneous \(\c \in \RPF\) such that 
	\begin{align*}
	 S_0 (\varphi \circ \theta_{-t} ) (\theta_t (\omega)) = \sup_{P \in \mathfrak{P}(\Omega)} \Big( E^P \big[ \varphi \big] - \c (t, \omega, P) \Big)
	\end{align*} 
for all \((t, \omega) \in \bR \times \Omega\) and \(\varphi \in \USA_b (\Omega; \bR)\). Using \eqref{eq: RPF homog} and \eqref{eq: SG homog}, we get that 
\begin{align*}
	S_t (\varphi)( \omega) &= \sup_{P \in \mathfrak{P}(\Omega)} \Big( E^P \big[ \varphi \circ \theta_{2t}\big] - \c (t, \theta_{-t} (\omega), P) \Big)
	\\
	&= \sup_{P \in \mathfrak{P}(\Omega)} \Big( E^{P_t} \big[ \varphi \circ \theta_{t}\big] - \c (0, \omega, P_t) \Big)
	= \sup_{P \in \mathfrak{P}(\Omega)} \Big( E^P \big[ \varphi \circ \theta_t \big] - \c (0, \omega, P) \Big).
\end{align*} 
This is the claimed formula.	
\end{proof}

By restricting a convex semigroup on $\USA_b (\Omega; \bR)$ to $\USA_b (\Omega,\cF_0; \bR)$, we obtain an evolutionary semigroup.
For linear semigroups, this concept was introduced in \cite{DKK24}.
\begin{definition} \label{def: evolutionary SG}
An {\em evolutionary semigroup} on \(C_b (\Omega,\cF_0; \bR)\) is a family \((T_t)_{t \geq 0}\) of functions $T_t\colon C_b (\Omega,\cF_0; \bR)\to C_b (\Omega,\cF_0; \bR)$
of the form
\[
T_t(\varphi)=\cE_0(\varphi\circ\theta_t),
\]
for some expectation operator \(\cE_0\colon \USA_b (\Omega; \bR) \to \USA_b (\Omega, \cF_0; \bR)\), which satisfy the following properties:
\begin{enumerate} [label=\textup{(T\arabic*)},ref=(T\arabic*)]
\item \label{T2} \(T_{t + s} = T_t T_s\) for all \(t, s \in \bR_+\).
\item \label{T3}
		\(\cE_0 (\varphi)  = \varphi \) for all \(\varphi \in \USA_b (\Omega,\cF_0; \bR)\).
 \item \label{T4}
		\(\cE_0 (\varphi)  \leq \cE_0 (\psi) \) for all \(\varphi, \psi \in \USA_b (\Omega; \bR)\) such that \(\varphi \leq \psi\).
\item \label{T5}
		\(\cE_0 (\lambda \varphi + (1 - \lambda) \psi) \leq \lambda \cE_0 (\varphi) + (1- \lambda) \cE_0 (\psi)\) for all \(\varphi, \psi \in \USA_b (\Omega; \bR)\) and \(\lambda \in (0, 1)\).
\item \label{T6}
	The map \(\omega \mapsto \cE_0 (\varphi) (\omega)\) is \(\cF_0\)-measurable for all \(\varphi \in C_b (\Omega; \bR)\).
\item \label{T7}
   The operator \(\cE_0\) is continuous from above on $C_b (\Omega; \bR)$ and continuous from below. 
   \end{enumerate}

An evolutionary semigroup \((T_t)_{t \geq 0}\) has {\em no fixed times of discontinuity} if 
		$\cE_0 (a \1_{\{ \Delta X_s \, \not= \, 0\}})$ $= 0$ for all  $a, s > 0$.
\end{definition}
\begin{discussion}
Let $(T_t)_{t\ge 0}$ be an evolutionary semigroup with expectation operator $\cE_0$ having no fixed times of discontinuity. 
\begin{enumerate}
\item[(a)] Property \ref{T3} implies that  $T_0(\varphi)=\varphi$ for all \(\varphi \in C_b (\Omega,\cF_0; \bR)\).
\smallskip
\item[(b)] $(S_t)_{t\ge 0}$ defined by $S_t(\varphi):=\cE_0(\varphi\circ\theta_t)$ for all \(\varphi\in\USA_b (\Omega; \bR)\) is a homogeneous convex semigroup without fixed times of discontinuities. Indeed, the only non-immediate property is \ref{SG6}. For $\varphi\in C_b(\Omega,\cF_0;\bR)$ and $r\in \bR_+$, we obtain that
\begin{align*}
S_s\big(S_t(\varphi\circ\theta_r)\big) =S_s\big(S_{t+r}(\varphi)\big)=T_s\big(T_{t+r}(\varphi)\big)=T_{s+t+r}(\varphi)=S_{s+t}(\varphi\circ\theta_r).    
\end{align*}
Since both \( S_{s+t} \) and \( S_s \circ S_t \) satisfy the properties of Remark~\ref{rem:uniqueness}, it follows from the same remark that they coincide on \( \USA_b(\Omega; \bR) \). This establishes \ref{SG6}.
\end{enumerate}
\end{discussion}

\subsection{Convex expectations arising from relaxed control rules} \label{sec: example relaxed control}

The purpose of this section is to explain how convex expectations arise naturally within a classical relaxed control framework for controlled SDEs as, for instance, considered in \cite{C_24_arxiv_dirichlet, nicole1987compactification, EKNJ88, ElKa15, HausLep90}. As explained in \cite{nicole1987compactification, ElKa15}, the relaxed framework typically leads to the same value functions as other classical weak and strong control frameworks. 
For our purpose, 
the relaxed control framework has the fundamental advantage that it is tailor made to use compactness and convexity techniques. Indeed, broadly speaking, the relaxed framework can be viewed as a compactification and convexification of other classical control frameworks. 
Before we begin our program, we note that this section does not aim for the most general treatment. Instead, we present a broad framework that illustrates our results. Extensions 
will be investigated in future work.

\smallskip 

Throughout this subsection, let \(\Omega = C (\bR; \bR^d)\) for a fixed dimension \(d \in \mathbb{N}\).
Let \(\Lambda\) be a compact metrizable space and consider three coefficients
\begin{align*}
 &\mu \colon \bR \times \Omega \times \Lambda \to \bR^d, \\ 
 &\sigma \colon \bR \times \Omega \times \Lambda \to \bR^{d \times r}, \\
 &h \colon \Lambda \to \bR_+, \phantom {\bR^d}
\end{align*}
where \(r \in \mathbb{N}\) represents the dimension of the noise. We impose the following conditions. 
\begin{SA} \label{SA: control}
    \(\mu, \sigma\) and \(h\) are bounded and continuous and \[\{\lambda \in \Lambda \colon h (\lambda) = 0\} \not = \emptyset.\]
    Furthermore, 
    for every \(\lambda \in \Lambda\), the maps \((t, \omega) \mapsto \mu (t, \omega, \lambda)\) and \((t, \omega) \mapsto \sigma (t, \omega, \lambda)\) are \((\cF_t)\)-predictable,
    i.e.,  \(\mu (t, \omega, \lambda)\) and \(\sigma (t, \omega, \lambda)\) depend on \(\omega\) only through \((\omega (s))_{s < t}\). 
\end{SA}

In the relaxed control setting, controls are modeled as measures on the space \(\m\) of all Radon measures on \(\bR \times \Lambda\), whose projections to \(\bR\) coincide with the Lebesgue measure. 
We endow \(\m\) with the vague topology, which turns it into a compact metrizable space; see \cite[Theorem I.2.2]{EKNJ88}. 
The coordinate map on \(\m\) is denoted by \(M\).
Define the \(\sigma\)-field \[\M := \sigma (M_{t, s} (\phi)\colon \, t < s,\, \phi \in C_{c} (\bR \times \Lambda; \bR)),\] where
\[
M_{t, s} (\phi) := \int_t^s \int_\Lambda \phi (r, \lambda) \, M(dr, d\lambda).
\]
The product space \((\Omega \times \m, \cF \otimes \M)\) captures the controlled process together with its control. Adapting our above notation, we denote the coordinate process on this space by \((X, M)\). 
For \(\varphi \in C^2_b (\bR^d; \bR)\), we define
\begin{align} \label{eq: C}
N_{t, s} (\varphi) := \varphi (X_s) - \int_{t}^{s} \int_\Lambda  L (X_r, r, X, \lambda, \varphi ) \, M (dr , d \lambda), 
\end{align}
where
\begin{align} \label{eq: L}
L (x, s, \omega, \lambda, \varphi) := \langle \mu (s, \omega, \lambda), \nabla \varphi (x) \rangle + \tfrac{1}{2} \on{tr} \big[ \sigma (s, \omega, \lambda) \sigma^* (s, \omega, \lambda) \, \nabla^2 \varphi (x) \big].
\end{align}
A relaxed control rule with initial value \((t, \omega) \in \bR \times \Omega\) is a probability measure \(P\) on the product space \((\Omega \times \m, \mathcal{F} \otimes \M)\) with
\[P (X = \omega \text{ on } (- \infty, t]) = 1,\] 
such that for every  \(\varphi \in C^2_b (\bR^d; \bR)\), the process \((N_{t, s} (\varphi))_{s \geq t}\) is a \(P\)-martingale for the filtration
\[
\cG_r := \sigma \big(X_s, M_{t, s} (\phi)\colon\, t \leq s \leq r,\, \phi \in C_c (\bR \times \Lambda; \bR)\big), \quad r \geq t.
\]
For \((t, \omega) \in \bR \times \Omega\), we define 
\begin{align} \label{eq: KR}
	\cK (t, \omega)&:= \Big\{ \text{all relaxed control rules with initial value } (t, \omega) \Big \}.
\end{align}

Next, we collect some properties of the set \(\cK\) of relaxed control rules. For the definition of upper hemicontinuity of a correspondence, we refer to \cite[Definition~17.2]{charalambos2013infinite}; see also~\cite[Section~17.3]{charalambos2013infinite} for a useful characterization in terms of sequences. The proof follows the well-trotted path for the one-sided time interval; see \cite{CN_23_arxiv_jumps, nicole1987compactification}.
\begin{lemma} \label{lem: properties relaxed controls}
The correspondence \(\cK\) is upper hemicontinuous with nonempty compact convex values. In particular, it has a Borel measurable graph.
\end{lemma} 
\begin{proof} 
	We only sketch the main steps of the proof; see \cite{CN_23_arxiv_jumps} for more details. Take \((t, \omega) \in \bR \times \Omega\). \(\cK (t, \omega) \not = \emptyset\) is a direct consequence of Skorokhod's existence theorem (cf. \cite[Theorem~4, p.~265]{Skorokhod}) and its convexity is evident by fact that the set of solutions to martingale problems is convex.
	
	By \cite[Theorem~17.20]{charalambos2013infinite}, \((t, \omega) \mapsto \cK (t, \omega)\) is upper hemicontinuous with compact values if for every sequence \((t^n, \omega^n)_{n \in\N_0} \subset \bR \times \Omega\) such that \((t^n, \omega^n) \to (t^0, \omega^0)\), every sequence \((P^n)_{n \in\N }\) with \(P^n \in \cK (t^n, \omega^n)\) has an accumulation point in \(\cK (t^0, \omega^0)\). To prove this, one first shows that the sequence \((P^n)_{n \in \mathbb{N}}\) is tight. This follows from Kolmogorov's tightness criterion (cf. \cite[Theorem~23.7]{kallenberg}) and a moment estimate based on the Burkholder--Davis--Gundy inequality in a straightforward manner. Second, assuming that \(P^n \to P^0\), one needs to show that \(P^0 \in \cK (t^0, \omega^0)\). This can be proved by standard martingale problem arguments (cf., e.g., \cite[Proposition~IX.5.22]{JS} for details). 

Finally, the measurable graph follows from \cite[Theorem~18.6]{charalambos2013infinite} and the fact that upper hemicontinuous correspondences are weakly measurable.
\end{proof}

Next, we introduce a penalty function on the product space \(\Omega \times \m\). To do so,
for \((t, \omega, P) \in \bR \times \Omega \times \mathfrak{P} (\Omega \times \m)\), we define
\[
\hat{\alpha} (t, \omega, P) := \begin{cases} E^P \big[ \int_t^\infty \hspace{-0.1cm}\int_\Lambda h (\lambda)\, M(ds, d \lambda) \big], & P \in \cK (t, \omega), \\ + \, \infty, & P \not \in \cK (t, \omega). \end{cases} 
\]
Further, for \((t, \omega, \varphi) \in \bR \times \Omega \times \USA_b (\Omega; \bR)\), we set
\begin{align*}
    \cE_t (\varphi) (\omega) := \sup_{P \in \mathfrak{P} (\Omega \times \m)} \Big( E^P [ \varphi (X)] - \hat{\alpha} (t, \omega, P) \Big).
\end{align*}
Equivalently, \(\cE\) has the representation
\begin{align} \label{eq: control rep}
    \cE_t (\varphi) (\omega) = \sup_{P \in \cK (t, \omega)} E^P \Big[ \, \varphi (X) - \int_t^\infty \hspace{-0.1cm}\int_\Lambda h (\lambda) \, M(ds, d \lambda)\, \Big], 
\end{align}
which is the value function corresponding to a relaxed control problem. 
The following theorem clarifies that \(\cE\) is a convex expectation on the path space \(\Omega\). 
\begin{theorem} \label{theo: CE control setting}
    \(\cE\) is a convex expectation in the sense of Definition~\ref{def: sublinear expectation}. Furthermore, the corresponding representation penalty function \(\alpha\in \RPF\) is given by the duality formula
    \[
    \alpha (t, \omega, P) = \sup_{\varphi \in C_b (\Omega; \bR)} \Big( E^P \big[ \varphi \big] - \cE_t (\varphi) (\omega) \Big) 
    \]
    for \(P \in \mathfrak{P}(\Omega)\) and \((t, \omega) \in \bR \times \Omega\). In case \(h \equiv 0\), the penalty \(\alpha\) is given by
    \[
    \alpha (t, \omega, P) = \begin{cases} 0, & P \in \{ Q \circ X^{-1} \colon Q \in \cK (t, \omega)\}, \\ + \infty, &P \not \in \{ Q \circ X^{-1} \colon Q \in \cK (t, \omega) \}. \end{cases} 
    \]
\end{theorem}

\begin{proof}
We start with a preparatory fact. Let \(\lambda_0 \in \Lambda\) be such that \(h (\lambda_0) = 0\).
Thanks to Standing Assumption~\ref{SA: control}, and Skorokhod's existence theorem (cf. \cite[Theorem~4, p.~265]{Skorokhod}), there exists a \(P_0 \in \cK (t, \omega)\) such that \(P_0\)-a.s. \(M (ds, d \lambda) = ds \, \delta_{\lambda_0}(d \lambda)\). 

Recalling the representation \eqref{eq: control rep}, since \(\cK\) has a Borel graph by Lemma~\ref{lem: properties relaxed controls}, it follows from \cite[Proposition~7.47]{bershre} that \(\cE\) maps \(\USA_b (\Omega; \bR)\) to upper semianalytic functions.
That \(\cE (\varphi)\) is bounded for \(\varphi \in \USA_b (\Omega; \bR)\) follows from the existence of \(P_0\) as explained above, because it entails that
\begin{align} \label{eq: inf zero} 
\inf_{P \in \cK^R (t, \omega)} E^P \Big[ \int_t^\infty \hspace{-0.1cm} \int_\Lambda h (\lambda) \, M(ds, d \lambda)\Big] = 0.
\end{align} 
Further, \ref{NE0} follows from Lemma~\ref{lem: galmarino upper semianalytic} and the fact that, by definition, \(\cK (t, \omega)\) depends on \(\omega\) only through \((\omega (s))_{s \leq t}\).
The properties \ref{NE1} and \ref{NE3} are clearly satisfied. 

To see that \ref{NE2} holds, notice that, for \(\varphi\in \USA_b (\Omega, \cF_t; \bR)\), by \eqref{eq: control rep}, Lemma~\ref{lem: galmarino upper semianalytic} and \eqref{eq: inf zero}, we have
\[
\cE_t (\varphi) (\omega) = \varphi (\omega) - \inf_{P \in \cK (t, \omega)} E^P \Big[ \int_t^\infty \hspace{-0.1cm} \int_\Lambda h (\lambda) \, M(ds, d \lambda)\Big] = \varphi (\omega).
\]

Next, to show \ref{NE4}, for \(t \in \bR\) and \(\varphi \in C_b (\Omega; \bR)\), notice that \(\omega \mapsto \cE_t (\varphi) (\omega)\) is measurable by the measurable maximum theorem (cf. \cite[Theorem 18.19]{charalambos2013infinite}) and Lemma~\ref{lem: properties relaxed controls} (recall that upper hemicontinuity entails weak measurability for correspondences).
Now, \ref{NE4} follows from Galmarino's test (cf. \cite[Theorem~IV.96]{dellacheriemeyer}). 

As explained in the proof of \cite[Theorem~2.6]{bartl_20}, \ref{NE5} holds once \(P \mapsto \hat{\alpha}(t, \omega, P)\) is convex with nonempty compact level sets. Thus, as \(\cK\) has nonempty convex compact values (cf.~Lemma~\ref{lem: properties relaxed controls}), it suffices to understand that \(P \mapsto \hat{\alpha} (t, \omega, P)\) is lower semicontinuous. This follows from \cite[Theorem~A.3.12]{dupuisellis}, which states that maps of the form \(\mu \mapsto \int f d \mu\) are lower semicontinuous whenever \(f\) with values in \([0, \infty]\) is lower semicontinuous. We emphasize at this point that convexity and lower semicontinuity of \(P \mapsto \hat{\alpha}(t, \omega, P)\) are consequences of the relaxed control framework.

Finally, we need to prove the tower property \ref{NE6}. By the dynamic programming principle as given in \cite[Theorem~3.4]{ElKa15} (extended in the obvious manner to the two-sided time setting), for \(s > t\), \(\omega \in \Omega\) and \(\varphi \in \USA_b (\Omega; \bR)\), we have 
\begin{align*}
    \cE_t (\varphi) (\omega) = \sup_{P \in \cK (t, \omega)} E^P \Big[ \, \cE_s (\varphi) (X) - \int_t^s \hspace{-0.1cm} \int_\Lambda h (\lambda) \, M(dr, d\lambda) \, \Big]. 
\end{align*}
With little abuse of notation, let \(P \otimes_s Q \in \cK (t, \omega)\) be such that \(P \otimes_s Q = P\) on \(\mathcal{G}_s\) and \(P\)-a.s. \(\1_{(s, \infty)} (r) \, M (dr, d \lambda) = \1_{(s, \infty)} (r) \, dr \delta_{\lambda_0} (d \lambda)\), where \(\lambda_0 \in \Lambda\) satisfies \(h(\lambda_0) = 0\). By virtue of Standing Assumption~\ref{SA: control} and Skorokhod's existence theorem, it is clear that \(P \otimes_s Q\) exists, as \(\cK\) is stable under pasting (cf. \cite[Lemma~3.3]{ElKa15}). As a consequence, 
\begin{align*}
    \sup_{P \in \cK (t, \omega)} E^P  \Big[ \, \cE_s &(\varphi) (X) - \int_t^s \hspace{-0.1cm} \int_\Lambda h (\lambda) \, M(dr, d\lambda) \, \Big] 
    \\
    &= \sup_{P \in \cK (t, \omega)} E^{P\, \otimes_s\, Q} \Big[ \cE_s (\varphi) (X) - \int_t^s \hspace{-0.1cm} \int_\Lambda h (\lambda) \, M(dr, d\lambda) \, \Big]
\\
&= \sup_{P \in \cK (t, \omega)} E^{P\, \otimes_s\, Q} \Big[ \cE_s (\varphi) (X) - \int_t^\infty \hspace{-0.1cm} \int_\Lambda h (\lambda) \, M(dr, d\lambda) \, \Big]
\\
&\leq \sup_{P \in \cK (t, \omega)} E^{P} \Big[ \cE_s (\varphi) (X) - \int_t^\infty \hspace{-0.1cm} \int_\Lambda h (\lambda) \, M(dr, d\lambda) \, \Big] 
\\&= \cE_t (\cE_s (\varphi))(\omega). \phantom {\int^A}
\end{align*}
Because \(h \geq 0\), we obviously have 
\begin{align*}
\sup_{P \in \cK (t, \omega)} E^P \Big[ \, \cE_s  &(\varphi) (X) - \int_t^s \hspace{-0.1cm} \int_\Lambda h (\lambda) \, M(dr, d\lambda) \, \Big] 
\\
&\geq \sup_{P \in \cK (t, \omega)} E^P \Big[ \, \cE_s (\varphi) (X) - \int_t^\infty \hspace{-0.1cm} \int_\Lambda h (\lambda) \, M(dr, d\lambda) \, \Big]
= \cE_t (\cE_s (\varphi))(\omega), \phantom {\int^A}
\end{align*}
and therefore, \(\cE_t (\cE_s (\varphi)) (\omega) = \cE_t (\varphi)(\omega)\). This shows that \ref{NE6} holds. 

The duality formula follows from \cite[Remark~2.7]{bartl_20} and the formula for the case \(h \equiv 0\) is immediate by virtue of \eqref{eq: control rep}.
\end{proof}

Finally, we provide conditions under which the convex expectation \(\cE\) is homogeneous. 

\begin{theorem} \label{theo: CE relax homogeneous}
       Suppose that \(h \equiv 0\) and 
       \begin{align} \label{eq: homogeneous hypo}
       \mu (s + t, \omega, \lambda) = \mu (s, \theta_t (\omega), \lambda), \quad \sigma (s + t, \omega, \lambda) = \sigma (s, \theta_t(\omega), \lambda)
       \end{align} 
       for all \(s, t \in \bR\), \(\omega \in \Omega\) and \(\lambda \in \Lambda\). Then, the convex expectation \(\cE\) in Theorem~\ref{theo: CE control setting} is homogeneous. 
\end{theorem}
\begin{proof}
Since \(h \equiv 0\), we have 
\begin{align*}
    \cE_t (\varphi) (\omega) = \sup_{P \in \cK (t, \omega)} E^P \big[ \varphi (X) \big].
\end{align*}
As a consequence, \(\cE\) is homogeneous once 
\begin{align} \label{eq: K shift identity projection}
\Big\{ P \circ (X_{s})_{s \geq 0}^{-1} \colon P \in \cK (0, \theta_t (\omega)) \Big\} = \Big\{ P \circ (X_{s + t})_{s \geq 0}^{-1} \colon P \in \cK (t, \omega) \Big\} \subset \mathfrak{P} (\Omega)
\end{align} 
for all \((t, \omega) \in \bR \times \Omega\), see also Theorem~\ref{theo: represntation} below. In the following we prove this identity.
By a version of \cite[Lemma 3.2]{LakSPA15}, there exists a \((\mathcal{G}_t)\)-predictable probability kernel \(m \colon \bR \times  \m \times \mathcal{B}(\Lambda) \to \mathbb{R}_+\) such that 
\begin{align*} 
M (ds, d \lambda) = m_s (d \lambda) (M)  ds.
\end{align*} 
For \(t \in \bR\), we define the shift operator \(\bar{\theta}_t \colon \Omega \times \m \to \Omega \times\m\) by 
\[
(X, m_r (d \lambda) dr) (\bar{\theta}_t ( \overline{\omega})) := (X_{\,\cdot\, + t}, m_{r + t} (d \lambda) dr) (\overline{\omega}).
\]
With this definitions at hand, we claim that 
\begin{align} \label{eq: K shift identity control}
\cK (0, \theta_t (\omega)) = \Big\{ P \circ {\bar{\theta}}_t^{-1} \colon P \in \cK (t, \omega) \Big\}.
\end{align} 
Projecting this identity to \((X_s)_{s \geq 0}\), we immediately get \eqref{eq: K shift identity projection}. 
To prove \eqref{eq: K shift identity control}, we first establish a general fact: for \(t, t' \in \bR\) and \(\omega \in \Omega\), 
\begin{align} \label{eq: important fact} 
P \in \cK (t, \omega) \implies P \circ \bar{\theta}_{t'}^{-1} \in \cK (t - t', \theta_{t'} (\omega)).
\end{align} 
Notice that  
\(
P \circ \bar{\theta}^{-1}_{t'} ([ \theta_{t'} (\omega) ]_{t - t'}) = P ([\omega]_t) = 1.
\)
Further, with \eqref{eq: C} and \eqref{eq: L}, using \eqref{eq: homogeneous hypo}, for \(r > t\), we obtain that 
\begin{align*}
    N_{t, r} (\varphi) 
    &= \varphi (X_{r}) - \int_{t - t'}^{r - t'} \hspace{-0.1cm} \int_\Lambda \Big( \langle \mu (s + t', X, \lambda), \nabla \varphi (X_{s + t'}) \rangle 
    \\&\hspace{3cm}+ \tfrac{1}{2} \on{tr} \big[ \sigma \sigma^* (s + t', X, \lambda)\, \nabla^2 \phi (X_{s + t'}) \big] \Big) m_{s + t'} (M) (d \lambda) ds \phantom {\int^t}
    \\&= \varphi (\bar{\theta}_{t'} (X_{r - t'})) - \int_{t - t'}^{r - t'} \hspace{-0.1cm} \int_\Lambda \Big( \langle \mu (s, \bar{\theta}_{t'} (X), \lambda), \nabla \varphi (\bar{\theta}_{t'} (X_s)) \rangle 
    \\&\hspace{3cm}+ \tfrac{1}{2} \on{tr} \big[ \sigma \sigma^* (s, \bar{\theta}_{t'} (X), \lambda)\, \nabla^2 \phi (\bar{\theta}_{t'} (X_s)) \big] \Big) m_{s} (\bar{\theta}_{t'} (M)) (d \lambda) ds \phantom {\int^t}
    \\&= N_{t - t', r - t'} (\varphi) \circ \bar{\theta}_{t'}. \phantom {\int^t}
\end{align*}
From this formula, it follows that \((N_{t - t', r})_{r \geq t - t'}\) is a \(P \circ \bar{\theta}_t^{-1}\)-martingale for the filtration \((\cG_r)_{r \geq t - t'}\) (cf. \cite[Chapter 10.2]{jacod79}). Hence, \(P \circ \bar{\theta}^{-1}_{t'} \in \cK (t - t', \theta_{t'} (\omega))\) and \eqref{eq: important fact} is proved.

Using \eqref{eq: important fact} with \(t' = 0\) yields the inclusion \(\supset\) from \eqref{eq: K shift identity control}. 
For the converse inclusion, fix \(P \in \cK (0, \theta_t (\omega))\), and notice that \(P = P \circ \bar{\theta}_{-t}^{-1} \circ \bar{\theta}_{t}^{-1}\). Using \eqref{eq: important fact} with \(t' = -t\) yields that \(P \circ \bar{\theta}_{-t}^{-1} \in \cK (0 - (-t), \theta_{-t} (\theta_t (\omega))) = \cK (t, \omega)\), which shows the inclusion \(\subset\) from \eqref{eq: K shift identity control}.
\end{proof}

\begin{remark}
		Condition \eqref{eq: homogeneous hypo} can be reformulated as \(\mu (t, \omega, \lambda) = \mu (0, \theta_t (\omega), \lambda)\) and \(\sigma (t, \omega, \lambda) = \sigma (0, \theta_t (\omega), \lambda)\), i.e., the coefficients depend on \((t, \omega)\) only through \(\theta_t (\omega)\). This mirrors the ``classical homogeneous Markovian'' setting, where the coefficients should only depend on \(\omega (t)\). It is clear that \eqref{eq: homogeneous hypo} is satisfied for coefficients of the form   
		\[
		g \Big(\sup_{s \in [t - h, t]} a (X_s),\, \inf_{s \in [t - h, t]} b (X_s),\, \int_{t - h}^t c (X_s)\, ds, \, X_t,\, \lambda\, \Big), \quad \text{ for some \(h \in [0, \infty]\)}.
		\]
\end{remark}

\begin{remark}
In his 2010 ICM talk, Peng \cite{peng_10_IMC} outlined that backward SDEs (BSDEs) have a one-to-one relation to so-called path-dependent PDEs (PPDEs), which are PDEs on the space of continuous functions. The notion of Dupire's functional derivatives \cite{dupire_22} then initiated the study of control problems in terms of PPDEs. A recent stream of literature (cf.~\cite{cosso_russo_22,zhou_23}) deals with Crandall--Lions viscosity solutions, which are defined in terms of classical tangency conditions and Dupire's derivatives.
In this context, \cite{B-N_16} studies the connection between certain convex expectations and semilinear PPDEs within a BSDE related framework. For an overview on PPDEs, we refer to the recent survey \cite{peng_et_all_PPDE_survey}; see also \cite{BLT2023, BT2023, EKTZ2014}.

Related to the above setting, under suitable assumptions on the coefficients, it was shown in~\cite{CN_23_EJP} that the value function associated to so-called nonlinear continuous semimartingales solves the path-dependent PDE 
\[
\begin{cases}
    \frac{d}{dt}\, v (t, \omega) + G (t, \omega, v) = 0, & (t, \omega) \in [0, T) \times C ([0, T]; \bR^d), \\ v (T, \omega) = \varphi (\omega), & \omega \in C ([0, T]; \bR^d), 
\end{cases}
\]
with 
\[
G (t, \omega, \psi) := \sup \Big\{ \langle \mu (s, \omega, \lambda), \nabla \psi (t, \omega) \rangle + \tfrac{1}{2} \on{tr} \big[ \sigma \sigma^* (s, \omega, \lambda) \, \nabla^2 \psi (t, \omega) \big] \colon \lambda \in \Lambda \Big\} 
\]
in a Crandall--Lions viscosity sense, where all derivatives are interpreted in Dupire's sense. As shown in \cite{CN_23_arxiv_jumps},  for \(h \equiv 0\), the value function \(v\) from \cite{CN_23_EJP} has a representation of the form~\eqref{eq: control rep}. Therefore, in this case the convex expectation from Theorem~\ref{theo: CE control setting} can be connected to a PPDE. For a state-dependent setting; see the recent paper \cite{C_24_arxiv_dirichlet}.
\end{remark}

\section{The sublinear case} \label{sec: sublinear} 
In this section, we specialize the results from the previous section to the case where the convex expectation or the convex semigroup is additionally positively homogeneous. We note that positive homogeneity together with convexity implies sublinearity. 
\begin{definition}
A convex expectation \(\cE\) is called a {\em sublinear expectation} if it satisfies the positive homogeneity property:
\begin{enumerate}[label=\textup{(Ep)},ref=(Ep)]
    \item \label{SE} 
    \(\cE_t(c \, \varphi)(\omega) = c \, \cE_t(\varphi)(\omega)\) for all \(t\in\bR\), \(c \in \mathbb{R}_+\) and \(\varphi \in \text{USA}_b(\Omega; \mathbb{R})\).
\end{enumerate}

Similarly, a convex semigroup \((S_t)_{t \geq 0}\) is called a {\em sublinear semigroup} if it satisfies the positive homogeneity property:
\begin{enumerate}[label=\textup{(Sp)},ref=(Sp)]
    \item \label{SL1} 
    \(S_t(c \, \varphi) = c \, S_t(\varphi)\) for all \(t, c \in \mathbb{R}_+\) and \(\varphi \in \text{USA}_b (\Omega, \cF_0; \mathbb{R})\).
\end{enumerate}
\end{definition}
In the sublinear case, the representation penalty function assumes values only in \(\{0, \infty\}\). Consequently, sublinear expectations can be represented using representation maps, which are set-valued functions, also called correspondences. This concept is defined next.
\begin{definition} \label{def: RS}
A {\em representation map} is a set-valued function \[\cU \colon \bR \times \Omega \twoheadrightarrow \mathfrak{P}(\Omega),\]
which satisfies the following properties:
	\begin{enumerate} [label=\textup{(R\arabic*)},ref=(R\arabic*)]
		\item \label{M1} \(\cU\) has nonempty, convex and compact values.
		\item \label{M2}
		\(\cU\) has an analytic graph, i.e., \(\on{gr} \, (\cU)\) is an analytic subset of \(\bR \times \Omega \times \mathfrak{P}(\Omega)\).
		\item \label{M3} For every \(t \in \bR\), the correspondence \(\omega \mapsto \cU (t, \omega)\) is weakly \(\cFt\)-measurable.\footnote{A correspondence \(y \mapsto \cV(y)\) between a measurable space \((Y, \mathcal{Y})\) and a topological space \(Z\) is called {\em weakly measurable} (resp. {\em measurable}) if \(\varphi^\ell (F) = \{ y \in Y \colon \cV (y) \cap F \not = \emptyset\}\in \mathcal{Y}\) for all open (resp. closed) \(F \subset Z\). For correspondences with nonempty compact values in a separable metrizable space (such as \(\cU\) under \ref{M1}), weak measurability is equivalent to measurability, see \cite[Theorem~18.10]{charalambos2013infinite}. For a thorough discussion of measurable correspondences, we refer to \cite[Chapter~17]{charalambos2013infinite}.}
		\item \label{M4} For every \((t, \omega) \in \bR \times \Omega\) and \(P \in \cU (t, \omega)\), 
		\(P ( [\omega]_t ) = 1\).
        \item \label{M5}
		\(\cU\) is stable under conditioning, i.e., 
		for every \(s, t \in \bR\), \(\omega \in \Omega\) with \(t < s\) and \(P \in \cU (t, \omega)\), the regular conditional probability $P (\, \cdot \mid \cFs)$ satisfies \[P\text{-a.s.} \ \ P (\, \cdot \mid \cFs) \in \cU (s, X).\]
		\item \label{M6}			
		\(\cU\) is stable under pasting, i.e., for every \(s, t \in \bR\) with \(t <  s\), \(\omega \in \Omega\),  \(P \in \cU (t, \omega)\), and any \(\cFs\)-measurable stochastic kernel \(\omega \mapsto Q_\omega\), the following implication holds:
		\[
		P\text{-a.s.} \ \ Q \in \cU (s, X) \quad \Longrightarrow \quad P \otimes_s Q \in \cU (t, \omega).
		\]
		\end{enumerate}
  We denote by \(\cR\) the set of all representation maps. 
	
	We say that \(\cU\) has {\em no fixed times of discontinuity} if for every \(t \in \bR\), \(\omega \in \Omega\), and \(P \in \cU (t, \omega)\),
	\[
	P (\Delta X_s = 0) = 1 \text{ for all } s > t.
	\]
	
	We call \(\cU \) {\em homogeneous} if
		\begin{align} \label{eq: cU homog}
		\cU (0, \theta_{t} (\omega)) = \big\{ P_t \colon P \in \cU (t, \omega) \big\}
		\end{align}
		for all \((t, \omega) \in \bR_+ \times \Omega\).
\end{definition}

\begin{discussion}
\begin{enumerate}
\item[(a)] The properties \ref{M2}, \ref{M5} and \ref{M6} are standard conditions that entail the dynamic programming principle on path spaces, see  \cite{nicole1987compactification,ElKa15,NVH}. Below, we will see that these conditions are, in a certain sense, not only sufficient but also necessary. 

\item[(b)] The property \ref{PF7} is closely related to the concept of `scalar measurability' that is very natural for correspondences with closed and convex values (cf.~\cite[Section~18.5]{charalambos2013infinite}).
 The question under which conditions scalar measurability is equivalent to weak measurability is classical and frequently investigated, see \cite{barbati_hess_98,castaing_valadier_77}. 
\end{enumerate}
\end{discussion}

Next, we connect representation maps to {\em sublinear} expectations and their homogeneity concepts.

\begin{theorem} \label{theo: represntation}
	The following are equivalent:
	\begin{enumerate}
		\item[\textup{(i)}] 
		\(\cE\) is a sublinear expectation without fixed times of discontinuity. 
		\item[\textup{(ii)}] 
		There exists a representation map \(\cU \in \cR\) without fixed times of discontinuity such that 
		\begin{align} \label{eq: worst case expectation}
		\cE_t (\varphi) (\omega) = \sup_{P\, \in\, \cU (t, \omega)} E^P \big[ \varphi \big]
		\end{align}
		for all \((t, \omega) \in \bR \times \Omega\) and \(\varphi \in \USA_b (\Omega; \bR)\).
	\end{enumerate}
	Moreover, \(\cE\) is homogeneous if and only if \(\cU\) is homogeneous. 
\end{theorem}

\begin{proof}

(i) \(\Longrightarrow\) (ii): By Theorem~\ref{theo: represntation convex}, there exists a representation penalty function \(\c \in \RPF\) without fixed times of discontinuity such that 
	\[
	\cE_t (\varphi) (\omega) = \sup_{P \in \mathfrak{P} (\Omega)} \Big( E^P \big[ \varphi \big] - \c (t, \omega, P) \Big) 
	\]
	for all \((t, \omega) \in \bR\times\Omega\) and \(\varphi \in \USA_b(\Omega; \bR)\). As observed in the proof of \cite[Theorem~1.1]{bartl_20}, the sublinearity property \ref{SE} implies that \(\c \in \{0, \infty\}\). Consequently, as \(\varphi \in \USA_b (\Omega; \bR)\) is bounded, we obtain that 
	\begin{align*}
	\cE_t (\varphi) (\omega) &= \sup_{P \in \{ \c (t, \omega, \, \cdot)\, =\, 0\}}  \Big( E^P \big[ \varphi \big] - \c (t, \omega, P) \Big) 
	=  \sup_{P \in \{ \c (t, \omega, \, \cdot)\, =\, 0\}}  E^P \big[ \varphi \big]
	\end{align*} 
for all \((t, \omega) \in \bR \times \Omega\). Thus, 
	\[
\cU (t, \omega) := \Big\{ P \colon \c(t, \omega, P) = 0 \Big\}, \quad (t, \omega) \in \bR \times \Omega, 
\] 
is a good candidate for a representation map. We already note that \(\cU\) has no fixed times of discontinuity, as \(\c\) has no fixed times of discontinuity.
It remains to show that \(\cU\) satisfies the properties \ref{M1}-\ref{M6}.

Fix \((t, \omega) \in \bR \times \Omega\). Note that \(\cU (t, \omega)\) is compact by \ref{PF5} and convex by \ref{PF3}. 
Furthermore, we have \( \cU (t, \omega) = \{ \alpha (t, \omega, \, \cdot \,) = 0 \} = \{ \alpha (t, \omega, \, \cdot \,) < \infty \} \not = \emptyset\), as otherwise we get a contradiction to \ref{PF4}.
Hence, \ref{M1} holds. 
Note that 
\(
\on{gr} \, (\cU) = \{ \c \leq 0\},
\)
which is analytic by the definition of lower semianalyticity and \ref{PF1}, i.e.,  \ref{M2} holds. 	
	By Lemma~\ref{lem: scalar measurability}, \ref{M3} follows from \ref{PF7}.
Moreover, \ref{M4} holds due to \ref{PF6}.
We now prove \ref{M5} and \ref{M6}. 
Fix \(s > t\) and \(P \in \cU (t, \omega)\). By \ref{PF8} and Lemma~\ref{lem: rem idenity}, 
\begin{align} \label{eq: zero}
\inf \big\{ \c (t, \omega, Q) \colon Q = P \text{ on } \cFs \big\} + E^P \big[ \c (s, X, P (\, \cdot \mid \cFs)) \big] = 0.
\end{align} 
Since \(\c \geq 0\), we obtain that \(P\)-a.s., 
\(
\c (s, X, P (\, \cdot \mid \cFs))  = 0, 
\)
so that \(P\)-a.s. \(P (\, \cdot \mid \cFs) \in \cU (s, X)\) and \ref{M5} holds. 
To show \ref{M6}, let \(\omega' \mapsto Q_{\omega'}\) be an \(\cFs\)-measurable stochastic kernel such that \(P\)-a.s. \(Q \in \cU (s, X)\). Set \(P^* := P \otimes_s Q\) and recall that \(P^* (d\omega) = E^P [ Q_\cdot ( d \omega) ]\). Clearly, we have \(P^* = P\) on \(\cFs\) by property \ref{M4}. 
Then, for \(A \in \cFs\) and \(G \in \cF\), it follows from \ref{M4} that 
\[
E^{P^*} \big[\1_A \1_G \big] = E^P \big[ \1_A Q (G) \big] = E^{P^*} \big[ \1_A Q (G) \big], 
\]
which shows that 
\begin{align} \label{eq: cond ex}
P^*\text{-a.s.} \quad P^* (\, \cdot \mid \cFs) = Q.
\end{align}
By \ref{PF2} and Lemma~\ref{lem: galmarino upper semianalytic}, it follows from \eqref{eq: cond ex} that
\[
E^{P^*} \big[ \c (s, X, P^* (\, \cdot \mid \cFs)) \big] = E^{P^*} \big[ \c (s, X, Q)\big] = E^{P} \big[ \c (s, X, Q)\big] = 0,
\]
where we used that \(P\)-a.s. \(Q \in \cU (s, X)\).
Using this identity, \ref{PF8} for \(P^*\), \(P = P^*\) on \(\cFs\) and \eqref{eq: zero}, we conclude that
\begin{align*}
\c (t, \omega, P^*) &= \inf \big\{ \c (t, \omega, Q) \colon Q = P^* \text{ on } \cFs \big\} + E^{P^*} \big[ \c (s, X, P^* (\, \cdot \mid \cFs)) \big]
\\&= \inf \big\{ \c (t, \omega, Q) \colon Q = P \text{ on } \cFs \big\} = 0.
\end{align*} 
This shows that \(P^* \in \cU (t, \omega)\), and therefore \ref{M6}. We conclude that~\(\cU \in \cR\). 

\smallskip
(ii) \(\Longrightarrow\) (i): 
Take a representation map \(\cU \in \cR\) without fixed times of discontinuity and define 
	\begin{align} \label{eq: convex indicator} 
	\c (t, \omega, P) := \begin{cases} 0, & P \in \cU (t, \omega), \\ + \infty, & P \not \in \cU (t, \omega), \end{cases}
	\end{align} 
	which is the indicator function (in the sense of convex analysis) of the correspondence~\(\cU\).
	In the following, we prove that \(\c \in \RPF\). Indeed, as \(\cU\) has no fixed times of discontinuity, it is evident that the same is true for \(\c\).
	
	For any \(c \in \bR\), we have \(\{\c \leq c\} = \on{gr}\, (\cU)\), which shows that \ref{M2} implies \ref{PF1}. 
	Take \(t \in \bR\) and \(\omega' \in \Omega\).
	By \ref{M2}, there is a Souslin scheme \(\{A_{n_1, \dots, n_k}\}\subset\mathcal{B}(\bR) \otimes \cF \otimes \mathcal{B} (\mathfrak{P}(\Omega))\) so that 
	\(
	\on{gr}\, (\cU) = \bigcup_{(n_i) \in \mathbb{N}^\infty} \bigcap_{k = 1}^\infty A_{n_1, \dots, n_k}.
	\)
	Define \(\phi \colon \Omega \times \mathfrak{P}(\Omega) \to \bR \times \Omega \times \mathfrak{P}(\Omega)\) by \(\phi (\omega, P) := (t, \omega \, \otimes_t \, \omega', P)\). Then, for every \(c \in \bR\), due to \ref{M3}, it holds 
	\begin{align*}
	\{ (\omega, P) \colon \c (t, \omega, P) \leq c \} &= \{ (\omega, P) \colon P \in \cU (t, \omega \, \otimes_t \, \omega')\} = \phi^{-1} (\on{gr}\, (\cU)) 
	\\&= \bigcup_{(n_i) \in \mathbb{N}^\infty} \bigcap_{k = 1}^\infty \phi^{-1} (A_{n_1, \dots, n_k}).
	\end{align*}
	Since \(\phi\) is \((\cFt \otimes \mathcal{B}(\mathfrak{P}(\Omega)) / \mathcal{B}(\bR) \otimes \cF \otimes \mathcal{B} (\mathfrak{P}(\Omega)))\)-measurable, we have \( \phi^{-1} (A_{n_1, \dots, n_k}) \in \cFt \otimes \mathcal{B}(\mathfrak{P}(\Omega))\), and consequently, \ref{PF2} holds.
	Next, \ref{PF3}-\ref{PF5} follow from \ref{M1}, and \ref{PF6} is due to \ref{M4}. 
	By Lemma~\ref{lem: scalar measurability}, the property \ref{PF7} follows from \ref{M3}.
		Finally, it remains to show \ref{PF8}. Fix \(s, t \in \bR\) with \(s > t\), \(\omega \in \Omega\) and \(P \in \mathfrak{P}(\Omega)\). First, assume that \(\c (t,\omega, P) < \infty\), i.e., \(P \in \cU (t, \omega)\). Then, by Lemma~\ref{lem: rem idenity}, we obtain that
	\[
		\sup_{ \varphi \in \USA_b (\Omega, \cF_s; \bR)} \Big( E^P \big[ \varphi \big] - \cE_t (\varphi) (\omega) \Big) = \inf \big\{ \c (t, \omega, Q) \colon Q = P \text{ on } \cFs \big\} = 0.
	\]
	Moreover, \ref{M5} implies that \(P\)-a.s. \(P (\, \cdot \mid \cFs) \in \cU (s, X)\) and hence, 
	\[
	E^P \big[ \c (s, X, P (\, \cdot \mid \cFs))\big] = 0.
	\]
	Consequently, \eqref{eq: P8 equation} holds. 
	Conversely, suppose that \(P \not \in \cU (t, \omega)\) and, by contradiction, that \eqref{eq: zero} holds (cf.~ Lemma~\ref{lem: rem idenity}). 
	Then, there exists \(Q \in \cU (t, \omega)\) such that \(Q = P\) on \(\cFs\) and \(P\)-a.s. \(P (\, \cdot \mid \cFs) \in \cU(s, X)\). But \ref{M6} yields that
	\[
	P = P\, \otimes_s \, P (\, \cdot \mid \cFs) = Q\, \otimes_s \, P (\, \cdot \mid \cFs) \in \cU (t, \omega), 
	\]
	which is a contradiction. Consequently, \eqref{eq: zero} is violated, and hence, as \(\c\) takes only values in \(\{0, \infty\}\), \eqref{eq: P8 equation} holds. 

 Finally, we show the equivalence of the statements about homogeneity. First, we show that the correspondence \(\cU\) is homogeneous in case \(\cE\) has this property. To that end, fix \((t, \omega) \in \bR_+ \times \Omega\) and let \(\c\) be as in \eqref{eq: convex indicator}.
	By Theorem~\ref{theo: represntation convex}, the representation penalty function \(\c\) is homogeneous and consequently, we obtain that 
	\begin{align*}
		\cU (0, \theta_{t} (\omega)) &= \Big\{ P \colon \c (0, \theta_{t} (\omega), P) = 0 \Big\} 
		= \Big\{ P \colon \c (0, \theta_{t} (\omega), (P_{-t})_{t}) = 0 \Big\} 
		\\&= \Big\{ P \colon \c (t, \omega, P_{-t}) = 0 \Big\} 
		= \Big\{ P_t \colon \c (t, \omega, P) = 0 \Big\} 
		= \Big\{ P_t \colon P \in \cU (t, \omega) \Big\}, 
	\end{align*}
	which shows that \(\cU\) is homogeneous. Conversely, suppose that \(\cU\) is homogeneous. 
	Then, for all \((t, \omega) \in \bR_+ \times \Omega\), using \eqref{eq: cU homog}, we obtain that
	\begin{align*}
		\cE_t (\varphi \circ \theta_t ) (\omega) &= \sup_{P \in \cU (t, \omega)} E^{P_t} \big[ \varphi \big]  
		= \sup_{P \in \cU (0, \theta_{t} (\omega))} E^P \big[ \varphi \big] 
		= \cE_0 (\varphi) (\theta_{t} (\omega)), \phantom \int
	\end{align*}
	which shows that \(\cE\) is homogeneous. 
\end{proof}

\subsection{Nonlinear L\'evy processes} \label{sec: ex Levy}
As an illustration of the previous results, we relate the concept of nonlinear L\'evy processes, as introduced in \cite{neufeld2017nonlinear}, to our framework.
Extensions to more general nonlinear semimartingales (as considered in \cite{CN_23_EJP, CN_23_arxiv_jumps, CN_24_SPA}) are possible, but they are more technical. Here, we aim for an illustration of our framework in a setting with jumps and set \(\Omega := D (\bR; \bR^d)\) for a fixed dimension \(d \in \mathbb{N}\). 

Let \(\bS^d_+\) be the space of all symmetric positive semidefinite \(d \times d\) matrices with real entries. Denote by \(\mathcal{L}\) the space of all L\'evy measures, i.e., all measures \(\nu\) on \((\bR^d, \mathcal{B}(\bR^d))\) such that \(\nu (\{0\}) = 0\) and \(\int_{\bR^d} (1 \wedge \|x\|^2) \, \nu (dx) < \infty\). We endow \(\mathcal{L}\) with the topology induced by the test functions \(x \mapsto (1 \wedge \|x\|^2) f (x)\), \(f \in C_b (\bR^d; \bR)\). It is known from \cite[Lemma~A.1]{C_24_JMAA} that \(\mathcal{L}\) is Polish with this topology. 

Let \(\fPas\) be the set of all Borel probability measures \(P\) on \(D (\bR_+; \bR^d)\) such that the one-sided coordinate process \((X_t)_{t \geq 0}\) is a \(P\)-semimartingale for its natural filtration\footnote{It is well-known (cf.~\cite[Proposition~2.2]{neufeld2014measurability}) that it does not matter whether this filtration is taken raw or right-continuous.} whose semimartingale characteristics \((B^P, C^P, \nu^P)\) are absolutely continuous w.r.t. the Lebesgue measure \(\llambda\). For notational convenience, we denote the densities by \((b^P, c^P, K^P)\) in the product space \(\bR^d \times \bS^d_+ \times \mathcal{L}\). 

We introduce uncertainty to the semimartingale characteristics, by considering a Borel subset \(\Theta \subset \bR^d \times \bS^d_+ \times \mathcal{L}\). 
We use the notation
\[
\mathfrak{P} (\Omega) \ni P \mapsto \bar{P}_t := P \circ (X_{s + t})_{s \geq 0}^{-1} \in \mathfrak{P} (D (\bR_+; \bR)), \quad t \in \bR.
\]
For \((t, \omega) \in \bR \times \Omega\), we define the uncertainty set 
\begin{equation} \label{eq: C levy}
\begin{split}
 \cU (t, \omega) := \Big\{ P \in \mathfrak{P} (\Omega) \colon &P ( [\omega]_t) = 1, \ \bar{P}_t \in \fPas \\&\text{with  \((\bar{P}_t \otimes \llambda)\)-a.e. \((b^{\bar{P}_t}, c^{\bar{P}_t}, K^{\bar{P}_t}) \in \Theta\)} \, \Big\}.
\end{split}
\end{equation} 
The following theorem provides 
conditions on the domain of uncertainty \(\Theta\) such that \(\cU\) is a homogeneous representation map without fixed times of discontinuity. 

\begin{theorem} \label{theo: main levy}
Suppose that \(\Theta\) is nonempty, convex and compact. Then, 
    \(\cU\) is a homogeneous representation map without fixed times of discontinuity. 
\end{theorem}
Similar statements within the one-sided time framework with path-dependent domains of uncertainty were proved in \cite{CN_23_EJP,CN_23_arxiv_jumps}. Topological properties of the ambiguity set associated to nonlinear L\'evy process were investigated in \cite{LN_19_transactions}, although with a different topology for the space \(\mathcal{L}\) of L\'evy measures.
It is worth commenting more precisely on the relation between Theorem~\ref{theo: main levy} and
\cite[Theorem~2.5]{LN_19_transactions}. The latter, in particular, highlights a tightness condition that is
necessary for the closedness of a set of L\'evy processes. At first glance, this tightness condition appears to
be missing in Theorem~\ref{theo: main levy}. However, as we explain now, it is implicitly enforced by our choice
of topology on \(\mathcal{L}\). The paper \cite{LN_19_transactions} endows the space \(\mathcal{L}\) with the weakest topology under which all the mappings \(\nu \mapsto \int f (x) \, \nu (dx)\)
are continuous for every bounded continuous function \(f \colon \bR^d \to \bR\) vanishing around the origin. As outlined in \cite[Remark~5.3]{CN_23_arxiv_jumps}, a set \(K \subset \mathcal{L}\) is compact in the topology used in the present paper if and only if it is compact in the topology used in \cite{LN_19_transactions} while satisfying the tightness condition \(\lim_{\delta \to 0} \sup_{\nu \in K} \int_{\|x\| \leq \delta} \|x\|^2 \, \nu (dx) = 0\). This observation shows that our framework incorporates condition (J) from \cite{LN_19_transactions}, thereby clarifying the connection between Theorem~\ref{theo: main levy} and \cite[Theorem~2.5]{LN_19_transactions}.
The proof of Theorem~\ref{theo: main levy} can be reduced to properties of the correspondence \(\omega \mapsto
\cU (0, \omega)\) that we establish in the following lemma.

\begin{lemma} \label{lem: main levy}
Suppose that \(\Theta\) is nonempty, convex and compact.
    Then, the correspondence \(\omega \mapsto \mathcal{U} (0,\omega)\) is weakly measurable with nonempty convex compact values. 
\end{lemma}

\begin{proof} 
We define
\begin{align} \label{eq: P theta}
    \mathcal{P} (\Theta) :=  \Big\{ P \in \fPas \colon P (X_0 = 0) = 1, \text{ \((P \otimes \llambda)\)-a.e. \((b^{P}, c^{P}, K^{P}) \in \Theta\)} \, \Big\}
\end{align}
Clearly, as \(\Theta \not = \emptyset\), also \(\mathcal{P} (\Theta) \not = \emptyset\). Since \(\Theta\) is convex, \cite[Theorem~III.3.40]{JS} yields that \(\mathcal{P}(\Theta)\) is convex. Furthermore, it follows from \cite[Theorem~5.1, Propositions~6.1,~6.3]{CN_23_arxiv_jumps} that \(\mathcal{P} (\Theta)\) is compact. 

Next, we transfer these observations to the correspondence \(\omega \mapsto \mathcal{U} (0, \omega)\). 
First, for every \(\omega \in \Omega\), the map 
\(\omega' \mapsto \omega \otimes_0 \omega'\) is continuous from \(D (\bR_+; \bR^d)\) to \(\Omega\). This follows in a straightforward manner from the characterization of Skorokhod \(J_1\) convergence as given by \cite[Theorem~15.10]{HWY}.

As a consequence, by \cite[Theorem~8.10.61]{bogachev}, the map 
\begin{align} \label{eq: continuity measure map Levy}
\Omega \times \mathfrak{P} (D (\bR_+; \bR^d))\to \mathfrak{P} (\Omega),\quad  (\omega, P) \mapsto \varkappa (\omega, P) := P \circ (\omega \, \otimes_0 \, X)^{-1}  
\end{align} 
is measurable in the first and continuous in the second argument, i.e., it is a Carath{\'e}odory function. Next, we claim that  
    \begin{align} \label{eq: main Levy}
    \mathcal{U} (0,\omega) = \Big\{ P \circ (\omega \, \otimes_0 \, X)^{-1} \colon P \in \mathcal{P}(\Theta) \Big\} \, \big( = \varkappa (\omega, \mathcal{P} (\Theta)) \big).
    \end{align}
    To prove this identity, first let \(P \in \mathcal{U} (0,\omega)\) be a measure, define 
    \[
    Q := P \circ (X_s - \omega (0))_{s \geq 0}^{-1} \in \mathfrak{P} (D (\bR_+; \bR^d)),
    \] 
    and notice that \(P = Q \circ (\omega \, \otimes_0 \, X)^{-1}\). For \(\bar{P} := \bar{P}_0 = P \circ (X_s)_{s \geq 0}^{-1}\), let \((b^{\bar{P}}, c^{\bar{P}}, K^{\bar{P}})\) be the Lebesgue densities of the \(\bar{P}\)-characteristics of \((X_s)_{s \geq 0}\). 
    Then, \cite[Lemma~2.9]{jacod80} yields \(Q \in \fPas\) with~\((Q \otimes \llambda)\)-a.e. 
    \[(b^Q, c^Q, K^Q) = (b^{\bar{P}}, c^{\bar{P}}, K^{\bar{P}}) \circ (X_s + \omega (0))_{s \geq 0}.\]
    Now, we get that 
    \begin{align*}
        (Q \, \otimes \, \llambda) ( (b^Q, c^Q, K^Q) \not \in \Theta ) = (\bar{P} \, \otimes \, \llambda) ( (b^{\bar{P}}, c^{\bar{P}}, K^{\bar{P}}) \not \in \Theta) = 0. 
    \end{align*}
    In summary, \(Q \in \mathcal{P} (\Theta)\), which shows the inclusion \(\subset\) in \eqref{eq: main Levy}. For the converse inclusion, let \(P \in \mathcal{P} (\Theta)\). Then, \cite[Lemma~2.9]{jacod80} yields \(Q := P \circ (X_s + \omega (0))^{-1}_{s \geq 0} \in \fPas\) with \((Q \otimes \llambda)\)-a.e. \((b^{Q}, c^{Q}, K^{Q}) \in \Theta\). Hence, \(Q \circ (\omega \, \otimes_0 \, X)^{-1} \in \mathcal{U} (0,\omega)\), which finishes the proof of \eqref{eq: main Levy}.

 The identity \eqref{eq: main Levy} shows that \(\mathcal{U} (0,\omega)\) is nonempty and convex, because \(\mathcal{P} (\Theta)\) has these properties. Furthermore, by the continuity of \eqref{eq: continuity measure map Levy} in the second variable, also compactness of \(\mathcal{P}(\Theta)\) transfers to \(\mathcal{U} (0,\omega)\). 
    Finally, the weak measurability of \(\omega \mapsto \mathcal{U} (0, \omega)\) follows from \cite[Theorem~8.2.8]{aubin}, because \(\varkappa\) is a Carath{\'e}odory function.
\end{proof} 

\begin{proof}[Proof of Theorem~\ref{theo: main levy}]
Notice that 
\begin{align*}
    \cU (t, \omega) = \Big\{ P_t \colon P \in \mathcal{U} (0,\theta_{-t} (\omega)) \Big\}.
\end{align*}
This shows that \(\cU\) is homogeneous. 
Furthermore, nonemptiness and convexity are entailed by Lemma~\ref{lem: main levy}. 
By \cite[Proposition~B.5]{DKK24}, the map \((t, \omega) \mapsto \theta_t (\omega)\) is continuous from \(\bR \times \Omega\) into \(\Omega\). Thus, the map \((t, P) \mapsto P_t\) is continuous from \(\bR \times \mathfrak{P}(\Omega)\) into \(\mathfrak{P}(\Omega)\) by \cite[Theorem~8.10.61]{bogachev}. This observation shows that \(\cU (t, \omega)\) is compact by the compactness of \(\mathcal{U} (0,\theta_{-t} (\omega))\) which follows from Lemma~\ref{lem: main levy}. We conclude that \ref{M1} holds. 
Since \(\cU (t, \omega) = \tau (t, \mathcal{U} (0,\theta_{-t} (\omega)))\) with \(\tau (t, P) := P_t\), \cite[Theorem~8.2.8]{aubin} yields that weak measurablilty transfers from \(\mathcal{U}(0, \,\cdot \, )\) to \(\cU\). In particular, the correspondence \(\cU\) has a measurable graph by \cite[Theorem~18.6]{charalambos2013infinite}, which implies \ref{M2}. Moreover, as \(\cU (t, \omega)\) depends on \(\omega\) only through \((\omega (s))_{s \leq t}\), Galmarino's test (cf.~\cite[Theorem~IV.96]{dellacheriemeyer}) upgrades the weak measurability of \(\cU\) to \ref{M3}.
Property \ref{M4} follows directly from the definition. 
Finally, \ref{M5} and \ref{M6} can be proved similar to \cite[Theorem~2.1 (i)]{neufeld2017nonlinear}. As this proof is rather lengthy, we skip the details for brevity. 
It remains to note that \(\cU\) has no fixed times of discontinuity, because semimartingales with absolutely continuous characteristics are quasi-left continuous, see \cite[Corollary~II.1.19]{JS}. The proof is complete.
\end{proof}

In the remainder of this subsection, we relate the sublinear expectation associated with the representation map from Theorem~\ref{theo: main levy} to the concept of nonlinear L\'evy processes. The following definitions are adapted from \cite{denk2020semigroup, K21}.

\begin{definition} 
We say that \((\hat{\Omega}, \hat{\cF}, \hat{\cE})\) is a \emph{sublinear expectation space} if \((\hat{\Omega}, \hat{\cF})\) is a measurable space and \(\hat{\cE} \colon \mathcal{L}^\infty  (\hat{\Omega}, \hat{\cF}) \to \cF\) satisfies the following properties:\footnote{As usual, \(\mathcal{L}^\infty (\hat{\Omega}, \hat{\cF})\) denotes the space of all bounded \(\hat{\cF}\)-measurable functions \(Y \colon \Omega \to \bR\).}
\begin{enumerate}
    \item[\textup{(a)}] \(\hat{\cE} (Y) \leq \hat{\cE} (Z)\) for all \(Y, Z \in \mathcal{L}^\infty  (\hat{\Omega}, \hat{\cF})\) with \(Y \leq Z\).
    \item[\textup{(b)}] \(\hat{\cE}(c) = c\) for all \(c \in \bR\).
    \item[\textup{(c)}] \(\hat{\cE} (Y + Z ) \leq \hat{\cE} (Y) + \hat{\cE} (Z)\) for all \(Y, Z \in \mathcal{L}^\infty  (\hat{\Omega}, \hat{\cF})\).
    \item[\textup{(d)}] \(\hat{\cE} (\lambda\, Y) = \lambda \, \hat{\cE} (Y)\) for all \(Y \in \mathcal{L}^\infty  (\hat{\Omega}, \hat{\cF})\) and \(\lambda \in \bR_+\).
\end{enumerate}
\end{definition} 

\begin{definition} 
Let \((\hat{\Omega}, \hat{\cF}, \hat{\cE})\) be a sublinear expectation space and let \((Y_t)_{t \geq 0}\) be a \cadlag \(\bR^d\)-valued stochastic process on the measurable space \((\hat{\Omega}, \hat{\cF})\). Then,  \((Y_t)_{t \geq 0}\) is called \emph{ \(\hat{\cE}\)-L\'evy process} if 
\begin{enumerate}
    \item[\textup{(i)}] \(Y_t\) is \(\hat{\cF}\)-measurable for every \(t \in \bR_+\).
    \item[\textup{(ii)}] \(\hat{\cE} ( f (Y_0) ) = f (0)\) for all \(f \in C_b (\bR^d; \bR)\).
    \item[\textup{(iii)}] \(\hat{\cE} ( f (Y_{t} - Y_s)) = \hat{\cE} (f (Y_{t - s}))\) for all \(t \geq s\) and \(f \in C_b (\bR^d; \bR)\).
    \item[\textup{(iv)}] For all \(n \in \mathbb{N}\), \(t_1 < t_2 < \dots < t_{n + 1}\) and \(f \in C_b (\bR^{d \, (n + 1)}; \bR)\), 
    \begin{align*}
    \hat{\cE} \Big( \hat{\cE} \big(f (& Y_{t_1}, \dots, Y_{t_n}, Y_{t_{n + 1}} - Y_{t_n}) \big) \Big) 
    \\&= \hat{\cE} \Big( \hat{\cE} \big(f ( y_1, \dots, y_n, Y_{t_{n + 1}} - Y_{t_n}) \big) \big|_{ (y_1, \dots, y_n)\,=\,(Y_{t_1}, \dots, Y_{t_n})} \Big).
    \end{align*}
    \item[\textup{(v)}] \(\hat{\cE} ( f (Y_t)) \to f (0)\) as \(t \searrow 0\), for all \(f \in C_b (\bR^d; \bR)\).
\end{enumerate}
\end{definition} 

The properties (iii) and (iv) correspond to the stationary independent increment property of a classical linear L\'evy process. 

\smallskip
It was shown in \cite{neufeld2017nonlinear} that the coordinate process on \(D (\bR_+; \bR^d)\) is an \(\hat{\cE}\)-L\'evy process under
\[
\hat{\cE} (Z) = \sup_{P \in \mathcal{P} (\Theta)} E^P \big[ Z \big], \quad Z \in \mathcal{L}^\infty (D (\bR_+; \bR^d), \mathcal{B} (D (\bR_+; \bR^d))), 
\]
where \(\mathcal{P} (\Theta)\) is defined in \eqref{eq: P theta}.
In the articles \cite{denk2020semigroup,K21}, nonlinear L\'evy processes are related to nonlinear convolution semigroups.  
Further, \cite{K21} provides a one-to-one correspondence between the finite-dimensional marginal distributions of a nonlinear L\'evy process and an associated family of L\'evy--Khinchine triplets. 
A fundamental difference between the modeling approaches from \cite{neufeld2017nonlinear} and \cite{denk2020semigroup,K21} is that the framework from \cite{neufeld2017nonlinear} is built on the path space \(D (\bR_+; \bR^d)\), while \cite{denk2020semigroup,K21} work with semigroups on the state space and finite dimensional distributions. A nonlinear L\'evy process as constructed in \cite{neufeld2017nonlinear} always leads to a nonlinear L\'evy process in the sense of \cite{denk2020semigroup,K21}. In the following, we investigate the converse direction, i.e., we ask whether the finite dimensional objects considered in \cite{denk2020semigroup,K21} extend uniquely to operators on the path space as studied in~\cite{neufeld2017nonlinear}.

\begin{definition}
\label{def: LK triplet}
Let \((Y_t)_{t \geq 0}\) be an \(\hat{\cE}\)-L\'evy process. We call \(\hat{\Theta} \subset \bR^d \times \mathbb{S}^d_+ \times \mathcal{L}\) the \emph{nonlinear L\'evy--Khinchine tiplet} of \((Y_t)_{t \geq 0}\) if, for every \(f \in C^\infty_c (\bR^d; \bR)\) and \(x \in \bR^d\), 
\[
\lim_{t \searrow 0} \, \frac{\hat{\cE} (f (x + Y_t) ) - f (x)}{t} = H (f, x), 
\]
where
\begin{align*} 
H (f, x) := \sup_{(b, c, K) \, \in \, \hat{\Theta}} \Big( \langle b, \nabla & f (x) \rangle + \tfrac{1}{2} \on{tr} \big[ c \, \nabla^2 f (x) \big] 
\\&+ \int_{\bR^d} \big( f (x + y) - f (x) - \langle y, \nabla f (x) \rangle \1_{\{\|y\| \leq 1\}} \big) K (dy) \Big).
\end{align*} 
\end{definition} 

As we will see in the next section, the nonlinear L\'evy--Khinchine triplet is closely related to the pointwise generator of a Markovian semigroup on the state space.

\smallskip
For an \(\hat{\cE}\)-L\'evy process \((Y_t)_{t \geq 0}\), its finite-dimensional marginal distribution \(\hat{\cE} \circ Y^{-1}\) is defined by 
\[
\hat{\cE} \circ Y^{-1} ( f (X_{t_1}, \dots, X_{t_n}) ) = \hat{\cE} ( f (Y_{t_1}, \dots, Y_{t_n}))
\]
on 
\begin{align*}
\mathsf{FD} := \Big\{ f (X_{t_1}, \dots, X_{t_n}) \colon n \in \mathbb{N}, \, \ &t_1 < \dots < t_n, \, f \in C_b (\bR^{d \, n}, \bR), \\ & X \equiv \text{coordinate process on \(D (\bR_+; \bR^d)\)} \Big\}. 
\end{align*} 
The following theorem shows that the finite-dimensional marginal distribution of a nonlinear L\'evy process with suitable L\'evy--Khinchine triplet extends uniquely to the Neufeld--Nutz representation from \cite{neufeld2017nonlinear}. We recall from the Discussion~\ref{diss: NVH example} that the convexity and compactness assumptions on \(\Theta\) are necessary for the representation \eqref{eq: Phi Levy} to hold.

\begin{theorem} \label{theo: extension Levy}
Let \((Y_t)_{t \geq 0}\) be an \(\hat{\cE}\)-Levy process with nonlinear L\'evy--Khinchine triplet \(\Theta\), which is assumed to be nonempty, convex and compact. Then, \(\hat{\cE} \circ Y^{-1}\) extends uniquely to a map \(\Phi \colon \USA_b (\Omega; \bR) \to \USA_b (\Omega, \cF_0; \bR)\) as in Remark~\ref{rem:uniqueness} that is given by 
\begin{align} \label{eq: Phi Levy}
 \Phi (\varphi) (\omega) = \sup_{P \in \mathcal{P} (\Theta)} E^P \big[ \varphi (\omega \, \otimes_0 \, X) \big], \quad \varphi \in \USA_b (\Omega; \bR),
\end{align} 
where \(\mathcal{P}(\Theta)\) is defined in \eqref{eq: P theta}.
\end{theorem} 

\begin{proof}
Due to \eqref{eq: main Levy}, the function \(\Phi\) defined in \eqref{eq: Phi Levy} has the representation
\[
\Phi (\varphi) = \sup_{P \in \cU (0, \omega)} E^P \big[ \varphi \big], 
\]
where \(\cU\) is defined in \eqref{eq: C levy}.
Hence, Theorem~\ref{theo: represntation} and Theorem~\ref{theo: main levy} imply that \(\Phi\) has the properties in Remark~\ref{rem:uniqueness}. Further, by \cite[Theorem~2.1]{neufeld2017nonlinear} and \cite[Theorem~6.4]{K21}, \(\Phi\) is an extension of \(\hat{\cE} \circ Y^{-1}\). Here, we mention that the prerequisites from \cite[Theorem~6.4]{K21} are implied by our assumptions on \(\Theta\), see the proof of Lemma~\ref{lem: main levy} for details. Finally, the uniqueness of the extension follows from Remark~\ref{rem:uniqueness}. 
\end{proof}

\section{A comparison result for convex expectations with Markovian structure} \label{sec: 1D extensions}

The purpose of this section is to show that convex expectations with an additional Markovian structure are determined by their one-dimensional distributions. 
Furthermore, we use this result to relate such convex expectations to their infinitessimal description that is given through pointwise generators. Finally, we put our results in the context of optimal control. 

\begin{definition}
    Let \(\cE\) be a homogeneous convex expectation. We call \(\cE\) Markovian if, for all \(g \in C_b (E; \bR)\) and \(t > 0\), 
    \(
    \cE_0 ( g (X_t) ) (\omega)
    \)
    depends on \(\omega\) only through \(\omega (0)\).\footnote{I.e., there exists a function $h\colon E\to\bR$ such that \(\cE_0(g(X_t))(\omega)=h(\omega(0))\)
for all $\omega\in\Omega$.} In that case, we call 
    \[
    T_t (g) (x) := \cE_0 ( g (X_t) ) (\omega) \, \Big|_{\omega (0) = x,} \quad t \in \bR_+, \, g \in C_b (E; \bR), \, x \in E, 
    \]
    the 
    \emph{Markovian semigroup} on \(C_b (E; \bR)\) associated to \(\cE\). Furthermore, such a Markovian \(\cE\) is called \emph{continuity preserving} if the function
    \(
    (x_1, x_2, \dots, x_n) \mapsto T_t (g (x_1, \dots, x_{n}, \, \cdot \,)) (x_n)
    \)
    is an element of \(C_b (E^n; \bR)\) for all \(g \in C_b (E^{n + 1}; \bR)\), \(t > 0\) and \(n \in \mathbb{N}\).
\end{definition}

\begin{lemma} \label{lem: initial point SGP}
Let \(\cE\) be a Markovian homogeneous convex expectation with associated Markovian semigroup \((T_t)_{t \geq 0}\). Then, \(T_t T_s = T_{t + s}\) for all \(s, t \in \bR_+\).
\end{lemma} 
\begin{proof}
    Let g \(\in C_b (E; \bR)\), \(s, t \in \bR_+\), \(x \in E\) and \(\omega \in \Omega\) with \(\omega (0) = x\). 
    Then, by homogeneity and \ref{NE6}, we obtain that
    \begin{align*} 
    T_t (T_s (g)) (x) &= \cE_0 (\cE_0 (g (X_s)) (\theta_t (X))) (\omega) 
    = \cE_0 (\cE_t (g (X_{t + s}))) (\omega) 
    \\&= \cE_t (g (X_{t + s})) (\omega) = T_{t + s} (g) (x). \qedhere
    \end{align*} 
\end{proof}

\begin{example} \label{ex: nonlinear Levy Markovian}
The sublinear expectation associated to the class of nonlinear L\'evy processes in Section~\ref{sec: ex Levy} is Markovian and continuity preserving. Indeed, 
let \(\mathcal{U}\) be as in \eqref{eq: C levy} and set
\[
\cE_t (\varphi) (\omega) := \sup_{P \in \mathcal{U} (t, \omega)} E^P \big[ \varphi \big].
\]
With \(\mathcal{P}(\Theta)\) as in \eqref{eq: P theta}, we obtain that 
\begin{align*}
\cE_0 (g (x_1, \dots, x_n, X_t)) (\omega) &= \sup_{P \in \mathcal{U} (0, \omega)} E^P \big[ g (x_1, \dots, x_n, X_t) \big] 
\\&= \sup_{P \in \mathcal{P} (\Theta)} E^P \big[ g (x_1, \dots, x_n, X_t + \omega (0)) \big].
\end{align*} 
Hence, \(\cE\) is Markovian and, since \(\mathcal{P} (\Theta)\) is compact as shown in the proof of Lemma~\ref{lem: main levy}, Berge's maximum theorem (cf.~\cite[Theorem~17.31]{charalambos2013infinite}) implies that \(\cE\) is continuity preserving.
\end{example}

The following theorem explains that Markovian homogeneous convex expectations are determined by 
their associated Markovian semigroups.
 This observation enables us to translate uniqueness questions for convex expectations on the path space~\(\Omega\) to uniqueness questions for convex semigroups on the state space \(E\), allowing us to use powerful methods from viscosity and semigroup theory. We illustrate such an application below.

\begin{theorem} \label{theo: 1DD}
    Let \(\cE\) and \(\hcE\) be two Markovian homogeneous convex expectations without fixed times of discontinuity and denote by \((T_t)_{t \geq 0}\) and \((\hT_t)_{t \geq 0}\) the associated Markovian semigroups. Furthermore, assume that \(\hcE\) is continuity preserving. 
    Then, the following are equivalent:
    \begin{enumerate}
        \item[\textup{(i)}] 
        \(T_t \leq \widehat{T}_t\) on \(C_b (E; \bR)\) for all \(t > 0\). 
        \item[\textup{(ii)}] 
    \(\cE \leq \hcE\) on \(\USA_b (\Omega; \bR)\).
    \end{enumerate}
\end{theorem}

\begin{proof}
The implication (ii) \(\Longrightarrow\) (i) is trivial. We now prove the converse implication. By Theorem~\ref{theo: FDD}, it is enough to prove that
\[
\cE_0 \big( g (X_{t_1}, \dots, X_{t_n}) \big) (\omega) \leq \hcE_0 \big( g (X_{t_1}, \dots, X_{t_n}) \big) (\omega)
\]
for all \(g \in C_b (E^n; \bR)\), \(0 \leq t_1 < t_2 < \dots < t_{n - 1} < t_n\) and \(n \in \mathbb{N}\). 
We argue by induction. 
The induction base follows from~(i). Now, assume that the claim holds for \(n - 1\). In the following, for transparency, we suppress the \(\omega\) in our notation. Moreover, in case we consider the concatenation of two convex expectations, \(\omega'\) denotes the variable corresponding to the outer expectation, i.e., \(\cE (\cE (\, \cdot \,) (\omega'))\). 
Using Remark~\ref{rem: regular cond property convex}, the tower property \ref{NE6} and homogeneity of \(\cE\) and \(\hcE\), hypothesis~(i), the assumption that \(\hcE\) is continuity preserving and finally, the induction hypothesis, we obtain that 
\begin{align*}
    \cE_0 \big( \, g (X_{t_1}, \dots, X_{t_n}) \big) &= \cE_0 \big( \, \cE_{t_{n - 1}} \big( g (\omega' (t_1), \dots, \omega' (t_{n - 1}), X_{t_n}) \big) (\omega') \big) 
    \\&= \cE_0 \big( \, \cE_{0} \big( g (\omega' (t_1), \dots, \omega' (t_{n - 1}), X_{t_n - t_{n - 1}}) \big) (\theta_{t_{n - 1}} (\omega')) \big) 
    \\&= \cE_0 \big( \, T_{t_{n} - t_{n - 1}} \big( g (\omega' (t_1), \dots, \omega' (t_{n - 1}), \, \cdot \,) \big) (\omega' (t_{n - 1}))\big) 
    \\&\leq \cE_0 \big( \, \widehat{T}_{t_{n} - t_{n - 1}} \big( g (\omega' (t_1), \dots, \omega' (t_{n - 1}), \, \cdot \,) \big) (\omega' (t_{n - 1}))\big) 
    \\&\leq \hcE_0 \big( \, \widehat{T}_{t_{n} - t_{n - 1}} \big( g (\omega' (t_1), \dots, \omega' (t_{n - 1}), \, \cdot \,) \big) (\omega' (t_{n - 1}))\big) 
    \\&= \hcE_0 \big( \, \hcE_{t_{n - 1}} \big( g (\omega' (t_1), \dots, \omega' (t_{n - 1}), X_{t_n}) \big) (\omega') \big) 
    \\&= \hcE_0 \big( \, \hcE_{t_{n - 1}} \big( g (X_{t_1}, \dots, X_{t_n}) \big) \big)
    \\&= \hcE_0 \big( \, g (X_{t_1}, \dots, X_{t_{n}}) \big). 
\end{align*}
This completes the induction step. The proof is complete. 
\end{proof}

In the remainder of this section, our program is to illustrate an application of Theorem~\ref{theo: 1DD} in the context of stochastic optimal control. 
To pitch the idea, Theorem~\ref{theo: 1DD} explains that the considered class of convex expectations are characterized by semigroups on the state space. Under suitable structural assumptions such semigroups can be uniquely characterized by their \emph{pointwise generators}. On \(\bR^d\) such a unique characterization can be deduced from viscosity theory. 

From now on, let \(E = \bR^d\) and \(\Omega = C (\bR; \bR^d)\) for a fixed dimension \(d \in \mathbb{N}\).

\begin{definition}
For a convex semigroup \((T_t)_{t \geq 0}\) on \(C_b (\bR^d; \bR)\), we call \((A, D(A))\), with
\[
A (f) (x) := \lim_{t \searrow 0} \, \frac{T_t (f) (x) - T_0 (f) (x)}{t}
\]
and
\[
D (A) := \Big\{ f \in C_b (\bR^d; \bR) \colon A (f) (x) \text{ exists for all } x \in \bR^d \Big\},
\]
the {\em pointwise generator of \((T_t)_{t \geq 0}\)}.
\end{definition}

At this point, it is worth mentioning that the Levy--Khinchine triplet from Definition~\ref{def: LK triplet} has a one-to-one relation to the Markovian semigroup of the sublinear expectation associated to a nonlinear L\'evy process, see also Example~\ref{ex: nonlinear Levy Markovian} above.

Let \(\Lambda\) be a compact metrizable space and take two Borel functions
\begin{align*}
    &\mu \colon \bR^d \times \Lambda \to \bR^d,
    \\
    &\sigma \colon \bR^d \times \Lambda \to \bR^{d \times r},
\end{align*}
where \(r \in \mathbb{N}\) denotes the dimension of the noise.

\begin{condition} \label{cond: lipschitz}
    The functions \(\mu\) and \(\sigma\) are bounded and continuous.
    Furthermore, there exists a Lipschitz constant \(L > 0\) such that 
    \[
    \| \mu (x, \lambda) - \mu (y, \lambda)\| + \|\sigma (x, \lambda) - \sigma (y, \lambda)\| \leq L \, \|x - y\|
    \]
    for all \(x, y \in \bR^d\) and \(\lambda \in \Lambda\). Finally, the set
    \[
   \Theta (x) := \Big\{ (\mu (x, \lambda), \sigma \sigma^* (x, \lambda)) \colon \lambda \in \Lambda \Big\}
    \]
    is convex for all \(x \in \bR^d\).
\end{condition}

For \(f \in C^\infty_b (\bR^d; \bR)\) and \(x \in \bR^d\), we define the Hamiltonian 
\[
H (f) (x) := \sup \Big\{ \langle \mu (x, \lambda), \nabla f (x) \rangle + \tfrac{1}{2} \on{tr} \big[ \sigma (x, \lambda) \sigma^* (x, \lambda)\, \nabla^2 f (x) \big] \colon \lambda \in \Lambda \Big\}.
\]
Associated with the Hamiltonian, we introduce a sublinear semigroup as follows.
Similar to Section~\ref{sec: ex Levy}, let \(\fPas\) be the set of all probability measures on \(C (\bR_+; \bR^d)\) such that the one-sided coordinate process \((X_t)_{z \geq 0}\) is a continuous \(P\)-semimartingales with absolutely continuous characteristics whose Lebesgue densities are denoted by \((b^P, c^P)\).
Furthermore, as in Section~\ref{sec: ex Levy}, we set
\[
\mathfrak{P} (\Omega) \ni P \mapsto \bar{P}_t := P \circ (X_{s + t})_{s \geq 0}^{-1} \in \mathfrak{P} (C (\bR_+; \bR)), \quad t \in \bR.
\]
For \((t, \omega) \in \bR \times \Omega\), we define 
\[
\cU (t, \omega) := \Big\{ P \in \mathfrak{P} (\Omega) \colon P ( [\omega]_t) = 1, \ \bar{P}_t \in \fPas \text{ with  \((\bar{P}_t \otimes \llambda)\)-a.e. \((b^{\bar{P}_t}, c^{\bar{P}_t}) \in \Theta (X_\cdot)\)} \, \Big\}.
\]
This uncertainty set is related to the so-called random \(G\)-expectation, see Discussion~\ref{diss: NVH example} and \cite[Section~4]{NVH}. Under the convexity assumption in Condition~\ref{cond: lipschitz}, it was proved in \cite{CN_23_arxiv_jumps} that \(\cU\) has a control representation in terms of the relaxed control framework from Section~\ref{sec: example relaxed control}. We formally state this connection in the next lemma. It is worth mentioning that it relates our observations from Section~\ref{sec: example relaxed control} to the language of random \(G\)-expectations. We believe it is interesting to illustrate both perspectives. In part (v) of Discussion~\ref{diss: G vs relax} below, we also comment on this connection in more detail.

\begin{lemma} \label{lem: control representation}
    Assume that Condition~\ref{cond: lipschitz} holds and let \(\cK\) be defined as in \eqref{eq: KR} with \(\mu (t, \omega, \lambda) \equiv \mu (\omega (t), \lambda)\) and \(\sigma (t, \omega, \lambda) \equiv \sigma (\omega (t), \lambda)\). Then, 
    \begin{align*}
    \cU (t, \omega) = \Big\{ P \circ X^{-1} \colon P \in \cK (t, \omega) \Big\}. 
\end{align*}
\end{lemma}
\begin{proof}
Defining
\begin{align*}
\mathcal{R} (x) := \Big\{ P \in \mathfrak{P} (C (\bR_+; \bR^d)) \colon &P (X_0 = x) = 1, \, P \in \fPas \text{with } (P \otimes \llambda)\text{-a.e. } (b^P, c^P) \in \Theta \Big\}, 
\end{align*}
we observe that
\begin{align} \label{eq: rep U}
    \cU (t, \omega) = \Big\{ P \in \mathfrak{P} (\Omega) \colon P ([\omega]_t) = 1, \, \bar{P}_t \in \mathcal{R} (\omega (t))  \Big\}.
\end{align}
    We deduce from \cite[Theorem~5.1]{CN_23_arxiv_jumps} and \eqref{eq: K shift identity projection} that 
\begin{align} \label{eq: R = KR 1}
    \mathcal{R} (\omega (t)) = \Big\{ \bar{P}_0 \colon P \in \cK (0, \theta_t (\omega)) \Big\} = \Big\{ \bar{P}_t \colon P \in \cK (t, \omega) \Big\}.
\end{align}
The identities \eqref{eq: rep U} and \eqref{eq: R = KR 1} yield the claim.
\end{proof}

\begin{theorem} \label{theo: stochastic representation}
Assume that Condition~\ref{cond: lipschitz} holds. Let \(\cE\) be a Markovian homogeneous sublinear expectation that is continuity preserving. Denote the associated Markovian semigroup by \((T_t)_{t \geq 0}\) and its pointwise generator by \((A, D(A))\).
Assume that \((t, x) \mapsto T_t (f) (x)\) is continuous for all \(f \in C_b (\bR^d; \bR)\), \(C_b^\infty (\bR^d; \bR) \subset D (A)\) and 
    \(
    A = H \text{ on } C_b^\infty (\bR^d; \bR).
    \)
Then, 
\begin{align} \label{eq: stochastic representation equality}
    \cE_t (\varphi) (\omega) = \sup_{P \in \cU (t, \omega)} E^P \big[ \varphi \big]
\end{align}
for all \((t, \omega) \in \bR \times \Omega\) and \(\varphi \in \USA_b (\Omega; \bR)\).
\end{theorem} 
\begin{proof}
First, it follows from the representation in Lemma~\ref{lem: control representation}, Theorem~\ref{theo: CE control setting} and Theorem~\ref{theo: CE relax homogeneous} that
\[
\hat{\cE}_t (\varphi) (\omega) := \sup_{P \in \cU (t, \omega)} E^P \big[ \varphi \big]
\]
is a homogeneous sublinear expectation (that trivially has no fixed times of discontinuity as \(\Omega = C (\bR; \bR^d)\)). 

\smallskip
Second, we argue that \(\hat{\cE}\) is Markovian. 
On the one hand, for \(g \in C_b (\bR^d; \bR) \) and \(t > 0\), 
\[
\hat{\cE}_0 (g (X_t)) (\omega) = \sup_{P \in \cU (0, \omega)} E^{\bar{P}_0} \big[ g (X_{t})\big] \leq \sup_{P \in \mathcal{R} (\omega (0))} E^P \big[ g (X_{t})\big].
\]
On the other hand, for \(P \in \mathcal{R} (\omega (0))\), we have \(P \circ (\omega \, \otimes_0 \, X)^{-1} \in \cU (0, \omega)\) and consequently, 
\[
\sup_{P \in \mathcal{R} (\omega (0))} E^P \big[ g (X_{t})\big] = \sup_{P \in \mathcal{R} (\omega (0))} E^{P\, \circ\, (\omega\, \otimes_0\, X)^{-1}} \big[ g (X_{t})\big] \leq \hat{\cE}_0 (g (X_t)) (\omega). 
\]
In summary, 
\[
\hat{\cE}_0 (g (X_t)) (\omega) = \sup_{P \in \mathcal{R} (\omega (0))} E^P \big[ g (X_t) \big]. 
\]
This shows that \(\hat{\cE}\) is Markovian. In particular, we identified 
\[
\hat{T}_t (g) (x) = \sup_{P \in \mathcal{R} (x)} E^P \big[ g (X_t) \big]
\]
as the associated Markovian semigroup of \(\hat{\cE}\). 

Third, we verify that \(\hat{\cE}\) is continuity preserving. It follows form \cite[Theorem~4.4 and Theorem~4.7]{CN_23_EJP} that the correspondence \(x \mapsto \mathcal{R}(x)\) is continuous (in the set-valued sense) with nonempty compact values. By virtue of the representation of \(\hat{T}\) from the previous step, for every \(n \in \mathbb{N}\), \(g \in C_b (\bR^{d \, (n + 1)}; \bR)\) and \(t > 0\), Berge's maximum theorem (cf. \cite[Theorem~17.31]{charalambos2013infinite}) yields that
\[
(x_1, \dots, x_{n}) \mapsto \hat{T}_t (g (x_1, \dots, x_{n}, \, \cdot \,)) (x_{n}) = \sup_{P \in \mathcal{R} (x_{n})} E^P \big[ g (x_1, \dots, x_{n}, X_t)\big]
\]
is continuous. This shows that \(\hat{\cE}\) is continuity preserving.

Finally, we are in the position to complete the proof. 
By Theorem~\ref{theo: 1DD}, \(\cE = \hat{\cE}\) on \(\USA_b (\Omega; \bR)\) follows from \(T = \hat{T}\) on \(C_b (\bR^d; \bR)\). 
We use a viscosity method to establish the latter.
Take \(f \in C_b (\bR^d; \bR)\).
By \cite[Theorem~3.7]{CN_23_arxiv_viscosity}, \((t, x) \mapsto \hat{T}_t (f) (x)\) is the unique bounded viscosity solutions to the Hamilton--Jacobi--Bellman PDE
    \begin{align} \label{eq: HJB}
    \frac{du}{dt} - H (u) = 0, \quad u (0, \cdot \,) = f.
    \end{align}
Thanks to our hypothesis, \cite[Proposition~4.10]{hol16} yields that \((t, x) \mapsto T_t (f) (x)\) is also a bounded viscosity solution to \eqref{eq: HJB}. Hence, we obtain that \(T_t (f) (x) = \hat{T}_t (f) (x)\) for all \((t, x) \in \bR_+ \times \bR^d\).  
\end{proof}

\begin{discussion} \label{diss: G vs relax}
\begin{enumerate}

  \item[\textup{(a)}] Theorem~\ref{theo: represntation convex} relates the infinitesimal description of a sublinear semigroup on the state space to a sublinear expectation on the path space. In the realm of classical linear Markov process theory, this corresponds to the connection of linear semigroups and Feller processes.

  \item[\textup{(b)}] Building on the dynamic programming principle that holds for the above control setting (cf.~\cite{CN_23_EJP}), Theorem~\ref{theo: represntation convex} extends the time consistency property \ref{NE6} from deterministic times to stopping times, i.e.,  
\[ 
\cE_\tau (\cE_\rho) = \cE_\tau 
\]
for arbitrary finite stopping times \(\tau \leq \rho\). In the context of classical linear Markov process theory, this corresponds to the fact that Feller processes have the strong Markov property.

  \item[\textup{(c)}]  
The Lipschitz assumption from Condition~\ref{cond: lipschitz} cannot be dropped. Indeed, the Stroock--Varadhan example (cf. \cite[Exercise~12.4.2]{SV}) for a martingale problem without Feller selection shows that the associated Markovian semigroup of the r.h.s. in \eqref{eq: stochastic representation equality} lacks the necessary continuity properties.
Let us detail the example. We take \(\Lambda = \{0\}, d = 1, \sigma \equiv 0\) and \(b (x, \lambda) \equiv b^* (x)\) such that \(b^*\in C_b (\bR; \bR)\) with \(b^* (x) = \on{sgn} (x) \sqrt{|x|}\) for \(x \in [-1,1]\) and \(b^*\) is continuously differentiable outside \((-1, 1)\). Then, \(\cU \in \cR\) due to \cite[Theorem~4.4, Lemmata~6.6, 6.17, Corollary~6.12]{CN_23_EJP}, and 
\[
\hat{\cE}_t (\phi) (\omega) := \sup_{P \in \cU (t, \omega)} E^P \big[ \varphi \big]
\]
defines a convex expectation by Theorem~\ref{theo: represntation convex}. Furthermore, it follows as in the proof of Theorem~\ref{theo: stochastic representation} that \(\cE\) is Markovian and homogeneous. 
As explained in \cite[Exercise~12.4.2]{SV}, for \(\omega \in \Omega\) with \(\omega (0) \not = 0\), \(\cU (0, \omega) = \{P_{\omega (0)}\}\) such that \(\bR \setminus \{0\} \ni x \mapsto P_x\) has different limits at zero when going through positive or negative values. In particular, it is clear that each \(P_x\) is a Dirac measure. As a consequence, the associated Markovian semigroup \((\hat{T}_t)_{t \geq 0}\) of \(\hat{\cE}\) lacks continuity in the initial value, i.e., there exists an \(t > 0\) and a function \(f \in C_b (\bR; \bR)\) such that \(x \mapsto T_t (f) (x)\) is discontinuous. This shows that the claimed identity from Theorem~\ref{theo: stochastic representation} cannot hold. 
\item[\textup{(d)}]
 Recalling Discussion~\ref{diss: NVH example}, the identity \eqref{eq: stochastic representation equality} might fail without the convexity assumption from Condition~\ref{cond: lipschitz}. As mentioned above, the relaxed control framework convexifies random \(G\)-expectations. In this regard it is natural to ask whether it is possible to drop the convexity part from Condition~\ref{cond: lipschitz} in case \eqref{eq: stochastic representation equality} is replaced by 
\[
\cE_t (\varphi) = \sup_{P \in \cK (t, \omega)} E^P \big[ \varphi (X) \big].
\]
The answer to this question is affirmative. Indeed, adapting arguments from \cite{nicole1987compactification, K90} on the convergence properties of relaxed control rules, one can prove this identity. As the full argument is quite technical, we will not detail it here.

\end{enumerate}
    
\end{discussion}

\section{A Laplace principle for controlled diffusion processes} \label{sec: application} 
In this section, we illustrate an application of our results to a small-noise asymptotic analysis of controlled diffusions by establishing a Laplace principle on path space. A key step is a stochastic control representation of the entropic risk measure. To obtain this representation, we use our comparison principle from Theorem~\ref{theo: 1DD} to reduce the problem to a standard Markovian control setting, in which the value functions are characterized by HJB equations and can be treated by classical viscosity arguments. 
To set up our framework, take \(d, r \in \mathbb{N}\), set \(\Omega := C (\bR; \bR^d)\), let \(\Lambda\) be a compact metrizable space and consider the coefficients
\[\mu^\varepsilon \colon \bR^d \times \Lambda \to \bR^d, \qquad
 \sigma^\varepsilon \colon \bR^d \times \Lambda \to \bR^{d \times r}, 
\]
which are indexed over \(\varepsilon > 0\). Throughout this section, for every \(\varepsilon > 0\), we assume:
\begin{enumerate}
    \item[\((i)^\varepsilon\)] The coefficients \(\mu^\varepsilon\) and \(\sigma^\varepsilon\) are continuous. 
    \item[\((ii)^\varepsilon\)] There exists a constant \(C = C (\varepsilon) > 0\) such that, for all \(\lambda \in \Lambda\) and \(x, y \in \bR^d\),  
    \begin{equation} \label{eq: Lipschitz const}
    \begin{split}
\| \mu^\varepsilon (x, \lambda) - \mu^\varepsilon (y, \lambda) \| + \| \sigma^\varepsilon (x, \lambda) - \sigma^\varepsilon (y, \lambda) \| &\leq C \| x - y \|,
\\
\| \mu^\varepsilon (x, \lambda) \| + \| \sigma^\varepsilon (x, \lambda) \| &\leq C.
\end{split}
\end{equation} 
\end{enumerate}
For \((t, \omega) \in \bR \times \Omega\), let \(\cK^\varepsilon (t, \omega)\) be the set of all relaxed control rules from \eqref{eq: KR} as defined in Section~\ref{sec: example relaxed control}, but associated to the Markovian coefficients \((t, \omega, \lambda) \mapsto \mu^\varepsilon (\omega (t), \lambda)\) and \((t, \omega, \lambda) \mapsto \sigma^\varepsilon (\omega (t), \lambda)\). Furthermore, let 
\[
{\overline{\cE}}^\varepsilon (\varphi) (t, \omega) := \sup_{P \in \cK^\varepsilon (t, \omega)} E^P \big [ \varphi (X) \big], \quad (t, \omega) \in \bR \times \Omega, \ \varphi \in \USA_b (\Omega; \bR), 
\]
be the associated sublinear expectation; cf.~Theorem~\ref{theo: CE control setting}.
In the remainder of this section, we investigate the convergence of the entropic risk measures 
\[
\cE^\varepsilon (\varphi) (t, \omega) := \varepsilon \log \Big\{ {\overline{\cE}}^\varepsilon \Big( \exp \Big\{ \frac{\varphi (X)}{\varepsilon} \Big\} \Big) (t, \omega) \Big\}, \quad \varphi \in \USA_b (\Omega; \bR), \ (t, \omega) \in \bR \times \Omega, 
\]
as \(\varepsilon \to 0\). To propose a limiting object, let 
\[\mu^0 \colon \bR^d \times \Lambda \to \bR^d, \qquad
\sigma^0 \colon \bR^d \times \Lambda \to \bR^{d \times r},
\]
be two coefficients that satisfy \((i)^0\) and \((ii)^0\) above, i.e., they are bounded, continuous, and Lipschitz continuous in the first variable (uniformly in the second variable). 
The following Laplace principle is the main result in this section.
\begin{theorem} \label{theo: laplace principle} 
    Assume that 
    \[
    \| \mu^\varepsilon - \mu^0 \|_\infty + \| \varepsilon^{- 1/ 2} \sigma^\varepsilon - \sigma^0 \|_\infty \to 0, \quad \text{as }\varepsilon \to 0.
    \]
    Then, for all \((t, \omega) \in \bR \times \Omega\) and \(\varphi \in C_b (\Omega; \bR)\) such that \(\varphi = \varphi (X_{\cdot \wedge T})\) for some \(T > t\), 
    \begin{equation} \label{eq: det equ}
\begin{split}
\cE^\varepsilon (\varphi) (t, \omega) \to &\sup \Big\{ \varphi (X^{t, \omega}) - \frac{1}{2} \int_{- \infty}^\infty \|a_s\|^2 \, ds \colon \\ &\hspace{2cm}a \colon \bR \to \bR^r, \lambda \colon \bR \to \Lambda \text{ Borel measurable}, \phantom \int 
\\&\hspace{2cm}d X^{t, \omega}_s = (\mu^0(X^{t, \omega}_s, \lambda_s) + \sigma^0 (X^{t, \omega}_s, \lambda_s) \, a_s ) \, ds, \, s > t, \phantom \int 
\\&\hspace{2cm} X^{t, \omega}_s = \omega (s), \, s \leq t \, \Big\}, 
\end{split}
\end{equation} 
as \(\varepsilon \to 0\).
\end{theorem}

\begin{discussion}
Our proof for Theorem~\ref{theo: laplace principle} proceeds in two steps. First, we establish a stochastic control representation of the entropic risk measure \(\cE^\varepsilon\) and its proposed limit. This representation is the key ingredient: it allows us, in the second step, to apply stochastic calculus techniques to prove the asserted convergence and provides the foundation for a Laplace principle that yields convergence for continuous functions on the path space. The proof for the stochastic representation crucially relies on our comparison Theorem~\ref{theo: 1DD}, which reduces the identity on path space to a simpler Markovian setting that can be analyzed by classical viscosity methods.

Our methodology is close in spirit to analytic viscosity approaches to large deviations (cf., e.g., \cite[Chapter~VI]{FleSon_06}) and to more recent BSDE-based methods (cf.~\cite{MRTZ_16}). As in the viscosity approach, we rely on comparison results; however, instead of directly using these results to deduce a Laplace principle for specific functionals with running and terminal parts, we combine them with our comparison Theorem~\ref{theo: 1DD} to obtain a stochastic representation on path space. The BSDE approach of~\cite{MRTZ_16} likewise rests on a stochastic representation, but it is derived via BSDE techniques.

Our approach extends naturally to more general settings, such as non-local jump-diffusion frameworks. It also worth to mention that the uniform convergence assumptions in Theorem~\ref{theo: laplace principle} may be relaxed considerably. At this stage, our goal is to illustrate the method rather than to optimize the convergence conditions.
\end{discussion}

\begin{remark} 
(i) 
The assumptions of Theorem~\ref{theo: laplace principle} are trivially satisfied for \(\mu^\varepsilon \equiv \mu^0\) and \(\sigma^\varepsilon \equiv \sqrt{\varepsilon} \sigma^0\). Moreover, taking the action space \(\Lambda\) as singleton shows that classical uncontrolled diffusion frameworks are covered by our setting.
These observations relate Theorem~\ref{theo: laplace principle} to classical Freidlin--Wentzell large deviation results as given in the monographs~\cite{DZ_98,dupuisellis,FleSon_06}. Our analysis also shares common ground with the small-noise analysis of risk-sensitive control problems from \cite{J_92} and \cite[Section~XI.7]{FleSon_06}, although our limiting object is a control problem rather than a differential game.
As explained in \cite{CN_23_arxiv_jumps}, many stochastic models under uncertainty can be translated into the relaxed control framework, including the prominent \(G\)-Brownian motion framework. This relates Theorem~\ref{theo: laplace principle} to the large deviation principle for \(G\)-Brownian motion that has been established in \cite{GJ_10}.

(ii) It is worth to mention that in classical linear frameworks Varadhan's lemma and Bryc's inverse (cf. \cite[Theorems~4.3.1, 4.4.2]{DZ_98}) relate the Laplace principle to general large deviation principles. For related results in a sublinar expectation setting, see the Appendix of \cite{GJ_10}.
\end{remark} 

The remainder of this section is dedicated to the proof of Theorem~\ref{theo: laplace principle}. We first derive the relaxed control representation for \(\cE^\varepsilon\) and then, based on this representation, analyze its convergence. 

Let \(\ovM\) be the set of all Radon measures \(\overline{M}\) on \((\bR \times \bR^r \times \Lambda, \mathcal{B} (\bR) \otimes \mathcal{B} (\bR^r) \otimes \mathcal{B} (\Lambda))\) such that
\[
\forall \, T > 0 \colon\ \ \int_{- T}^T \int_{\bR^r \times \Lambda} \| x \|^2 \, \overline{M} (ds, dx, d \lambda) < \infty, \quad \overline{M} (ds \times \bR^r\times \Lambda) = ds, 
\]
and endow \(\ovM\) with the local in time weak topology, which turns it into a Polish space. Now, similar to Section~\ref{sec: example relaxed control}, define \(\ocK^\varepsilon (t, \omega)\) to be the set of all relaxed control rules associated to the controlled equation
\[
d Y_s = \int_{\bR^r \times \Lambda} \Big( \mu^\varepsilon (Y_s, \lambda) + \frac{\sigma^\varepsilon (Y_s, \lambda)}{\sqrt{\varepsilon}} \, x \Big) \, \overline{m}_s(dx, d \lambda) \, ds + \sqrt{ \int_{\bR^r \times \Lambda} \sigma^\varepsilon (\sigma^\varepsilon)^* (Y_s, \lambda) \, \overline{m}_s (dx, d \lambda) } \, d W_s, 
\]
where the relaxed control \(\overline{M} (ds, dx, d \lambda) = \overline{m}_s (dx, d \lambda) \, ds\) is taken from \(\ovM\) and \(W = (W_s)_{s \geq 0}\) is a \(d\)-dimensional standard Brownian motion.
For \(\varphi \in \USA_b (\Omega; \bR)\), we define 
\[
\widehat{\cE}^\varepsilon (\varphi) (t, \omega) := \sup_{P \in \ocK^\varepsilon (t, \omega)} E^P \Big[ \varphi (X) - \frac{1}{2} \int_{- \infty}^\infty \int_{\bR^r \times \Lambda} \| x \|^2 \, \overline{M} (ds, dx, d \lambda) \Big ], \quad (t, \omega) \in \bR \times \Omega.
\]
\begin{remark}
    For \(\varphi \in \USA_b (\Omega; \bR)\) with \(\varphi = \varphi (X_{\cdot \wedge T})\) and any \((t, \omega) \in (- \infty, T] \times \Omega\), notice that
    \[
    \widehat{\cE}^\varepsilon (\varphi) (t, \omega) = \sup_{P \in \ocK^\varepsilon (t, \omega)} E^P \Big[ \varphi (X) - \frac{1}{2} \int_{t}^T\int_{\bR^r \times \Lambda} \| x \|^2 \, \overline{M} (ds, dx, d \lambda) \Big ],
    \]
    since we may take the control \(\delta_{(0, \lambda_0)} (dx, d \lambda) \, ds\) outside the set \([t, T] \times \bR^r \times \Lambda\). 
\end{remark}
One may adapt the proof of Theorem~\ref{theo: CE control setting} to show that \(\widehat{\cE}^\varepsilon\) is a convex expectation in the sense of Definition~\ref{def: sublinear expectation}. The only necessary changes in the proof are related to the lack of compactness of the set of control rules \(\ocK^\varepsilon (t, \omega)\) and the unboundedness of the drift coefficient. However, by Aldous' tightness criterion (see~\cite[Theorem~VI.4.5]{JS}; and use \eqref{eq: moment bound} and \eqref{eq: time bound} below) and the discussion in \cite[Section~3.10]{HausLep90} or \cite{B_01},\footnote{By \cite[Theorem~2.5]{B_01} and a projective limit argument (see \cite[Proposition~8, p.~89]{B_98}), for every \(r > 0\), the set \(\{ \overline{m} \in \overline{\mathbb{M}} \colon \int_{- \infty}^\infty \int_{\bR^r \times \Lambda} \| x \|^2 \, \overline{m} (ds, dx, d\lambda) \leq r \}\) is compact in \(\overline{\mathbb{M}}\).} it follows that for any \(R \in \mathbb{R}_+\), the set
\begin{align} \label{eq: restricted moments}
\ocK^\varepsilon_R (t, \omega) := \Big\{ P \in \cK^\varepsilon (t, \omega) \colon E^P \Big[ \int_{- \infty}^\infty \int_{\bR^r \times \Lambda} \|x \|^2 \, \overline{M} (ds, dx, d \lambda) \Big] \leq R \Big\}
\end{align} 
is relatively compact. 
With this observation at hand, the proof can essentially be repeated. Moreover, arguing as in Lemma~\ref{theo: CE relax homogeneous}, one proves that \(\widehat{\cE}^\varepsilon\) is homogeneous. 

The next result gives the announced stochastic representation, deduced from Theorem~\ref{theo: 1DD}.
\begin{lemma} \label{lem: iden sensitivity}
\(\cE^\varepsilon (\varphi) = \widehat{\cE}^\varepsilon (\varphi)\) for all \(\varphi \in \USA_b (\Omega; \bR)\) and \(\varepsilon > 0\).
\end{lemma} 
\begin{proof}
The proof is split into two steps. First, we show that \(\cE^\varepsilon\) and \(\widehat{\cE}^\varepsilon\) satisfy the prerequisites of Theorem~\ref{theo: 1DD} and second, we apply this result to deduce the claim. 

{\em Step 1}: 	We discussed already that \(\widehat{\cE}^\varepsilon\) is a homogeneous convex expectation. The same is true for~\(\cE^\varepsilon\), which follows from the corresponding properties of \(\overline{\cE}^\varepsilon\); see Theorems~\ref{theo: CE control setting} and \ref{theo: CE relax homogeneous}. 
The Markovianity of \(\overline{\cE}^\varepsilon\) and \(\widehat{\cE}^\varepsilon\) follows similar to the proof of Theorem~\ref{lem: control representation}. We now explain that \(\widehat{\cE}^\varepsilon\) is continuity preserving;  the argument for \(\overline{\cE}^\varepsilon\) is similar and omitted. At this point, we mention that \(\cE^\varepsilon\) is then also Markovian and continuity preserving, as these properties transfer from \(\overline{\cE}^\varepsilon\).
The associated Markovian semigroup \((\widehat{T}^\varepsilon_t)_{t \geq 0}\) of the convex expectation \(\widehat{\cE}^\varepsilon\) is given by 
\begin{align} \label{eq: M sg application}
\widehat{T}_t^\varepsilon (g) (x) = \sup_{P \in \ocR^\varepsilon (x)} E^P \Big[ g (X_t) - \frac{1}{2} \int_{- \infty}^\infty \int_{\bR^r \times \Lambda} \| z \|^2 \, \overline{M} (ds, dz, d \lambda) \Big], \quad g \in C_b (\bR^d; \bR), 
\end{align} 
where 
\(
\ocR^\varepsilon (x) := \ocK^\varepsilon  (0, \mathbf{x})
\)
with \(\mathbf{x} (s) \equiv x\) for all \(s \in \bR\). 
Take \(g \in C_b (\bR^{nd}; \bR)\), \(T > 0\), \(x, y \in \bR^d\) and \(v, w \in \bR^{(n - 1)d}\). 
We may restrict the optimization in \eqref{eq: M sg application} to the set \(\ocR_R^\varepsilon (x) := \ocK^\varepsilon_R (0, \mathbf{x})\) as defined in \eqref{eq: restricted moments} with \(R = 2 \| g \|_\infty + 1\). For a moment, fix a measure \(P \in \ocR^\varepsilon_R (x)\). 
By the definition of relaxed control rules and the representation result \cite[Theorem~IV-2]{EM_90}, possibly on an extension of the probability space with the probability measure~\(P\), there exist orthogonal martingale measures \(N = (N^k)_{k = 1}^r\) with intensity \(\overline{m}_s (dz, d \lambda) \, ds\) such that \(P\)-a.s. 
\begin{equation} \label{eq: dynamics martingale measure}
\begin{split}
d X_s &=  \int_{\bR^r \times \Lambda} \Big( \mu^\varepsilon (X_s, \lambda) + \frac{\sigma^\varepsilon (X_s,  \lambda)}{\sqrt{\varepsilon}} \, z \Big) \, \overline{m}_s (dz, d \lambda) \, ds + \int_{\bR^r \times \Lambda} \sigma^\varepsilon (X_s, \lambda) \, N (dz, d \lambda, ds), \\ X_0 &= x.
\end{split}
\end{equation} 
Adapting classical methods for SDEs with Lipschitz coefficients (see p. 100 in \cite{EM_90}), on the same filtered probability space there exists an \(\bR^d\)-valued continuous process \(Y = (Y_t)_{t \geq 0}\) such that 
\begin{align*}
d Y_s &=  \int_{\bR^r \times \Lambda} \Big( \mu^\varepsilon (Y_s, \lambda) + \frac{\sigma^\varepsilon (Y_s, \lambda)}{\sqrt{\varepsilon}} \, z \Big) \, \overline{m}_s (dz, d \lambda) \, ds + \int_{\bR^r \times \Lambda} \sigma^\varepsilon (Y_s, \lambda) \, N (dz, d \lambda, ds), 
\\ 
Y_0 &= y.
\end{align*} 
Fix a \(\delta > 0\). 
Precisely as in the proof for the M{\'e}tivier--Pellaumail stability theorem \cite[Theorem, p.~88]{MP_80}, we obtain a stopping time \(\tau = \tau (\delta) \leq T\) and a deterministic function \(\kappa \colon \bR_+ \to \bR_+\) with \(\kappa (z) \to 0\) as \(z \to 0\) such that 
\begin{enumerate} 
\item[(i)] \(P ( \tau < T) \leq \delta\); 
\item[(ii)] \(E^P [ \sup_{t \in [0, \tau]} \| X_t - Y_t \|^2 ] \leq \kappa (\|x - y\|)\).
\end{enumerate} 
The function \(\kappa\) only depends on \(\delta, T, R, \varepsilon\) and the constant \(C\) from \eqref{eq: Lipschitz const}. In particular, it is independent of the choice of \(P\). 
Using the Burkholder--Davis--Gundy inequality, for \(Z = X, Y\), we obtain that
\begin{equation} \label{eq: moment bound} 
\begin{split} 
E^P \Big[ \sup_{s \in [0, T]} \| Z_s \|^2 \Big] &\leq 4 \Big ( T^2 \| \mu^\varepsilon \|^2_\infty + \frac{T \| \sigma^\varepsilon \|^2_\infty}{\varepsilon} E^P \Big[ \int_0^T \int_{\bR^r \times \Lambda} \| z \|^2 \, \overline{m}_s (dz, d\lambda) \, ds \Big] \Big) + 8 T \| \sigma^\varepsilon \|_\infty^2 
\\&\leq 4 \Big( T^2 \| \mu^\varepsilon \|_\infty^2 + \frac{T \| \sigma^\varepsilon \|^2_\infty R}{\varepsilon} \Big) + 8 T \| \sigma^\varepsilon\|_\infty^2 =: C'. 
\end{split}
\end{equation} 
We are in the position to prove the continuity preservation property. 
Take \(\ell > 0\). Using (i) above, we obtain that 
\begin{align*}
     \widehat{T}_T^\varepsilon ( g (v, \cdot \,) ) (x) &- \widehat{T}_T^\varepsilon (g (w, \cdot \, ) ) (y) 
    \\&\leq \sup_{P \in \ocR^\varepsilon_R (x)} E^P \big[ \big| g (v, X_T) - g (w, Y_T) \big| \big] 
    \\&\leq \sup_{P \in \ocR^\varepsilon_R (x)} E^P \big[ \big| g (v, X_T) - g (w, Y_T) \big| \1_{\{ \sup_{t\in [0, T]} \| X_s \| \vee \|Y_s \| > \ell \}} \big]
    \\&\qquad + \sup_{P \in \ocR^\varepsilon_R (x)} E^P \big[ \big| g (v, X_T) - g (w, Y_T) \big| \1_{\{ \sup_{t\in [0, T]} \| X_s \| \vee \|Y_s \| \leq \ell \}} \big]
    \\&\leq \frac{4 \| g \|_\infty C'}{\ell^2} + \sup_{\| z \| \leq \ell} | g (v, z) - g (w, z) | 
    \\&\qquad + \sup_{P \in \ocR^\varepsilon_R (x)} E^P \big[ \big| g (w, X_T) - g (w, Y_T) \big| \1_{\{ \sup_{t\in [0, T]} \| X_s \| \vee \|Y_s\| \leq \ell \}} \big]
    \\&\leq \frac{4 \| g \|_\infty C'}{\ell^2} + \sup_{\| z \| \leq \ell} | g (v, z) - g (w, z) | + 2 \| g \|_\infty \delta 
    \\&\qquad + \sup_{P \in \ocR^\varepsilon_R (x)} E^P \big[ \big| g (w, X_\tau) - g (w, Y_\tau) \big| \1_{\{ \sup_{t \in [0, \tau]} \| X_s \| \vee \|Y_s\| \leq \ell \}} \big].
\end{align*}
Let \(\gamma = \gamma (w, \ell) \colon \bR_+ \to \bR_+\) be an increasing concave modulus of continuity for \(z \mapsto g (w, z)\) when restricted to the ball \(\{ z \in \bR^d \colon \| z \| \leq \ell \}\). It is well-known that such a modulus exists; see \cite[Proposition~3.15]{DN_11}.
Using Jensen's inequality twice, and (ii) above, we obtain that 
\begin{align*}
    \sup_{P \in \ocR^\varepsilon_R (x)} E^P \big[ \big| g (w, X_\tau) - g (w, Y_\tau) \big| \1_{\{ \sup_{t \in [0, \tau]} \| X_s \| \vee \|Y_s\| \leq \ell \}} \big] 
    & \leq \sup_{P \in \ocR^\varepsilon_R (x)} E^P \big[ \gamma (\| X_\tau- Y_\tau \|) \big] 
    \\&\leq \sup_{P \in \ocR^\varepsilon_R (x)} \gamma \Big( \sqrt{ E^P \big[ \| X_\tau - Y_\tau \|^2 \big] } \Big)
    \\&\leq \gamma \big( \sqrt{ \kappa ( \|x - y\| ) } \big). 
\end{align*}
All together, exchanging also the roles of \(x\) and \(y\), taking \(\ell\) large and \(\delta\) small shows that 
\[
\big| \widehat{T}_T^\varepsilon ( g (v, \cdot \,) ) (x) - \widehat{T}_T^\varepsilon (g (w, \cdot \, ) ) (y) \big| 
\]
can be made arbitrarily small when \(v \to w\) and \(x \to y\). Here, we use that \[\sup_{\| z \| \leq \ell} | g (v, z) - g (w, z) | \to 0\] as \(v \to w\) by Berge's maximum theorem. 
We conclude that \(\widehat{\cE}^\varepsilon\) is continuity preserving. 

{\em Step 2:} Thanks to Theorem~\ref{theo: 1DD}, as both \(\cE^\varepsilon\) and \(\widehat{\cE}^\varepsilon\) are Markovian homogeneous convex expectations (trivially without fixed times of discontinuity) that are further continuity preserving, the identity \(\cE^\varepsilon = \widehat{\cE}^\varepsilon\) follows once we show that \(T_t^\varepsilon = \widehat{T}_t^\varepsilon\) on \(C_b (\bR^d; \bR)\) for all \(t > 0\). Take \(f \in C_b (\bR^d; \bR)\) and set 
\[
H (z, \nabla u, \nabla^2 u) := \sup_{\lambda \in \Lambda} \Big\{ \langle \mu^\varepsilon (z, \lambda), \nabla u \rangle + \tfrac{1}{2} \on{tr} \big[ \sigma^\varepsilon (\sigma^\varepsilon)^* (z, \lambda) \nabla^2 u \big] \Big\}. 
\]
Using Step 1 and the standard estimate 
\begin{equation}\label{eq: time bound} 
\begin{split} 
    \forall \, P \in \ocR_R^\varepsilon (x), \, 0 \leq s < t \colon \, \ E^P \big[ \| &X_t - X_s \| \big] 
    \\&\leq C (\|\mu^\varepsilon\|_\infty, \varepsilon^{- 1/ 2} \| \sigma^\varepsilon\|_\infty, \| \sigma^\varepsilon\|_\infty, R) \, \big( (t - s) + \sqrt{t - s} \big),
\end{split}
\end{equation} 
it follows as in the proof of \cite[Lemma~2.36]{CN_24_SPA} that the function
\[
\overline{u}^\varepsilon (t, z) := \sup_{P \in \ocR^\varepsilon (z)} E^P \big[ \exp \{ \varepsilon^{-1} f (X_{T - t}) \} \big], \quad (t, z) \in [0, T] \times \bR^d,  
\]
 is continuous.
Now, using dynamic programming arguments shows that \(\overline{u}^\varepsilon\) is a bounded viscosity solution to 
\begin{align} \label{eq: PDE vor trafo} 
-\frac{d u}{dt} - H (z, \nabla u, \nabla^2 u) = 0, \quad u (T, \, \cdot \, ) = \exp \{ \varepsilon^{-1} f \};
\end{align} 
see the proofs of \cite[Theorem~4.13]{CN_23_EJP} or \cite[Lemma~2.38]{CN_24_SPA} for some details.
By a change of variable formula for viscosity solutions (cf., e.g., \cite[Proposition~6.7]{T_13}), \(u^\varepsilon := \varepsilon \log \overline{u}^\varepsilon\) is a bounded (because \(\overline{u}^\varepsilon\) is bounded away from zero) viscosity solution to 
\begin{align} \label{eq: PDE trafo} 
  -  \frac{d u}{dt} - H (z, \nabla u,  \varepsilon^{-1}\nabla u (\nabla u)' + \nabla^2 u ) = 0, \quad u (T, \, \cdot \,) = f.
\end{align}
A short computation shows that 
\begin{align*}
    H (z&, \nabla u, \varepsilon^{-1} \nabla u (\nabla u)'  + \nabla^2 u ) 
    \\&= \sup_{\lambda \in \Lambda} \Big\{ \langle \mu^\varepsilon (z, \lambda), \nabla u) \rangle + \frac{1}{2 \varepsilon} \| (\sigma^\varepsilon)^* (z, \lambda)\nabla u \|^2 + \frac{1}{2} \on{tr} \big[ \sigma^\varepsilon (\sigma^\varepsilon)^* (z, \lambda) \nabla^2 u \big] \Big\} 
    \\&= \sup_{\lambda \in \Lambda, \, a \in \bR^r } \Big\{ \langle \mu^\varepsilon (z, \lambda), \nabla u) \rangle + \langle \varepsilon^{- 1/ 2} \sigma^\varepsilon (z, \lambda) \, a, \nabla u \rangle - \frac{1}{2} \| a \|^2 + \frac{1}{2} \on{tr} \big[ \sigma^\varepsilon (\sigma^\varepsilon)^* (z, \lambda) \nabla^2 u \big]  \Big\}.
\end{align*}
In view of this representation, dynamic programming arguments show that 
\[
\widehat{u}^\varepsilon (t, z) := \sup_{P \in \widehat{\mathcal{R}}^\varepsilon (z)} E^P \Big[ f (X_{T - t}) - \frac{1}{2} \int_{- \infty}^\infty \int_{\bR^r \times \Lambda} \| a \|^2 \, \overline{M} (ds, da, d \lambda) \Big], \quad (t, z) \in [0, T] \times \bR^d, 
\]
is a bounded viscosity solution to \eqref{eq: PDE trafo}; see the proof of \cite[Lemma~2.38]{CN_24_SPA} for some details (and notice that its continuity follows from Step 1 and an estimate of the form \eqref{eq: time bound}). 
The change of variable rule entails that the PDEs \eqref{eq: PDE vor trafo} and \eqref{eq: PDE trafo} share uniqueness properties. Since uniqueness among bounded viscosity solutions holds for \eqref{eq: PDE vor trafo} by \cite[Theorem~4.4.5]{pham}, using the transformation \(z \mapsto \exp \{ \varepsilon^{-1} z \}\), the same is true for \eqref{eq: PDE trafo}. 
As a consequence, we obtain that 
\(
u^\varepsilon = \widehat{u}^\varepsilon, 
\)
which implies that \(T_t^\varepsilon = \widehat{T}_t^\varepsilon\) on \(C_b (\bR^d; \bR)\) for all \(t > 0\). This identity completes the proof. 
\end{proof}

We also provide a convenient representation for the limit in Theorem~\ref{theo: laplace principle}. 
Let \(\widehat{\cE}^0\) be the convex expectation, for \((t, \omega) \in \bR \times \Omega, \varphi \in \USA_b (\Omega; \bR),\)
\[
\widehat{\cE}^0 (\varphi) (t, \omega) := \sup_{P \in \cK^0 (t, \omega)} E^P \Big[ \varphi (X) - \frac{1}{2} \int_{- \infty}^\infty \int_{\bR^r \times \Lambda} \| z \|^2 \, \overline{m}_s (dz, d \lambda) \, ds \Big], 
\]
where \(\cK^0 (t, \omega)\) are the relaxed control rules associated to
\begin{align} \label{eq: limiting dynamics} 
\frac{d Y_t}{dt} = \int_{\bR^r \times \Lambda} \Big( \mu^0 (Y_t, \lambda ) + \sigma^0 (Y_t, \lambda) z \Big) \, \overline{m}_t (dz, d \lambda). 
\end{align} 
Similar to Lemma~\ref{lem: iden sensitivity}, we obtain the following:
\begin{lemma} \label{lem: sensitity limit iden}
For all \((t, \omega) \in \bR \times \Omega\) and \(\varphi \in \USA_b (\Omega; \bR)\), 
\begin{align*}
    \widehat{\cE}^0 (\varphi) (t, \omega) &= \sup \Big\{ \varphi (X^{t, \omega}) - \frac{1}{2} \int_{- \infty}^\infty \|a_s\|^2 \, ds \colon \\ &\hspace{2cm}a \colon \bR \to \bR^r, \lambda \colon \bR \to \Lambda \text{ measurable}, \phantom \int 
\\&\hspace{2cm}d X^{t, \omega}_s = (\mu^0(X^{t, \omega}_s, \lambda_s) + \sigma^0 (X^{t, \omega}_s, \lambda_s) \, a_s ) \, ds, \, s > t, \phantom \int 
\\&\hspace{2cm} X^{t, \omega}_s = \omega (s), \, s \leq t \, \Big\}.
\end{align*}
\end{lemma}
\begin{proof}
Using the uniqueness result \cite[Theorem~VII.11.1]{FleSon_06} and the fact that it suffices to identify the Markovian semigroups on \(\textit{Lip}_b (\bR^d; \bR)\),\footnote{For any \(f \in C_b (\bR^d; \bR)\), the sequence \(f_n (x) = \sup_{y \in \bR^d} \{ f (y) - n \|x - y\| \}\) approximates \(f\) from above, and it consists of bounded Lipschitz functions. Using that convex expectations are continuous from above on \(C_b (\Omega; \bR)\), it follows that two Markovian semigroups coincide on \(C_b (\bR^d; \bR)\) once they coincide on \(\textit{Lip}_b (\bR^d; \bR)\).} the proof is similar to that of Lemma~\ref{lem: iden sensitivity}. We omit the details for brevity.
\end{proof}

We are now in the position to prove our convergence Theorem~\ref{theo: laplace principle}.

\begin{proof}[Proof of Theorem~\ref{theo: laplace principle}]
Thanks to Lemmas~\ref{lem: iden sensitivity} and \ref{lem: sensitity limit iden}, as well as the homogeneity property of \(\widehat{\cE}^\varepsilon\) and \(\widehat{\cE}^0\), it suffices to prove that 
\[
\widehat{\cE}^\varepsilon (\varphi) (0, \mathbf{x}) \to \widehat{\cE}^0 (\varphi) (0, \mathbf{x}), \quad \varepsilon \to 0, 
\]
for all \(\sigma (X_s, s \in [0, T])\)-measurable \(\varphi \in C_b (\Omega; \bR)\) and all \(x \in \bR^d\).
Fix \(\delta > 0\) and \(x \in \bR^d\).

{\em Step 1:} We start with a preparatory step.
As our convergence assumptions entail \(\varepsilon\)-independent global bounds for all coefficients, we deduce from \eqref{eq: time bound} and Aldous' tightness criterion (cf.~\cite[Theorem~VI.4.5]{JS}) that the set 
    \[
    \overline{\mathcal{R}} := \bigcup_{\varepsilon' \geq 0} \big\{ Q \circ (X_t)_{t \in [0, T]}^{-1} \colon Q \in \mathcal{R}^{\varepsilon'}_R (x) \big\} 
    \]
    is tight. Hence, there exists a compact set \(K \subset C ([0, T]; \bR^d)\) such that 
    \begin{align} \label{eq: tightness bound}
    \sup_{Q \in \overline{\mathcal{R}}} Q ( K^c ) \leq \delta. 
    \end{align} 

{\em Step 2:} Next, we provide a pointwise estimate. 
Fix \(\varepsilon > 0\) and take \(P \in \overline{\mathcal{R}}^\varepsilon_R (x)\) with \(R = 2 \| \varphi \|_\infty + 1\). Under \(P\), possibly on an extended probability space, the coordinate process \((X_t)_{t \geq 0}\) has the dynamics \eqref{eq: dynamics martingale measure}. On the same probability space, let \(Y = (Y_t)_{t \geq 0}\) be an \(\bR^d\)-valued continuous process with dynamics \eqref{eq: limiting dynamics}. 
    Repeating the proof of \cite[Theorem, p.~88]{MP_80}, there exists a stopping time \(\tau = \tau (\delta) \leq T\) and a function \(\kappa \colon \bR_+^3 \to \bR_+\) with \(\gamma (z_1, z_2, z_3) \to 0\) as \(z_1, z_2, z_3 \to 0\) such that 
    \begin{enumerate}
        \item[(i)] \( P (\tau < T) \leq \delta\); 
        \item[(ii)] \(E^P [ \sup_{s \in [0, \tau]} \| X_s - Y_s\|^2 ] \leq \kappa (\|\mu^\varepsilon - \mu^0\|_\infty, \|\varepsilon^{- 1/2}\sigma^\varepsilon - \sigma^0\|_\infty, \|\sigma^\varepsilon\|_\infty)\),
    \end{enumerate}
    where the function \(\kappa\) only depends on \(\delta, T, R\), and the Lipschitz constants and global bounds of \(\mu^0\) and \(\sigma^0\). Let \(K\) be as in \eqref{eq: tightness bound}.
    Now, for an increasing concave modulus of continuity \(\gamma\) of the uniformly continuous function~\(\varphi |_K\), using Jensen's inequality and (i), (ii) above, we obtain the bound
    \begin{align*}
        E^P \big[ \big| \varphi (X) - \varphi (Y) \big| \big] &\leq E^P \Big[ \gamma \Big( \sup_{s \in [0, \tau]} \| X_s - Y_s \| \Big)  \Big] + 4 \delta \| \varphi \|_\infty
        \\&\leq \gamma \Big( \sqrt{ \kappa (\|\mu^\varepsilon - \mu^0\|_\infty, \|\varepsilon^{- 1/2}\sigma^\varepsilon - \sigma^0\|_\infty, \|\sigma^\varepsilon\|_\infty)} \Big) + 4 \delta \| \varphi \|_\infty.
    \end{align*}

    {\em Step 3:} We now enhance the pointwise estimate from Step 2 to a one-sided global estimate. 
    As \(P\) in Step 2 was arbitrary, we obtain
    \begin{align*}
     \widehat{\cE}^\varepsilon (\varphi) (0, \mathbf{x}) - \widehat{\cE}^0 (\varphi) (0, \mathbf{x})
\leq \gamma \Big( \sqrt{ \kappa (\|\mu^\varepsilon - \mu^0\|_\infty, \|\varepsilon^{- 1 / 2}\sigma^\varepsilon - \sigma^0\|_\infty, \|\sigma^\varepsilon\|_\infty)} \Big) + 4 \delta \| \varphi \|_\infty.
    \end{align*} 

    {\em Step 4:} 
Starting with a probability measure \(P \in \mathcal{R}^0_R (x)\) and repeating the Steps 2 and 3, we get the converse bound
\[
 \widehat{\cE}^0 (\varphi) (0, \mathbf{x}) - \widehat{\cE}^\varepsilon (\varphi) (0, \mathbf{x}) \leq \gamma \Big( \sqrt{ \kappa (\|b^\varepsilon - b^0\|_\infty, \|\varepsilon^{- 1/2}\sigma^\varepsilon - \sigma^0\|_\infty, \|\sigma^\varepsilon\|_\infty)} \Big) + 4 \delta \| \varphi \|_\infty, 
\]
and, together with Step 3, 
\[
\big| \widehat{\cE}^\varepsilon (\varphi) (0, \mathbf{x}) - \widehat{\cE}^0 (\varphi) (0, \mathbf{x}) \big| \leq \gamma \Big( \sqrt{ \kappa (\|b^\varepsilon - b^0\|_\infty, \|\varepsilon^{- 1 / 2}\sigma^\varepsilon - \sigma^0\|_\infty, \|\sigma^\varepsilon\|_\infty)} \Big) + 4 \delta \| \varphi \|_\infty.
\]
As \(\delta > 0\) was arbitrary, and using the convergence assumptions, the claim follows. 
\end{proof}

\section{Connection to Related Concepts}\label{sec:related concepts}

This section provides a technical discussion of our results in relation to the existing literature. 
Over the past decades, the theory has evolved from linear to  sublinear and convex expectations, from dominated to non-dominated probabilistic models and from Markovian frameworks to path-dependent settings. 
Our aim here is to clarify how the present results fit into this development and to explain which structural advances allow for a unified identification of expectations, penalty representations, and semigroups on canonical path space.

At the core of this paper lies the identification of a single path-dependent object in three equivalent formulations: 
a time-consistent convex expectation on canonical path space, a dynamic penalty or control representation and a convex semigroup acting on path functionals. 
In dominated Markovian models, such correspondences are classical: dynamic programming leads to Nisio-type semigroups, while convex duality yields penalized control representations \cite{FleSon_06, pham}. 

Within the theory of risk measures and nonlinear expectations, convex expectations extend linear expectations by replacing linearity with a penalized formulation \cite{FS2016,P2019,W1991}. 
In discrete time, time consistency is typically characterized by stability under pasting and is equivalent to a dynamic programming principle \cite{bershre, FS2016}. 
In continuous time, and especially under non-dominated priors, the main challenge lies in the construction of conditional versions that satisfy suitable measurability and concatenation properties, as developed in \cite{nutz_13, NVH}. 

Our representation results rely on convex duality: while in dominated settings duality is usually formulated in $L^p$-spaces via the Fenchel--Moreau theorem, the non-dominated path-space setting requires a pointwise formulation, for which general conditional representation results are available in \cite{bartl_20}. 
The space $\USA_b$ of bounded upper semianalytic functions plays a central role in this context, as it is stable under taking certain uncountable suprema and therefore constitutes the natural domain for the construction of nonlinear expectations \cite{NVH}.
Building on these foundations, we provide minimal conditions under which convex expectations and their representing penalty functions are in one-to-one correspondence (cf.\ Theorem~\ref{theo: represntation convex}). 
In the sublinear case, these conditions reduce to the classical stability under conditioning and pasting.

Several well-established concepts can be linked to our framework. 
From the expectation viewpoint, filtration-consistent nonlinear expectations in Brownian settings include $g$-expectations, which admit representations via BSDEs \cite{CHMP2002,EPQ1997,padoux_peng_90}, as well as Peng’s $G$-expectation, which models volatility uncertainty directly on path space and is linked to fully nonlinear PDEs through 2BSDEs,  and PPDEs \cite{CSTV2007,denis_hu_peng_11,EKTZ2014,peng_et_all_PPDE_survey,STZ2012}  in the path-dependent case.
From the penalty or control perspective, such expectations arise as value functionals of stochastic control problems over families of semimartingale laws with uncertain characteristics, where time consistency is encoded through stability under conditioning and concatenation \cite{P2019}. 
This perspective extends to jump uncertainty and nonlinear Lévy-type dynamics, leading to sublinear or convex semigroups on path space \cite{CN_23_arxiv_jumps, CN_23_arxiv_viscosity,HP_21,neufeld2017nonlinear}, and more recent work allows for genuinely non-dominated and path-dependent uncertainty sets~\cite{C_24_JMAA,CN_23_EJP,CN_24_SPA}.

A key contribution of the present work is the finite-dimensional identification result (Theorem~\ref{theo: FDD}), which establishes a structural link between the expectation-based and the control-based viewpoints. 
It shows that, under the standing assumptions, a convex expectation and its associated penalty or control representation are uniquely determined by their finite-dimensional marginals. 
For instance, this yields a precise bridge between the axiomatic specification of $G$-L\'evy dynamics via finite-dimensional distributions in the sense of Hu and Peng \cite{HP_21} and the characteristics/control-based construction of nonlinear L\'evy processes developed by Neufeld and Nutz \cite{neufeld2017nonlinear}: once the finite-dimensional marginals coincide, both approaches extend uniquely to the same path-space expectation and hence to the same penalty representation. 

Finally, we emphasize the role of homogeneity and the use of a two-sided time axis. 
On path space, time-homogeneous dynamics are most naturally expressed through invariance with respect to canonical time shifts of entire trajectories. 
Working with two-sided time renders these shifts invertible and avoids boundary effects when combined with conditioning and concatenation. 
This shift-based viewpoint provides the structural basis for the homogeneous semigroup correspondence on path space and explains how conditional nonlinear expectations give rise to convex evolutionary semigroups in the sense of \cite{DKK24}. 
In this way, homogeneity emerges as the key mechanism linking expectations, penalties and semigroups within a unified and fully path-dependent dynamic framework.

\appendix

\section{On the tower property for convex expectations}\label{app:tower property}
 The tower rule of conditional convex expectations was profoundly studied in \cite{bartl_20}. 
The purpose of this appendix is to refine an assumption in the statement of \cite[Theorem~2.11]{bartl_20} within the context of our dynamic framework.
 Let \(D_{t} \equiv D_{t, \infty}\) be as in \eqref{eq: uni cont D}, i.e., 
\[
D_t = \Big\{ g (X_{t_1}, \dots, X_{t_n}) \colon g \in UC_b (E^n; \bR), \, t_1, \dots, t_n \in (t, \infty)\cap \mathbb{Q}, \, n \in \mathbb{N} \Big\},
\]
where \(UC_b (E^n; \bR)\) is the space of real-valued bounded functions on \(E^n\) that are uniformly continuous w.r.t. a metric that induces the product topology. The result can be proved following the lines of \cite[Theorem~2.11]{bartl_20}. We skip the details for brevity.

 \begin{theorem} \label{theo: theorem 2.11 bartl}
 	Fix \(t \in \bR\), let \(\alpha \colon \mathfrak{P} (\Omega) \to [0, \infty]\) be a convex function with compact level sets such that \(\inf_P \alpha (P) = 0\), and let \(\gamma \colon \Omega \times \mathfrak{P}(\Omega) \to [0, \infty]\) be a function with the properties \ref{PF3}-\ref{PF6} and without fixed times of discontinuity. For \(\varphi \in \USA_b (\Omega; \bR)\) and \(\omega \in \Omega\), define 
 	\begin{align*}
 	\cE (\varphi) &:= \sup_{P \in \mathfrak{P}(\Omega)} \Big(E^P \big[ \varphi \big] - \alpha (P)\Big), 
 	\\
 	\cE_t(\varphi)(\omega) &:= \sup_{P \in \mathfrak{P}(\Omega)} \Big(E^P \big[ \varphi \big] - \gamma (\omega, P) \Big),
 	\\
 	\beta (P) &:= \sup_{\psi \in \USA_b (\Omega, \cF_t; \bR)} \Big( E^P \big[ \psi \big] - \cE (\psi) \Big).
 	\end{align*}
 	Assume that \(\omega \mapsto \cE_t (\varphi )(\omega)\) is \(\cF_t\)-measurable for all \(\varphi \in D_{t}\). Then, 
 	\begin{align*}
 		\cE (\, \cdot \,) \leq \cE (\cE_t (\, \cdot \, )) \quad &\Longleftrightarrow \quad \alpha (P) \geq \beta (P) + E^P \big[ \gamma (\, \cdot \,, P (\, \cdot \, | \cF_t)) \big] \ \text{ for all } P \in \mathfrak{P} (\Omega). 
 	\end{align*}
 \end{theorem}

  \section{On the relation of the properties \ref{PF7} and \ref{M3}} \label{app: pf lemma scalar mb}
The following lemma clarifies the relationship between \ref{PF7} and \ref{M3}.
    \begin{lemma} \label{lem: scalar measurability}
	Let \((A, \mathcal{A})\) be a measurable space 
    and let \(A \ni a \mapsto \cV (a) \subset \mathfrak{P}(\Omega)\) be a set-valued map with nonempty, compact and convex values. Then, \(a \mapsto \cV (a)\) is (weakly) \(\cA\)-measurable if and only if the scalar maps 
    \[ 
    a \mapsto \sup_{P \in \cV (a)} E^P \big[ \phi \big], \quad \phi \in C_b (\Omega; \bR), 
    \]
    are \(\cA\)-measurable.
 \end{lemma}
 \begin{proof}
 	Let \(\ca (\Omega)\) be the space of bounded variation Radon measures on \((\Omega, \cF)\). Endowed with the topology of convergence in distribution, this space is a Hausdorff LCS with the Lindelöf property (cf. \cite[Lemma~6.6.4, Theorem~8.9.6]{bogachev}). Considering \(\cV\) as a set-valued map with nonempty, compact and convex values in \(\ca (\Omega)\), the claimed equivalence follows from \cite[Theorem~5.1]{barbati_hess_98}.
\end{proof} 

\bibliographystyle{abbrv}
\bibliography{references}

\end{document}